\documentclass[11pt, a4paper]{amsart}
\usepackage[utf8]{inputenc}
\usepackage[english]{babel}
\usepackage{enumerate}
\usepackage{color}
\usepackage{stmaryrd}
\usepackage[text={14cm,21cm},centering]{geometry}
\usepackage{marginnote} \usepackage{framed}
\usepackage{siunitx}
\usepackage{xcolor}
\usepackage{esint,amsmath,amssymb}
\usepackage{amsrefs}
\usepackage{graphicx}
\usepackage{mathtools}
\usepackage[colorlinks=true, pdfstartview=FitV, linkcolor=blue, citecolor=blue, urlcolor=+,pagebackref=false]{hyperref}
\usepackage{microtype}

\usepackage{bm}
\usepackage{scalerel}
\usepackage{mathrsfs}
\usepackage{tikz}
\usepackage[all]{xy} \xyoption{arc} \xyoption{color}
\usepackage{epsfig}
\usepackage{amsbsy}
\usepackage[font={footnotesize}]{caption}

\usepackage{caption}
\usepackage[skip=0cm,list=true,labelfont=it]{subcaption}
\usepackage{enumitem}

\usepackage{comment}

\newcommand{\eps}{\varepsilon}

\newtheorem{proposition}{Proposition}
\newtheorem{theorem}[proposition]{Theorem}
\newtheorem{lemma}[proposition]{Lemma}

\theoremstyle{remark}
\newtheorem{remark}[proposition]{Remark}

\theoremstyle{definition}

\numberwithin{equation}{section}
\numberwithin{proposition}{section}

\newcommand{\R}{\mathbb{R}}
\newcommand{\pa}{\partial}

\title[2D one-phase Muskat problem with contact points]{
Global-in-time estimates for the 2D one-phase Muskat problem with contact points
}

\author{Edoardo Bocchi, Ángel Castro and Francisco Gancedo}
	\address[Edoardo Bocchi]{Dipartimento di Matematica, Politecnico di Milano, Piazza Leonardo da
Vinci 32, 20133 Milano, Italy}
	\email{edoardo.bocchi@polimi.it}
    \address[Ángel Castro]{ICMAT-CSIC, C/ Nicol\'as Cabrera 13-15, 28049 Madrid, Spain}
	\email{angel\_castro@icmat.es}
\address[Francisco Gancedo]{Departamento de Análisis Matemático $\&$ IMUS, Universidad de Sevilla,
C/ Tarfia s/n, 41012 Sevilla, Spain\newline
and School of Mathematics, Institute for Advanced Study, 1 Einstein Dr., Princeton, NJ 08540, USA}
	\email{fgancedo@us.es}

\begin{document}

	\begin{abstract}
        In this paper, we study the dynamics of a two-dimensional viscous fluid evolving through a porous medium or a Hele-Shaw cell, driven by gravity and surface tension. A key feature of this study is that the fluid is confined within a vessel with vertical walls and below a dry region. Consequently, the dynamics of the contact points between the vessel, the fluid and the dry region are inherently coupled with the surface evolution. A similar contact scenario was recently analyzed for more regular viscous flows, modeled by the Stokes \cite{GuoTice2018} and Navier-Stokes \cite{GuoTice2024} equations. Here, we adopt the same framework
        but use the more singular Darcy's law for modeling the flow. We prove global-in-time a priori estimates for solutions initially close to equilibrium.
Taking advantage of the Neumann problem solved by the velocity potential, the analysis is carried out in non-weighted $L^2$-based Sobolev spaces and without imposing restrictions on the contact angles.
	\end{abstract}

	\maketitle
	\tableofcontents

	\section{Introduction}

This paper deals with the evolution of contact point dynamics between a solid, a fluid, and a dry region in incompressible flows. The scenario considers fluids in porous media or Hele-Shaw cells, whose evolution equations are explained below.

A two-dimensional incompressible flow
\begin{equation}\label{incompre}
\nabla\cdot u(t,x,y)=0,\quad t\geq 0,\quad (x,y)\in\R^2,
\end{equation}
confined in a porous medium is modeled by the classical Darcy's law \cite{Darcy1856}, given by
\begin{equation}\label{Darcy}
\frac{\mu}{\nu}u(t,x,y)=-\nabla p(t,x,y)-g\rho(0,1).
\end{equation}
Above, $\mu$ is the viscosity of the fluid, $\nu$ is the permeability of the homogeneous medium and $u$ the velocity of the flow. This momentum equation includes the gradient of the fluid pressure $p$ and the effect of gravity, with $g$ being its constant and $\rho$ the density. The fluid bulk is contained within the moving domain
\begin{equation}
\Omega(t)=\left\{(x,y)\in \mathbb{R}^2 \ | \ x\in\mathcal{I}=(-1,1), \ h_w(x)<y< h(t,x) \right\},
\end{equation}
with boundary
$\partial\Omega(t) = \overline{\Gamma}(t)\cup\overline{\Gamma}_w(t)$, which is divided into two parts. The first part corresponds to the fluid's moving boundary, given by \begin{equation}			\Gamma(t)=\left\{(x, h(t,x))\ |\ x\in \mathcal{I} \right\},
\end{equation}
while the second part corresponds to the rigid boundary, a wall, given by
\begin{equation}		\Gamma_w(t)=\left\{(\pm 1, y) \ | \  h_w(\pm 1)\leq  y \leq h(t,\pm1) \right\}\cup\left\{(x, h_w(x)) \ |  \ x\in \mathcal{I} \right\}, 
\end{equation}
The stated configuration
is for a fluid filtered inside a vessel with a smooth boundary. The contact points where the fluid, the vessel, and the dry region meet occur along the vertical lateral walls, as shown in Figure \ref{fig-vessel}.
\begin{figure}
    \centering
 \includegraphics[scale=1]{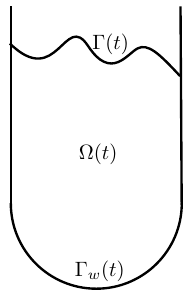}
    \caption{Configuration of the vessel.}
    \label{fig-vessel}
\end{figure}

We assume that the lower part of the vessel boundary is a curve that smoothly connects to the vertical walls, that is,  $h_w \in C^\infty(\mathcal{I})$ and $\lim_{x\rightarrow \pm 1} h'_w(x) = \pm \infty$. Within this geometrical configuration, the fluid domain exhibits corners only at the contact points
        \begin{equation*}
            \overline{\Gamma}(t)\cap \overline{\Gamma}_w(t)= \left\{(-1,h(t,-1)),(1,h(t,1 ))\right\}.
        \end{equation*}The regime investigated here is stable, as the fluid lies below a dry region or another fluid with negligible viscosity.

In this scenario, the main interest is the  evolution of the contact points between the fluid, the vessel and the dry region. Because of that, it is of crucial importance the effect of capillarity in the model. For the interaction among the fluid and the dry region, Laplace-Young condition is given as follows
\begin{equation}\label{L-Y}
p=-\sigma\kappa_h\quad\mbox{on}\quad\Gamma(t).
\end{equation}
Above, $\sigma>0$ is the surface tension coefficient and $\kappa_h$ the moving surface curvature given by
\begin{equation}
\label{curvature}
\kappa_h(t,x)=\frac{h''(t,x)}{(1+(h'(t,x))^2)^{3/2}}=\Big(\frac{h'(t,x)}{\sqrt{1+(h'(t,x))^2}}\Big)',
\end{equation}
with prime spatial derivative. The contact points' evolution are also affected by Laplace-Young condition as follows
\begin{equation}
\label{F-evo-contact}
\partial_t h (t,\pm1)= F\Big( \frac{\llbracket\gamma\rrbracket }{\sigma}\mp  \frac{h'}{\sqrt{1 + (h')^2}}\Big)(t,\pm 1),
\end{equation}
where $F$ is a  given injective function and $\llbracket \gamma\rrbracket$ is a physical constant due to the three-phase contact points. See Section 1.1 below for more details on the dynamics of contact points. Finally, the system of equations is closed by giving the kinematic boundary condition
\begin{equation}\label{evolintro}
\partial_t h(t,x)=u(t,x,h(t,x))\cdot(-h'(t,x),1),\quad x\in\mathcal{I},
\end{equation}
together with non-penetration condition on the vessel
\begin{equation}\label{nonslip}
u\cdot n=0\quad\mbox{on}\quad\Gamma_{w}(t),
\end{equation}
with $n$ the outward normal vector to the vessel boundary.

We refer to this physical scenario as the Muskat problem \cite{Muskat34} with contact points. Remarkably, it is mathematically analogous to the evolution of a fluid in a Hele-Shaw cell \cite{SaffmanTaylor1958}. In it, two parallel plates are close enough together so that the fluid contained has a two-dimensional evolution. In particular, the dynamical equation that models this problem is equivalent to Darcy's law by relating the permeability constant with the distance between the plates. In the setting established in this paper, the vessel is the Hele-Shaw cell, considering the interaction with the fluid and the contact points.

\subsection{Contact points} \label{subsec-dyncont}
It is well-known \cites{Young1805,Laplace,Gauss1829 } that in equilibrium configurations the angle at the contact points is determined by the different surface tension coefficients between the three phases fluid-solid-dry. More precisely, denoting by $\gamma_{\rm sd}$ the solid-dry surface tension and by $\gamma_{\rm sf}$ the solid-fluid surface tension, the equilibrium contact angle $\omega_{\rm eq}$ (see Figure \ref{fig-statvessel}) verifies Young's law
\begin{equation*}
    \cos(\omega_{\rm eq})= \frac{\gamma_{\rm sd} - \gamma_{\rm sf}}{\sigma} = \frac{\llbracket \gamma \rrbracket}{\sigma}.
\end{equation*}
\begin{figure}
    \centering
\includegraphics[scale=1]{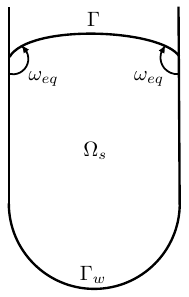}
    \caption{A possible equilibrium configuration.}
    \label{fig-statvessel}
\end{figure}
In this paper we focus in the \emph{partial wetting} regime, that is, when the surface does not touch tangentially the walls but forms real angles. Thus, $$\omega_{\rm eq}\in (0, \pi) \qquad \text{or, equivalently,} \qquad \frac{\llbracket \gamma \rrbracket}{\sigma}\in (-1,1).
$$
On the other hand, regarding dynamical configurations, experiments and simulations have shown that the normal velocity of the contact point $V_{\rm cp}$ is related to the deviation of the dynamical angle $\omega_t$ from the equilibrium angle, that is,
\begin{equation*}
V_{\rm cp} = F \left(\cos(\omega_{\rm eq}) - \cos(\omega_t)\right).
\end{equation*} This type of relation has been derived in several studies \cites{Blake1969,Cox1986,RenWeinan07, RenWeinan11} using thermodynamical, molecular and  hydrodynamical arguments. All of them brought to the same general form of $F$, which is an increasing function of its argument such that $F(0)=0$. For simplicity, we choose here $F(s)= \sigma s$ and, writing the previous relation in our setting, we derive the evolution equation for the contact points \begin{equation}\label{dyn-law}
			\partial_t h (t,\pm1) = \llbracket \gamma\rrbracket \mp \sigma \frac{h'}{\sqrt{1 + (h')^2}}(t, \pm 1).
		\end{equation}
We remark that this additional evolution equation is due to capillarity, while in the pure gravity-driven case the dynamics of the contact points is directly determined by the kinematic condition \eqref{evolintro}. See also \cites{GuoTice2018,GuoTice2024} for mathematical discussion and Section \ref{sec-prevres} below.

\subsection{Main results}

In this paper we prove a priori global-in-time estimates for the one-phase Muskat system \eqref{incompre}-\eqref{dyn-law} for initial data closed enough to the stationary state with exponential convergence. See Proposition \ref{sta-exiuni} below for details of the stationary state. Our result allows the contact points with any angle determined by the stationary state.

In this setting, we disclose a synthesize version of the main result for the sake of exposition. In Theorem \ref{mainTh} it is possible to find a detailed version, once the appropriate formulation of the problem and the detailed functional spaces are explained in detail below.

\begin{theorem}\label{ThIntro}
Let $h(t,x)=h_s(x)+\eta(t,x)$ be the surface profile of a smooth solution to \eqref{incompre}-\eqref{dyn-law}, where $h_s$ is the surface profile of its stationary solution. Assume that $h$ and $h_s$ have the same mean. Then, for any $T>0$, the following bound holds:
$$
\sup_{t\in[0,T]}\mathcal{E}(t)+\int_0^T\mathcal{D}(t)dt\leq C\mathcal{E}(0),
$$
provided that $\mathcal{E}(0)$ is small enough.
Above, the energy term $\mathcal{E}(t)$ is given by
$$
\mathcal{E}(t)=\mathcal{E}_\parallel(t) +\|\eta(t)\|^2_{H^{3/2+\delta}(\mathcal{I})}+\|\partial_t\eta(t)\|^2_{H^{3/2+\delta}(\mathcal{I})}\mbox{ with }\, \mathcal{E}_\parallel(t) =\sum_{j=0}^2\|\partial_t^j\eta(t)\|^2_{H^1(\mathcal{I})},
$$
for any $0<\delta<1/2 $, and the dissipation term $\mathcal{D}(t)$ in \eqref{tot-diss} includes
\begin{equation*}
\sum_{j=0}^2\left((\partial_t^{j+1}\eta)^2(t,-1)+(\partial_t^{j+1}\eta)^2(t,1)\right)+
\sum_{j=0}^1\|\partial_t^j\eta(t)\|^2_{H^{5/2}(\mathcal{I})} + \|\partial_t^2\eta(t)\|^2_{H^{3/2+\delta}(\mathcal{I})}.\end{equation*} In addition, there exists $\lambda>0$ such that the following decay estimate holds:
\begin{equation*}
\sup_{t\in[0,T]}\Big(\mathcal{E}_\parallel(t) +\int_{\Omega(t)} |u(t)|^2  +(\partial_t \eta)^2(t,-1) + (\partial_t \eta)^2(t,1)\Big)\leq C\mathcal{E}_\parallel(0)e^{-\lambda T}.
\end{equation*}
\end{theorem}

\begin{remark}
    The solution $h(t,x)$ to \eqref{incompre}-\eqref{dyn-law} with the regularity given in the theorem above has to satisfy
the fundamental energy-dissipation equality
\begin{equation*}
\begin{aligned}
\frac{d}{dt} \bigg[\int_\mathcal{I} &\Big(\frac{g}{2} h^2(t,x)+\sigma\sqrt{1+(h')^2(t,x)}\Big) dx -\llbracket \gamma \rrbracket \left( h(t,-1) +  h(t,1)\right) \bigg] \\
&+\int_{\Omega(t)} |u(t)|^2 +(\partial_t h)^2(t,-1) + (\partial_t h)^2(t,1)=0.\end{aligned}\end{equation*}
\end{remark}
In a future paper we expect to prove local-in-time well-posedness of the system for initial data such that $\mathcal{E}(0)$ is small.
Thus, global well-posedness and stability for the contact scenario would follow.
\subsection{Previous results and main novelties}\label{sec-prevres}
The Muskat problem and the evolution of fluids in Hele-Shaw cells are classical problems in fluid dynamics \cites{Muskat34,SaffmanTaylor1958}. Originally considered to study the evolution of oil and water in oil recovery, it deals with the interface dynamics among the two immiscible fluids. It models the interaction among two fluids (two-phase case) and also the surface evolution of one fluid with a dry region (one-phase case) in a porous medium or a Hele-Shaw cell. It is possible to find a large amount of recent literature with fundamental results \cite{Sema2017} as these interface evolution problems have interesting dynamics behaviors.

Considering no surface tension and no contact points, the system is locally-in-time well-posed when the Rayleigh-Taylor condition is satisfied \cites{CordobaGancedo2007,CCG2011}. It holds due to gravity when the denser fluid is below the least dense, and the more viscous fluid pushing the less viscous fluid. Sharp local-in-time well-posedness results in subcritical spaces can be found in \cites{CGSV2017,Matioc2019,AlazardLazard2019,NguyenPausader2019} and in critical spaces in \cites{AlazardNguyen02-2021,AlazardNguyen2023,DGN2023}. See also \cites{APW2023,G-JG-SHP2024} for the evolution of cusps. For small initial data, the lineal parabolic regime is stronger than the nonlinear terms, resulting in global-in-time well-posedness of the system \cites{CCGR-PS2016,Cameron2019,GG-JPS2019,NguyenMuskat2022,CordobaLazar2018,AlazardNguyen2023}. On the other hand, large initial data provide very different dynamics in the two-phase and one-phase scenarios. In the two-phase case, initial data given by graphs develop singularities in finite time with over turning profiles \cites{CCFGL-F2012} having loss of regularity \cites{CCFG2013}. In the one-phase case, large initial graphs do not turn \cites{Kim2003,AlazardOneFluid2019} and exist globally-in-time \cites{DGN2023,AlazardKoch2023}.
But it is possible to have finite-time particle collision on smooth interfaces for non-graph initial data \cites{CCFG2016}. On the other hand, finite-time particle collision on smooth interfaces is not possible in the two-phase case \cites{GancedoStrain2014}. Considering impermeable boundaries, global-in-time regularity holds for small data \cites{Granero-Belinchon2014} as well as finite-time blow-up \cite{ZlatosII2024} in similar scenario explored before for several quasi-geostrophic temperature front models \cites{KRYZ2016,GancedoPatel2021}. If the Rayleigh-Taylor condition is not satisfied, the contour evolution problem is ill-posed \cites{CordobaGancedo2007,GG-JPS2019}. However, weak solutions exist developing a mixing zone \cites{CCF2021,Mengual2022}.

Adding capillarity to the model, the system is locally-in-time well-posed with or without satisfying the Rayleigh-Taylor condition initially \cites{DuchonRobert1984,Chen1993,EscherSimonett1997} in subcritical spaces \cite{HQNguyen2019}. This is due to the fact that surface tension adds a higher-order nonlinear parabolic term to the system. However, initial data with lack of Rayleigh-Taylor condition produce instabilities related with fingering \cites{Otto1997,GHS2007,EscherMatioc2011}. Without gravity, there exist close to circle global-in-time solutions \cites{ConstantinPugh1993,Chen1993,Tanveer2011}. Gravity unstable solutions can be stabilized globally in time with surface tension for near flat solution \cite{GG-BS2020} and near circle moving bubbles \cites{GG-JPS2023,GG-JPS2024}. Nonetheless, initial stable solutions converge to no capillarity solutions as surface tension coefficient goes to zero \cite{Ambrose2014} in subcritical spaces \cite{FlynnNguyen2021}. In both cases, gravity stable and unstable solutions, with Lipschitz initial data, gain regularity \cite{Chen2024}. Arbitrarily large Lipschitz data with a small critical Sobolev norm have recently been proven to provide global-in-time solutions for stable two-phase cases \cite{Lazar2024}. With impermeable boundaries, there have been an intense study of thin-film models and lubrication approximations together with the convergence to the original Muskat solutions \cites{BertozziPugh96,laurencot2017self,CENV2018,BG-B19} even for gravity unstable situations \cite{BocchiGancedo2024}.

The surface tension case with constant points was first considered for the Muskat problem in \cites{BazaliyFriedman01,BazaliyFriedman02} for scenarios without gravity force. The fluid spreads over a flat impermeable fixed boundary, where well-posedness is obtained for initial configurations with no moving contact points. Further developments were given in \cites{KnupferMasmoudi2013,KnupferMasmoudi2015} for the same scenario, allowing movement of contact points and having a fixed small contact angle of size $\varepsilon$. The authors proved global-in-time well-posedness, obtaining uniform estimates in $\varepsilon$, showing convergence to a family of thin-film approximation equations.

The scenario considered in this paper allows the dynamics of the contact points and angles of any size, determined by the stationary state, modeled as in Section \ref{subsec-dyncont} above. A similar contact points study has recently been developed to model the evolution of a viscous fluid in a vessel using the 2D Stokes equations \cites{GuoTice2018,ZhengTice2017} and the 2D Navier-Stokes equations \cites{GuoTice2024,GTWZ2024}. The authors first developed global-in-time a priori estimates for the models \cites{GuoTice2018,GuoTice2024} and later the local-in-time well-posedness theory \cites{ZhengTice2017,GTWZ2024}
to obtain the stability of the contact points problem.

In this paper, we obtain global-in-time a priori estimates for the Muskat problem with contact points, and therefore deal with the more singular Darcy's law to model the flow. In this setting, we close a scheme of a priori estimates, obtaining a bootstrap argument from the energy-dissipation control of the time derivatives to higher spatial regularity via elliptic estimates for the system. The dissipation is obtained through the potential formulation of the system, taking $u=\nabla\phi$, as shown in Section \ref{subsec-pot} below. In this scenario, the zero Neumann condition (non-penetration) on the fixed boundary for the potential appears inherently, allowing the movement of the contact point. This is a big difference with the Stokes and Navier-Stokes equation to model the flow, where a zero Dirichlet (no-slip) condition on the fixed boundary for the velocity does not allow the movement of the contact points, and therefore in \cites{GuoTice2018,GuoTice2024} a Navier-slip boundary condition is used. In \cite{GuoTice2018}, the authors derived global-in-time a priori estimates without imposing restrictions on the contact angle $\omega\in (0,\pi)$. However, due to the eigenvalues of the pencil operator associated with an elliptic problem they studied, they are forced to work with weighted $L^2$-based Sobolev spaces with weight exponent $\delta$ satisfying
$$\max \left(0, 2-\frac{\pi}{\omega}\right)<\delta<1,$$
in order to gain regularity through elliptic estimates. In \cite{GuoTice2024}, the authors managed to kickstart the elliptic gain switching from weighted $L^2$-based spaces to non-weighted $L^q$-based spaces.
In this paper, we perform a priori energy estimates for a Darcy flow considering the same geometry. An important difference, with respect to these two works, is that the analysis is carried out in non-weighted Hilbert spaces, while, at the same time, we allow contact angles $\omega\in (0,\pi)$. This is done by taking advantage of the Neumann problem solved by the velocity potential and the spectral properties of the associated pencil operator around the corners. We are then able to bootstrap from the appropriate energy-dissipation control of the time derivatives to higher spatial regularity via elliptic estimates.

The dynamics of contact points has also been recently studied for the water waves problem, that is, for free-surface gravity-driven incompressible Euler flows. A priori estimates were shown for any dimension in \cite{Poyferre2019} for angles smaller than a dimensional constant, preventing singularities in the elliptic equations providing the model. Adding surface tension to the model, a priori estimates were shown for angles less than $\pi/6$ \cite{MingWang2020} and local-in-time well-posedness for angles less than $\pi/16$ \cite{MingWang2021}. The authors have recently improved the results to acute angles \cite{MingWang2024}.\\
For the floating body problem in water waves, the dynamics of contact points was first discussed in \cite{Lannes2017}. Recently, its local well-posedness for a fixed object
was established in \cite{LanMin24}. In particular, we emphasize a trace theorem for the homogeneous Sobolev space $\dot{H}^1$ developed there that
we use in our approach, see Section \ref{subsec-L2pot}. It allows us to get control of the mean of the potential and therefore of its $L^2$-norm by the dissipation, as Poincaré-type inequalities can not be used directly due to lack of zero Dirichlet boundary conditions. This represents another main difference with respect to \cites{GuoTice2018, GuoTice2024}, where the Navier-slip boundary condition used in the model permits to directly have such a control. In shallow water asymptotic models, considering partially immersed objects with vertical walls \cites{Bocchi2020, Bocchi2020-1, BeckLannes2022, BocchiHeVergara2023, IguchiLannes2025} results in fixed contact points or lines. In contrast, for boat-shaped geometries, contact dynamics was studied in \cite{IguchiLannes2021} and more recently in \cite{IguchiLannes2026}.

\subsection{General notation} We denote by $\mathcal{I}$ the interval $(-1,1)$.
We write $f'$ for the $x$-derivative of functions $f(x)$ or $f(t,x)$ that spatially depend only on the variable $x\in\mathcal{I}$, while we write $\partial_x$ or $\partial_y$ for functions that spatially depend on $(x,y)\in\mathbb{R}^2$.
We denote by $N_f$ the outward normal vector to the one-dimensional surface parameterized by the function $f$, while we denote by $n$ a generic unit outward normal vector. $B_r(x_0)$ and $B_r(x_0, y_0)$ denote the one-dimensional and two-dimensional ball with radius $r>0$ centered in $x_0\in\mathbb{R}$ and $(x_0, y_0)\in \mathbb{R}^2$, respectively.\\
\emph{Constants:} throughout the paper, $C > 0$  denotes a generic constant that can depend on the parameters of the problem; when necessary, we comment on the dependence of these constants.\\
\emph{Spaces, norms and traces:} we denote by $L^p$ the usual Lebesgue space for $1\leq p\leq \infty$, by $H^s$ the usual $L^2$-based Sobolev space of order $s\in \mathbb{R}$ and by $\|\cdot\|_{L^p}$, $\|\cdot\|_{H^s}$ their respective norms. We denote by $\dot{H}^{s}$ the homogeneous $L^2$-based Sobolev space of order $s>0$ and by $\|\cdot\|_{\dot{H}^s}$ its semi-norm. We denote by $\mathring{H}^2$ the quotient space $H^2/\mathbb{R}$ and by  $\|\cdot\|_{\mathring{H}^2}$ its norm $\|\nabla \cdot \|_{H^1}$.
For $\alpha\in \mathbb{R}$, we use the compact notation $H^{\alpha+}$ that stands for $H^{\alpha+\delta}$ for any $0<\delta<1/2$. For the sake of readability, we write functions defined on domains directly within boundary integrals and boundary norms, rather than using their corresponding traces.

\subsection{Outline of the paper} The rest of the paper is devoted to providing the proof of Theorem \ref{ThIntro}. Section \ref{sec-pot&fix} introduces the potential and fixed-boundary formulations used throughout the paper. In Section \ref{sec-ED}, we show the basic a priori energy-dissipation equalities satisfied by the solution. Section \ref{sec-enediss} establishes different energies and dissipation terms of various orders needed to obtain the a priori estimates. In Section \ref{sec-adddiss}, we address additional dissipation and a trace theorem required to obtain Poincaré-type inequalities. Section \ref{sec-ellest} details results on elliptic estimates used to bootstrap regularity. In Section \ref{sec-NL}, the inequalities developed in Sections \ref{sec-adddiss} and \ref{sec-ellest} are used to handle the nonlinear terms of the system. Gathering all previous estimates, we conclude the proof of the main result in Section \ref{sec-mainres}. For completeness, technical estimates are provided in Appendix \ref{appendix}.

\section{Potential and fixed-boundary formulations}  \label{sec-pot&fix}

\subsection{Potential formulation}\label{subsec-pot}
		Without any loss of generality, let us set $\mu=\nu=\rho=1$.
		Using Darcy's law \eqref{Darcy}, the fluid velocity can be written as the gradient of the potential $\phi=-p- g y$. Writing \eqref{evolintro} in terms of $\phi$ yields
		\begin{equation}\label{evo-eq}
			\partial_t h = \nabla \phi \cdot N_h \quad  \mbox{on}\quad \Gamma(t),
		\end{equation}where  $N_h(t,x)=(-h'(t,x),1)$ is the outward normal vector to $\Gamma(t)$,  coupled with the Dirichlet-Neumann elliptic problem
		\begin{equation}\label{DNpb}
			\begin{aligned}
				\Delta \phi=0\quad & \mbox{in}\quad \Omega(t),\\[5pt]
				\phi=-gh +\sigma \kappa_h\quad & \mbox{on}\quad \Gamma(t),\\[5pt]
				\nabla \phi\cdot n=0\quad & \mbox{on}\quad \Gamma_w(t),\\[5pt]
			\end{aligned}
		\end{equation}and the evolution of the contact points
        \begin{equation}\label{evo-contact}
			\partial_t h (t,\pm1) = \llbracket \gamma\rrbracket \mp \sigma \frac{h'}{\sqrt{1 + (h')^2}}(t, \pm 1).
		\end{equation}
		We will refer to Dirichlet-Neumann elliptic problems also as ``mixed" elliptic problems. Introducing the well-known Dirichlet-to-Neumann operator $\mathrm{DN}$ (see \cite{Lannes2013, MingWang2017}), which maps the Dirichlet datum in \eqref{DNpb} to the normal derivative of the solution to \eqref{DNpb} at $\Gamma(t)$, \eqref{evo-eq}-\eqref{DNpb} can be also written as
		\begin{equation*}
			\partial_t h =\mathrm{DN}(-gh + \sigma \kappa)  \quad  \mbox{on}\quad \Gamma(t).
		\end{equation*}The same problem can be described in a different way. Indeed, \eqref{evo-eq} can be understood as a Neumann boundary condition for $\phi$ if we think about $\partial_t h$ as a prescribed quantity. Then, \eqref{evo-eq}-\eqref{DNpb} can be reformulated as
		\begin{equation*}
			-gh +\sigma\kappa_h =\phi \quad \mbox{on} \quad \Gamma(t)
		\end{equation*}
		coupled with the Neumann-Neumann elliptic problem
		\begin{equation}\label{NNpb}
			\begin{aligned}
				\Delta \phi=0\quad & \mbox{in}\quad \Omega(t),\\[5pt]
				\nabla \phi \cdot N_h=\partial_t h \quad & \mbox{on}\quad \Gamma(t),\\[5pt]
				\nabla \phi\cdot n=0\quad & \mbox{on}\quad \Gamma_w(t).\\[5pt]
			\end{aligned}
		\end{equation}We will see in Section \ref{sec-ellest} that the gain of higher spatial regularity for the potential, necessary for the closure of the scheme of a priori estimates, strongly relies on the structure of the Neumann problem \eqref{NNpb}. Moreover, it allows us to avoid any restrictions on the contact angle at the corners of the stationary domain, in terms of which we reformulate the free-boundary problem in the next section.			
		
	\subsection{Fixed-boundary formulation}

		As usually done in free-boundary problems, we reformulate the problem in a fixed framework. Our choice for the reference domain in which we recast the problem is the stationary domain. Stationary solutions to \eqref{evo-eq}-\eqref{evo-contact} having only two contact points at the lateral walls of the vessel and none at its bottom are couples $(h_s, \phi_s)$, with $\phi_s\in \mathbb{R}$, that solve the elliptic problem
			\begin{align}
				& -gh_s + \sigma \kappa_{h_s} = \phi_s \quad\mbox{in}\quad \mathcal{I}, \label{stationary1}\\[5pt]
				&\frac{h'_s}{\sqrt{1+(h'_s)^2}}(\pm 1)= \pm\frac{\llbracket \gamma \rrbracket}{\sigma},\label{stationary2} \\[5pt]
                &h_s(x) > h_w(x) \quad\mbox{for}\quad x\in\overline{\mathcal{I}}\label{stationary3},
			\end{align}
		where $\kappa_{h_s}$ denotes the mean curvature \eqref{curvature} of stationary surface $h_s$. For technical reasons, see Lemma \ref{lemma-diffeo},  we consider the case \begin{equation}\label{min>max}
		    \min_{\overline{\mathcal{I}}} h_s > \max_{\overline{\mathcal{I}}} h_w.
		\end{equation} In this situation, there exists a unique (smooth) stationary solution and we refer to Figure \ref{fig-plots-hs} for a qualitative description of $h_s$ according to the sign of $\llbracket \gamma \rrbracket$.
        \begin{figure}
            \centering
\includegraphics[scale=0.5]{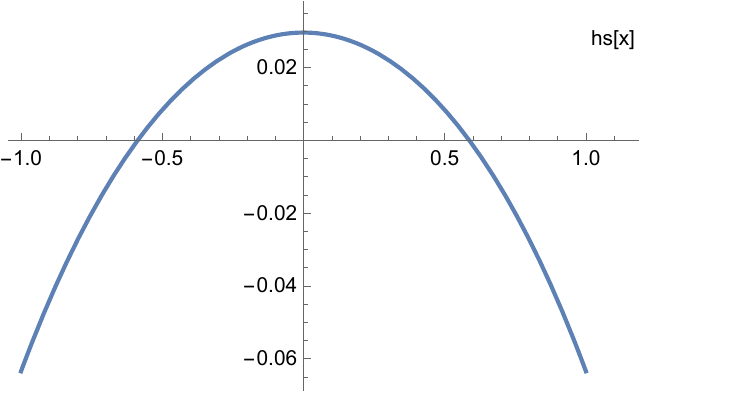}
\includegraphics[scale=0.5]{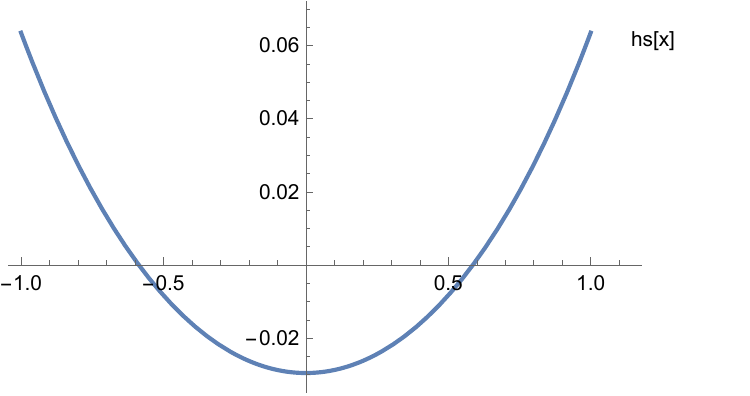}\vspace{1em}
\includegraphics[scale=0.5]{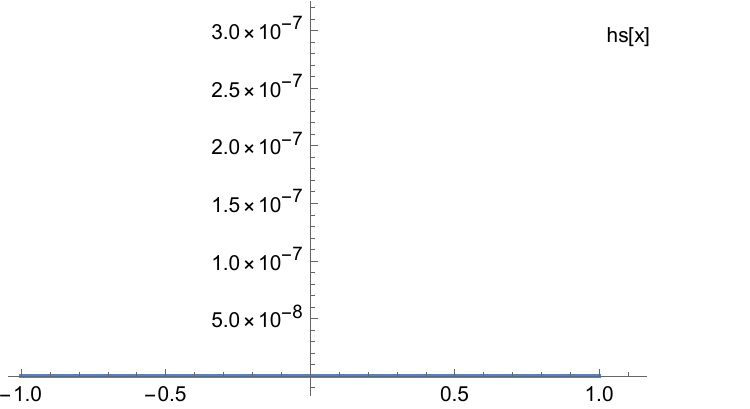}
            \caption{Plots of the different shapes of $h_s$ with zero mean over $\mathcal{I}$: concave when $\llbracket \gamma \rrbracket<0$, convex when $\llbracket \gamma \rrbracket>0$ and flat  when $\llbracket \gamma \rrbracket=0$.}
            \label{fig-plots-hs}
        \end{figure}

\begin{proposition}\label{sta-exiuni}
	Let $\llbracket \gamma \rrbracket/\sigma \in (-1,1)$. There exists some $m\geq 0$ such that, for any $M > m$, the stationary problem \eqref{stationary1}-\eqref{stationary2}-\eqref{min>max} admits a unique solution $(\phi_s, h_s)$, where $\phi_s\in \mathbb{R}$ and $h_s$ is an even $ C^\infty(\overline{\mathcal{I}})$  function, with
		 \begin{equation*}
			 \int_{\mathcal{I}} (h_s-h_w)(x)dx= M \qquad \text{and} \qquad \phi_s= \llbracket \gamma \rrbracket -\frac{g}{2}\Big(M + \int_\mathcal{I} h_w(x)dx\Big) .\end{equation*}
		\end{proposition}

        \begin{proof}
First of all, we set
\begin{equation}\label{def-phis}
    \phi_s
        =\llbracket \gamma \rrbracket -\frac{g}{2}\Big(M + \int_\mathcal{I} h_w(x)dx\Big).
\end{equation}
Second, we know from Theorems F.2 and F.3 of \cite{GuoTice2018} that there exists an explicit profile $\chi\in C^\infty(\overline{\mathcal{I}})$ satisfying
\begin{equation*}\label{stationary-0}
		\begin{aligned}
				& -g\chi + \sigma \kappa_{\chi} = 0 \quad\mbox{in}\quad \mathcal{I},\\[5pt]
				&\frac{\chi'}{\sqrt{1+(\chi')^2}}(\pm 1)= \pm\frac{\llbracket \gamma \rrbracket}{\sigma}.
			\end{aligned}
		\end{equation*}Then, we set \begin{equation}\label{def-hs}
    h_s(x)=\chi(x) -\min_{\overline{\mathcal{I}}} \chi + \max_{\overline{\mathcal{I}}} h_w +K \qquad \text{for} \quad x\in \overline{\mathcal{I}},
\end{equation} with $K\in \mathbb{R}$ to be determined. After injecting \eqref{def-phis}-\eqref{def-hs} into \eqref{stationary1} and integrating over $\mathcal{I}$, we obtain 
\begin{equation*}
-g \int_\mathcal{I} \Big(\chi(x)-\min_{\overline{\mathcal{I}}} \chi + \max_{\overline{\mathcal{I}}} h_w +K\Big)dx + 2\llbracket \gamma \rrbracket  = 2\llbracket \gamma \rrbracket - g \Big(M + \int_\mathcal{I} h_w(x)dx\Big),
\end{equation*}which yields both
\begin{equation*}
 \int_\mathcal{I}(h_s-h_w)(x)dx=M \qquad \text{and}\qquad   K= \frac{M-m}{2},
 \end{equation*}
 where
 \begin{equation*} \quad m=\int_\mathcal{I} \Big(\chi(x) -\min_{\overline{\mathcal{I}}} \chi + \max_{\overline{\mathcal{I}}} h_w - h_w(x)\Big)dx\geq 0.\end{equation*}
Choosing the fluid mass $M>m$ implies that $K>0$ and we deduce from \eqref{def-hs} that
\begin{equation*}
\min_{\overline{\mathcal{I}}} h_s= \min_{\overline{\mathcal{I}}} \chi -\min_{\overline{\mathcal{I}}} \chi + \max_{\overline{\mathcal{I}}} h_w + K > \max_{\overline{\mathcal{I}}} h_w.
\end{equation*}
The uniqueness part follows the same argument used in Theorem F.1 of \cite{GuoTice2018}.
\end{proof}

	We then denote the stationary domain by
		\begin{equation}\label{Omega-s}
  \Omega_s= \left\{(x,y)\in \mathbb{R}^2 \ | \ x\in \mathcal{I}, \  h_w(x)<y< h_s(x) \right\},\end{equation}
		where $h_s$ is the surface profile of the unique solution to \eqref{stationary1}-\eqref{stationary2}-\eqref{min>max}. Its  fluid and solid boundary parts are, respectively,
		\begin{equation*}\begin{aligned}
				&\Gamma=\left\{(x, h_s(x)) \ | \ x\in \mathcal{I} \right\},\\[5pt] &\Gamma_w=\left\{(\pm 1, y) \ | \   h_w(\pm 1)\leq y \leq h_s(\pm 1) \right\} \ \cup \ \left\{(x, h_w(x)) \ | \ x\in \mathcal{I}\right\}. \\[5pt]
				\end{aligned}
		\end{equation*}
        Note that the partial wetting assumption $\llbracket \gamma\rrbracket/\sigma \in (-1,1)$ guarantees that  $h_s\in C^\infty(\overline{\mathcal{I}})$, so that $\Omega_s$ is a Lipschitz domain and does not present cusps. In order to reformulate the problem as a fixed-boundary problem, we introduce a diffeomorphism $\varphi_\eta: (0,T) \times\overline{\Omega}_s\rightarrow \overline{\Omega}(t)$ such that \begin{equation}\label{diff-prop}
			\varphi_\eta (t,\Gamma_w)=\Gamma_w(t), \qquad \ \varphi_\eta(t,\Gamma)=\Gamma(t) \qquad  \forall t\in (0,T)\end{equation}
and the surface perturbation $\eta=h-h_s$. In the next lemma we make explicit the choice of the diffeomorphism and show its regularizing property.
To this end, we introduce a smooth cut-off function $\xi:\mathbb{R}\rightarrow \mathbb{R}$ with $0\leq \xi \leq 1$ and
\begin{equation}\label{cutoff-diffeo}\begin{aligned}
\xi(y)=0 \quad \text{for}\quad  y\leq \max_{\overline{\mathcal{I}}} h_w +(\min_{\overline{\mathcal{I}}} h_s -\max_{\overline{\mathcal{I}}} h_w)/4, \\ \xi(y)=1\quad  \text{for}\quad  y \geq \min_{\overline{\mathcal{I}}} h_s -(\min_{\overline{\mathcal{I}}} h_s -\max_{\overline{\mathcal{I}}} h_w)/4 .
    \end{aligned}
\end{equation}

\begin{lemma}\label{lemma-diffeo}
Let $\eta\in H^{3/2+}(\mathcal{I})$, $\xi$ be as in \eqref{cutoff-diffeo} and $E$ be a bounded extension operator from $H^s(\mathcal{I})$ to $H^s(\mathbb{R})$ for any real $s\geq0$. We define the mapping
	\begin{equation}\label{diffeo}\varphi_\eta(t,x,y)= \big(x, y+ \xi(y)\eta^\dag(t,x,y)\big)\end{equation} where $\eta^\dag(t,x,y)= P \ast E\eta (t,x,y-h_s(x))$ and
\begin{equation*}
 P(x,y)=-\frac{1}{\pi} \frac{y}{x^2 +y^2}
\end{equation*} is the Poisson kernel and $\ast$ denotes convolution with respect to $x$. There exists $\alpha>0$ such that, for $\|\eta\|_{H^{3/2+}(\mathcal{I})}<\alpha$, the mapping  $\varphi_\eta$ is a $C^1$-diffeomorphism from $\Omega_s$ to $\Omega(t)$ verifying \eqref{diff-prop}. Furthermore, the diffeomorphism is regularizing: for any $s\geq 2$,  there exists a constant $C=C(s,h_s)>0$ such that if $\eta\in H^{s-1/2}(\mathcal{I})$ then
\begin{equation}\label{regular-prop}\begin{aligned}
\| \eta^\dag\|_{H^s(\Omega_s)}\leq C \|\eta\|_{H^{s-1/2}(\mathcal{I})}.
	\end{aligned}
\end{equation}
\end{lemma}
\begin{proof} First, we show that the second component of $\varphi_\eta$ is monotone with respect to $y$ and bounded for fixed $x$, that is,
\begin{equation}\label{detJeta} c \leq \partial_y \varphi_\eta(x,y)=1 + \xi'(y) \eta^\dag(x,y)+\xi(y) \partial_y \eta^\dag(x,y) \leq \frac{1}{c}, \qquad (x,y)\in \Omega_s\end{equation}
for some constant $c>0$.
 From the definition of $\eta^\dag$ and using the change of variable $z=y-h_s(x)$, it follows that
\begin{align*}
\|\eta^\dag\|_{L^\infty(\Omega_s)} \leq \|P \ast E\eta \|_{L^\infty(\mathbb{R}^2_-)}.
\end{align*}Applying the convolution theorem and using the Fourier transform $\widehat{P}(\zeta,y) =e^{y|\zeta|}$ yield that
\begin{equation*}
\|P \ast E\eta (\cdot, y)\|_{L^\infty(\mathbb{R})}\leq \|\widehat{P \ast E\eta} (\cdot, y) \|_{L^1(\mathbb{R})}= \|e^{y|\cdot|}\widehat{E\eta}\|_{L^1(\mathbb{R})}
\end{equation*}for any $y\leq 0$ and, after writing
$$|\widehat{E\eta}(\zeta)|=\frac{1}{(1+ |\zeta|^2)^{1/4+\eps}}(1+ |\zeta|^2)^{1/4+\eps}|\widehat{E\eta}(\zeta)|$$ for $\eps>0$ arbitrarily small,
Cauchy-Schwarz inequality implies that there exists $C>0$ such that
$$
\|P \ast E\eta (\cdot, y)\|_{L^\infty(\mathbb{R})}\leq C \|E\eta\|_{H^{1/2+}(\mathbb{R})}
$$ for any $y\leq 0$. Thus, thanks to the boundedness of the extension operator, we obtain that \begin{equation}\label{est-etadagLinf}\|\eta^\dag\|_{L^\infty(\Omega_s)} \leq C \|E\eta\|_{H^{1/2+}(\mathbb{R})}\leq C \|\eta\|_{H^{1/2+}(\mathcal{I})}.\end{equation} Moreover, after computing that
\begin{align*}
&	\partial_ x \eta^\dag (t,x,y)= \partial_x P \ast E\eta  (t,x,y-h_s(x))- h'_s (x)\ \partial_y P \ast E\eta (t,x,y-h_s(x)),\\[5pt]
&	\partial_ y \eta^\dag (t,x,y)= \partial_y P \ast E\eta  (t,x,y-h_s(x)),
\end{align*}we combine the previous argument with the fact that $\widehat{\nabla P}(\zeta,y)= \binom{i\zeta}{\partial_y} e^{y|\zeta|}$ and find that there exists $C=C(h_s)>0$ such that \begin{equation*}
\|\nabla \eta^\dag \|_{L^\infty(\Omega_s)}\leq C \|\eta\|_{H^{3/2+}(\mathcal{I})}.
\end{equation*} It follows from \eqref{detJeta} that
\begin{equation}\label{est-detJeta}
 1- C_1\|\eta\|_{H^{3/2+}(\mathcal{I})}\leq \det(J_\eta)\leq  1 +C_2 \|\eta\|_{H^{3/2+}(\mathcal{I})}
\end{equation}for some constants $C_1,C_2>0$ (depending only on $\xi$ and $h_s$), so that
there exists $\alpha>0$ such that $$c(\alpha)<\det(J_\eta)<\frac{1}{c(\alpha)},$$ for some $c(\alpha)>0$
  when $\|\eta\|_{H^{3/2+}(\mathcal{I})}<\alpha$.
From the definition of the Poisson kernel, we have that
\begin{equation}\label{trace-diffeo}\eta^\dag (t,x,h_s(x))= P \ast E\eta (t,x,0)= \eta(t, x) \quad \text{for} \quad x\in\mathcal{I},\end{equation}
which implies
\begin{align}\varphi_\eta(t,x,h_s(x))=(x, h(t,x)) \quad \text{for} \quad x\in\mathcal{I},\label{uno}\end{align}
while using the properties of the cut-off $\xi$ yields
\begin{equation}\label{dos} \begin{aligned}
 &\varphi_\eta(t,x, h_w(x))= (x,h_w(x)) \quad \text{for} \quad x\in\mathcal{I},\\[5pt]
 &\varphi_\eta(t,\pm 1, y)= (\pm 1,y + \xi(y)\eta^\dag(t,\pm 1, y)) \quad \text{for} \quad y\in(h_w(\pm 1), h_s(\pm 1)).
\end{aligned}\end{equation}
Hence, $C^1$-regularity and monotonicity of $\varphi_\eta$ together with \eqref{uno}-\eqref{dos} prove that $\varphi_\eta$ is a $C^1$-diffeomorphism from $\Omega_s$ to $\Omega(t)$ that satisfies \eqref{diff-prop}.

We now show that $\varphi_\eta$ is a regularizing diffeomorphism.
Since $\Omega_s$ is bounded, using Hölder inequality and \eqref{est-etadagLinf} we have that
\begin{equation*}
    \|\eta^\dag\|_{L^2(\Omega_s)}\leq |\Omega_s|^{1/2} \|\eta^\dag\|_{L^\infty(\Omega_s)}\leq C\|\eta\|_{H^{1/2+}(\mathcal{I})}.
\end{equation*}
After the change of variable $z=y-h_s(x)$ and thanks to the smoothness of $h_s$, we have that
\begin{equation*}\begin{aligned}
	\|\nabla \eta^\dag \|^2_{L^2(\Omega_s)} \leq &\int_{\mathbb{R}}\int_{-\infty}^{Eh_s(x)} |\partial_{x}P\ast E \eta (x, y - Eh_s(x))|^2 dydx\\[5pt]&+ \left(1+ \|h'_s\|^2_{L^\infty(\mathcal{I})}\right)\int_{\mathbb{R}}\int_{-\infty}^{Eh_s(x)} |\partial_{y}P\ast E\eta (x, y - Eh_s(x))|^2dydx
	\\[5pt]\leq & \ C \int_{\mathbb{R}}\int_{\mathbb{R}_-}|\nabla P\ast E\eta (x, z)|^2 dzdx = C \|\nabla P \ast E\eta \|^2_{L^2(\mathbb{R}^2_-)}.
\end{aligned}
\end{equation*}

Passing to Fourier coordinates in the variable $x$,  using Plancherel identity and the boundedness of the extension operator yield

\begin{equation*}\begin{aligned}
\|\nabla P  \ast E\eta \|^2_{L^2(\mathbb{R}^2_-)}=&	\int_{\mathbb{R}_-}	\int_{\mathbb{R}} \left| \binom{i\zeta}{\partial_y} e ^{y|\zeta|} \widehat{E\eta}(\zeta)\right|^2 d\zeta dy =  \int_{\mathbb{R}}\int_{\mathbb{R}_-} 2|\zeta|^2 e^{2y|\zeta|} |\widehat{E\eta}(\zeta)|^2 dyd\zeta \\[5pt]=&\int_{\mathbb{R}} |\zeta|  |\widehat{E\eta}(\zeta)|^2 d\zeta\leq \|E\eta\|^2_{H^{1/2}(\mathbb{R})}\leq C \|\eta\|^2_{H^{1/2}(\mathcal{I})}.
\end{aligned}
\end{equation*}Arguing in an analogous fashion, we can estimate also the $H^k$-norms for any integer $k\geq 0$. More precisely, for any integer $k\geq 0$, there exists a constant $C=C(k,h_s)>0$  such that
\begin{equation*}
\|\nabla \eta^\dag\|_{H^k(\Omega_s)}\leq C \|\eta\|_{H^{k+1/2}(\mathcal{I})}.
\end{equation*}
The non-integer version of \eqref{regular-prop} is derived by interpolation using the equivalence of the fractional Sobolev-Slobodeckij spaces $H^{k+s}$ with $s\in (0,1)$ and the Sobolev spaces defined as interpolation spaces.
\end{proof}
		
		With the diffeomorphism \eqref{diffeo} at hand, we reformulate the free boundary problem  \eqref{evo-eq}-\eqref{DNpb} in a fixed framework. Let us define the matrices
		\begin{equation}\label{Sigmaeta}\Sigma_\eta=(J_\eta^{-1})^T=\left(\begin{matrix}
				1& \dfrac{-\xi\partial_x \eta^\dag}{1+\xi' \eta^\dag+\xi\partial_y \eta^\dag}\\[10pt]
				0&\dfrac{1}{1+\xi' \eta^\dag+\xi\partial_y \eta^\dag}
			\end{matrix} \right),\end{equation} \begin{equation}\label{Aeta}A_\eta= \det (J_\eta) \ \Sigma^T_\eta \Sigma_\eta=  \left(\begin{matrix}
				1+\xi' \eta^\dag+\xi\partial_y \eta^\dag&-\xi\partial_x \eta^\dag  \\[5pt]
				-\xi \partial_x \eta^\dag &\dfrac{1+ (\xi\partial_x \eta^\dag)^2}{1+\xi' \eta^\dag+\xi\partial_y \eta^\dag}
			\end{matrix} \right).\end{equation}
	Then, \eqref{evo-eq}-\eqref{DNpb}   transforms into
		\begin{equation}\label{evo-eq-fix}
			\partial_t h = \Sigma_\eta\nabla \widetilde{\Phi}\cdot  N_h   \qquad  \mbox{on}\quad \Gamma
		\end{equation}
		coupled with the mixed elliptic problem for the transformed potential $\widetilde{\Phi}=\phi\circ \varphi_\eta$
		\begin{equation}\label{DNpb-fix}
			\begin{aligned}
				\nabla \cdot (A_\eta \nabla\widetilde{\Phi} ) =0\qquad  &\mbox{in}\quad \Omega_s,\\[5pt]
				\widetilde{\Phi}=-gh + \sigma \kappa_h\qquad&\mbox{on}\quad \Gamma,\\[5pt]
				\Sigma_\eta \nabla\widetilde{\Phi} \cdot n=0 \qquad&\mbox{on}\quad \Gamma_w,
			\end{aligned}
		\end{equation}
In fact,  \eqref{evo-eq-fix}-\eqref{DNpb-fix} can be reformulated as a perturbation problem. We know from Proposition \ref{sta-exiuni} that the stationary solution $(h_s, \phi_s)$ solves \eqref{stationary1}-\eqref{stationary2}-\eqref{min>max}. Writing $\widetilde{\Phi} = \phi_s + \Phi$, $h=h_s+ \eta$ and plugging into the equations, we find that the perturbation $(\eta, \Phi)$ solves the evolution equation
	\begin{equation}\label{evo-eq-per}
		\partial_t \eta = \Sigma_\eta \nabla \Phi \cdot N_h \quad \mbox{on}\quad \Gamma
	\end{equation}coupled with the mixed elliptic problem
	\begin{equation}\label{DNpb-per}
		\begin{aligned}
			\nabla \cdot (A_\eta\nabla \Phi)=0 \quad &\mbox{in} \quad \Omega_s,\\[5pt]
			\Phi= -g\eta + \sigma \Big(\frac{\eta'}{(1+ (h'_s)^2)^{3/2}} + \mathcal{R}(h'_s, \eta') \Big)'\quad &\mbox{on} \quad \Gamma,\\[5pt]
			\Sigma_\eta \nabla \Phi \cdot n=0    \quad &\mbox{on} \quad \Gamma_w.
		\end{aligned}
	\end{equation}
				Note that in the right-hand side of the second equation in \eqref{DNpb-per} we have expanded $\kappa_h$ at first order in $\eta'$ and the remainder is given by the smooth mapping  \begin{equation}\label{R}\mathcal{R}(z_1,z_2)= \frac{z_1 + z_2}{\sqrt{1+(z_1+ z_2)^2}} - \frac{z_1}{\sqrt{1+z_1^2}} -\frac{z_2}{(1+z_1^2)^{3/2}}.\end{equation}
	
	Moreover, the evolution equation for the contact points \eqref{evo-contact} becomes
	\begin{equation}\label{evo-contact-per}
		\partial_t \eta (t,\pm1) =  \mp \sigma \Big(  \frac{\eta'}{(1+(h_s')^2)^{3/2}} + \mathcal{R}(h'_s,\eta')\Big)(t,\pm1).
	\end{equation}

	\section{Basic energy-dissipation equalities}\label{sec-ED}
 In this section we derive several energy-dissipation equalities that will be the basis of the scheme of a priori estimates for \eqref{evo-eq-per}-\eqref{evo-contact-per}. Sufficiently regular solutions to \eqref{evo-eq}-\eqref{evo-contact} satisfy the fundamental energy-dissipation equality
	\begin{equation}\label{energy-eq}
		\frac{d}{dt} E(h)+ \int_{\Omega(t)} |\nabla \phi(t)|^2 + (\partial_t h)^2(t,-1) + (\partial_t h)^2(t,1)  =0,
	\end{equation}with the physical energy $E(h)$ given by
	\begin{equation*}
    		E(h)= \int_\mathcal{I} \Big(\frac{g}{2}h^2 + \sigma\sqrt{1+(h')^2}\Big)dx -\llbracket \gamma \rrbracket \left(h(t,-1) +  h(t,1)\right).
	\end{equation*} In the expression of the physical energy, in addition to the standard energy contributions related to the gravity and the surface tension,  a localized energy term appears due to the interaction at the contact points.
	Instead of showing the details for the derivation of \eqref{energy-eq}, we prove an equivalent energy-dissipation equality in terms of the perturbation $(\Phi,\eta)$.
	\begin{proposition}\label{prop-energy}
		Let $h_s$ be the stationary surface as in Proposition \ref{sta-exiuni} and $(\Phi, \eta)$ be a regular solution to \eqref{evo-eq-per}-\eqref{evo-contact-per}. Then, the following equality holds:
		\begin{equation}\label{ene-eq-per}\begin{aligned}
				\frac{d}{dt}& \bigg[ \int_\mathcal{I}\Big(\frac{g}{2}\eta^2 + \frac{\sigma}{2} \frac{(\eta')^2}{(1+ (h'_s)^2)^{3/2}}  +\sigma \mathcal{Q}_0 (h_s',\eta') \Big)dx\bigg] \\[5pt]&
				+ \int_{\Omega_s} \det(J_\eta)|\Sigma_\eta\nabla \Phi|^2 + (\partial_t \eta)^2 (t,-1) + (\partial_t \eta)^2 (t,1)=0,\end{aligned}
		\end{equation}
		with
	$
			\mathcal{Q}_0(h_s', \eta') = \int_0^{\eta'}\mathcal{R}(h_s', z)dz
$ and $\mathcal{R}$ given by \eqref{R}.
	\end{proposition}
	
	\begin{proof}
	Since $\Phi$ is regular and solves the elliptic problem \eqref{DNpb-per}, we have that
		\begin{equation}\label{green-Phi}\begin{aligned}
				0&= \int_{\Omega_s} \nabla \cdot (A_\eta\nabla \Phi) \Phi = \int_{\partial \Omega_s} A_\eta \nabla \Phi \cdot n \Phi - \int_{\Omega_s} \det(J_\eta)|\Sigma_\eta \nabla \Phi|^2 \\[5pt]&= \int_{\partial \Omega_s}   \det(J_\eta)\Sigma_\eta \nabla \Phi \cdot \Sigma_\eta n \Phi - \int_{\Omega_s} \det(J_\eta)|\Sigma_\eta \nabla \Phi|^2 \\[5pt]&= \int_{\Gamma} \Sigma_{\eta}\nabla \Phi \cdot \frac{N_h}{|N_{h_s}|} \Phi  - \int_{\Omega_s} \det(J_\eta)|\Sigma_\eta \nabla \Phi|^2 ,
			\end{aligned}
		\end{equation}
		where $N_{h_s}(x)=(-h'_s(x),1)$ is the outward normal vector to $\Gamma$. In the last equality, on the one hand, we have used that
		\begin{equation*}
				\Sigma_\eta n =n \quad \text{on} \quad \Gamma_w
		\end{equation*}
		and the Neumann-type condition in \eqref{DNpb-per} to derive that the boundary integral over $\Gamma_w$ vanishes. On the other hand,  the chain rule applied to \eqref{uno} implies that $$  \det(J_\eta)\Sigma_\eta \frac{N_{h_s}}{|N_{h_s}|} = \frac{N_h}{|N_{h_s}|} \quad \mbox{on} \quad \Gamma.$$
		Taking the $L^2(\mathcal{I})$-scalar product of \eqref{evo-eq-per} with $g\eta$,  employing the Dirichlet boundary condition in \eqref{DNpb-per} and \eqref{green-Phi} yields
		\begin{equation}\label{scaprod-eta}\begin{aligned}
				\frac{d}{dt}&\int_\mathcal{I} \frac{g}2\eta^2dx \\[5pt]&= -\int_{\Gamma}  \Sigma_{\eta}\nabla \Phi\cdot \frac{N_h}{|N_{h_s}|} \Phi  + \int_\mathcal{I} \sigma \left(\frac{\eta'}{(1+(h'_s)^2)^{3/2}} + \mathcal{R}(h'_s, \eta')\right)' \partial_t \eta dx\\[5pt]&=-\int_{\Omega_s} \det(J_\eta)|\Sigma_\eta \nabla \Phi|^2 + \int_\mathcal{I} \sigma \left(\frac{\eta'}{(1+(h'_s)^2)^{3/2}} + \mathcal{R}(h'_s, \eta')\right)' \partial_t \eta dx.
			\end{aligned}
		\end{equation}
		We focus on the second term in the last line. By integration by parts, it can be written as
		\begin{equation*}
			\begin{aligned}
			&\int_{\mathcal{I}} \sigma \Big(\frac{\eta'}{(1+(h'_s)^2)^{3/2}} + \mathcal{R}(h'_s, \eta')\Big)' \partial_t \eta dx
			\\[5pt]&= \sigma \Big(\frac{\eta'}{(1+(h'_s)^2)^{3/2}} + \mathcal{R}(h'_s, \eta')\Big)  \partial_t \eta   \Big|^{1}_{-1}  - \int_{\mathcal{I}}\sigma \Big(\frac{\eta'}{(1+(h_s')^2)^{3/2}}+\mathcal{R}(h'_s, \eta')\Big)\partial_t\eta' dx
				\\[5pt]&=-(\partial_t \eta)^2(t,-1) -(\partial_t \eta)^2(t,1)-\frac{d}{dt}\int_{\mathcal{I}}\sigma\Big( \frac{(\eta')^2}{2(1+ (h'_s)^2)^{3/2}}  + \mathcal{Q}_0(h'_s, \eta')\Big)dx
			\end{aligned}
		\end{equation*} where in the second equality we have used \eqref{evo-contact-per} and the chain rule
		$$\mathcal{R}(h'_s, \eta')\partial_t\eta' =\partial_{z_2} \mathcal{Q}_0(h'_s,\eta')\partial_t \eta' = \partial_t (\mathcal{Q}_0(h'_s, \eta')).$$
	\end{proof}
	
	\begin{remark}
		Equivalently, Proposition \ref{prop-energy} can be proved starting from the energy equality \eqref{energy-eq}, expressed in terms of the free surface $h$. Indeed, it can be showed that
		\begin{equation*}
			E(h)=E(h_s) + \int_{\mathcal{I}}\Big(\frac{g}{2}\eta^2 + \frac{\sigma}{2} \frac{(\eta')^2}{(1+(h'_s)^2)^{3/2}}  +\sigma \mathcal{Q}_0(h'_s,\eta')\Big)dx.
		\end{equation*}Hence, \eqref{ene-eq-per} follows writing the integral over $\Omega(t)$ in \eqref{energy-eq} as an integral over $\Omega_s$ and using the time-independence of the stationary energy $E(h_s)$.
		\end{remark}
We point out that \eqref{ene-eq-per} offers, as much, a control of the $H^1$-norm of $\eta$. In order to propagate higher regularity we will  take derivatives on the equations and try to find an analogous version of \eqref{ene-eq-per} for them. Due to the presence of different boundaries and in order to maintain the structure of the equations, we apply differential operators that are tangential to the boundaries, which turn out to be only time derivatives in our context.\\
	We then need to study time-differentiated versions of the problem  \eqref{evo-eq-per}-\eqref{DNpb-per}. Applying one time-derivative, we find that $(\partial_t \Phi, \partial_t \eta)$ solves the evolution equation
		\begin{equation}\label{evo-eq-per1}
			\partial_t^2 \eta = \Sigma_\eta \nabla\partial_t\Phi \cdot N_h + F^1_1 \quad \mbox{on} \quad \Gamma
		\end{equation}
		coupled with the mixed elliptic problem
		\begin{equation}\label{DNpb-per1}
			\begin{aligned}
				\nabla \cdot (A_\eta\nabla \partial_t\Phi)=F^1_2 \quad &\mbox{in} \quad \Omega_s,\\[5pt]
				\partial_t\Phi= -g\partial_t \eta + \sigma \left(\frac{\partial_t \eta'}{(1+(h'_s)^2)^{3/2}} + \partial_t (\mathcal{R}(h'_s, \eta'))\right)' \quad &\mbox{on} \quad \Gamma,\\[5pt]
				\Sigma_\eta \nabla \partial_t\Phi \cdot n=F^1_3 \quad &\mbox{on} \quad \Gamma_w,
			\end{aligned}
		\end{equation}with\begin{equation}\label{F1}
			\begin{aligned}
				F^1_1 =\partial_t (\Sigma_\eta^T N_h)\cdot \nabla \Phi, \qquad
				F^1_2=-\nabla \cdot (\partial_t A_\eta \nabla \Phi), \qquad
				F^1_3=- \partial_t\Sigma_\eta\nabla \Phi \cdot
				n .
			\end{aligned}
		\end{equation}The system is completed by the evolution equation for the contact points
		\begin{equation}\label{evo-contact-per1}
			\partial_t^2\eta (t,\pm 1) =  \mp\sigma \left(\frac{\partial_t \eta'}{(1+(h'_s)^2)^{3/2}} + \partial_t (\mathcal{R}(h'_s, \eta'))\right)(t, \pm1).
		\end{equation}
We will derive an energy-dissipation equality for the time-differentiated problem in the same fashion as for \eqref{ene-eq-per}. However, we will face a structural difficulty that prevents to control the new nonlinear terms due to the lack of regularity. Then, we are led to apply second-order time-derivatives to \eqref{evo-eq-per}-\eqref{DNpb-per} and study the twice time-differentiated problem. We find that  $(\partial_t^2\Phi, \partial_t^2 \eta)$ solves the evolution equation
		\begin{equation}\label{evo-eq-per2}
			\partial_t^3 \eta = \Sigma_\eta \nabla\partial^2_t\Phi \cdot N_h + F^2_1 \quad \mbox{on} \quad \Gamma
		\end{equation}
		coupled with the mixed elliptic problem
		\begin{equation}\label{DNpb-per2}
			\begin{aligned}
				\nabla \cdot (A_\eta\nabla \partial^2_t\Phi)= F^2_2 \quad &\mbox{in} \quad \Omega_s,\\[5pt]
				\partial_t^2\Phi= -g\partial_t^2 \eta + \sigma  \left(\frac{\partial_t^2 \eta'}{(1+(h'_s)^2)^{3/2}} + \partial_t^2 (\mathcal{R}(h'_s, \eta'))\right)' \quad &\mbox{on} \quad \Gamma,\\[5pt]
				\Sigma_\eta \nabla \partial^2_t\Phi \cdot n=F^2_3   \quad &\mbox{on} \quad \Gamma_w,
			\end{aligned}
		\end{equation}where
		\begin{equation}\label{F2}
			\begin{aligned}
				&F^2_1=2\partial_t (\Sigma_\eta^TN_h)\cdot  \nabla \partial_t\Phi +  \partial_t^2(\Sigma_\eta^T N_h)\cdot \nabla \Phi, \\[5pt]&
				F^2_2=-\nabla \cdot (2\partial_t A_\eta \nabla \partial_t\Phi +\partial_t^2A_\eta \nabla \Phi ),\\[5pt]&
				F^2_3=- \left(\partial_t\Sigma_\eta\nabla \partial_t\Phi  + \partial_t^2\Sigma_\eta\nabla\Phi \right)\cdot n,
			\end{aligned}
		\end{equation}
	completed by the evolution equation for the contact points
		\begin{equation}\label{evo-contact-per2}
			\partial_t^3\eta (t,\pm1 ) =  \mp\sigma  \left(\frac{\partial_t^2 \eta'}{(1+(h'_s)^2)^{3/2}} + \partial_t^2 (\mathcal{R}(h'_s, \eta'))\right)(t, \pm 1).
	\end{equation}
		
	We then derive the energy-dissipation equalities associated with the once and the twice time-differentiated problems.

		\begin{proposition}\label{prop-higherenergy}
			Let $h_s$ be given by Proposition \ref{sta-exiuni}, let  $(\partial_t^j\Phi,\partial_t^j\eta)$ be a regular solution to \eqref{evo-eq-per}-\eqref{evo-contact-per},\eqref{evo-eq-per1}-\eqref{evo-contact-per1} and \eqref{evo-eq-per2}-\eqref{evo-contact-per2}, respectively, for $j=0,1,2$. Then, the following energy-dissipation equality holds:
			\begin{equation}\label{highene-eq-per}
				\begin{aligned}
					&\frac{d}{dt}\int_{\mathcal{I}} \Big(\frac{g}{2}(\partial_t^j \eta)^2 + \frac{\sigma}{2}\frac{ (\partial_t^j\eta')^2}{(1+ (h'_s)^2)^{3/2}} + \sigma \mathcal{Q}_j(h'_s,\eta')\Big)dx\\[5pt]& + \int_{\Omega_s}\det(J_\eta)|\Sigma_\eta\nabla \partial_t^j \Phi|^2 + (\partial_t^{j+1} \eta)^2(t,-1) + (\partial_t^{j+1} \eta)^2(t,1)= \mathcal{S}_j,
				\end{aligned}
			\end{equation}	where the source terms are $\mathcal{S}_0=0$ and
			\begin{equation}\label{Sj}
			\begin{aligned}
				\mathcal{S}_j=& - \int_{\Omega_s}F_2^j \partial_t^j \Phi  + \int_{\Gamma_w} \det(J_\eta)F_3^j\partial_t^j\Phi\\[5pt]&-\int_{\Gamma} \frac{
					1}{|N_{h_s}|}F_1^j \partial_t^j \Phi
				+ \int_{\mathcal{I}} \sigma\mathcal{F}_j(h'_s,\eta')dx,  \qquad j=1,2
			\end{aligned}
			\end{equation} with

            \begin{equation}\label{Fj}
			\begin{aligned}
		\mathcal{F}_1(h'_s,\eta')&= \frac{(\partial_t \eta')^3}{2}\partial^2_{z_2}\mathcal{R} (h'_s,\eta'),\\
                \mathcal{F}_2(h'_s,\eta')&= \frac{5}{2} (\partial_t^2\eta')^2\partial_t\eta'\partial^2_{z_2}\mathcal{R}(h'_s,\eta')+ \partial_t^2 \eta'(\partial_t \eta')^3\partial^3_{z_2}\mathcal{R} (h'_s,\eta').
			\end{aligned}
			\end{equation}
			The residual energies are given by
			\begin{equation*}\begin{aligned}
		\mathcal{Q}_0(h'_s,\eta')&=\int_0^{\eta'}\mathcal{R}(h'_s, z)dz, \qquad \mathcal{Q}_1(h'_s,\eta')= \frac{(\partial_t\eta')^2}{2}\partial_{z_2}\mathcal{R}(h'_s,\eta'), \\[5pt]
				 \mathcal{Q}_2(h'_s,\eta')&= \frac{(\partial_t^2\eta')^2}{2}\partial_{z_2}\mathcal{R}(h'_s,\eta') + \partial_t^2\eta' (\partial_t\eta')^2\partial^2_{z_2}\mathcal{R}(h'_s,\eta').
		\end{aligned}
			\end{equation*}
		
		\end{proposition}
		\begin{proof}The case $j=0$ is exactly Proposition \ref{prop-energy}, while for $j=1,2$ we argue in the same fashion.
            More precisely, taking the $L^2(\mathcal{I})$-scalar product of \eqref{evo-eq-per1}, respectively \eqref{evo-eq-per2}, with $g\partial_t\eta$, respectively $g\partial_t^2\eta$, we obtain
			\begin{equation*}\begin{aligned}
					\frac{d}{dt}\int_{\mathcal{I}}&\frac{g}{2}(\partial_t^j \eta)^2 dx
		=-\int_{\Gamma}
					\partial_t^j\Phi\Sigma_\eta \nabla \partial_t^j \Phi \cdot \frac{N_h}{|N_{h_s}|}   + \int_{\Gamma} \frac{1}{|N_{h_s}|}g\partial_t^j \eta F_1^jdx\\[5pt]&\quad +\int_{\mathcal{I}} \sigma  \left(\frac{\partial_t^j \eta'}{(1+(h'_s)^2)^{3/2}} + \partial_t^j (\mathcal{R}(h'_s, \eta'))\right)' \left(\partial_t^{j+1}\eta - F^j_1\right)dx \\[5pt]
					&= -\int_{\Omega_s}\det(J_\eta)|\Sigma_\eta \nabla \partial_t^j\Phi|^2  - \int_{\Omega_s} F_2^j \partial_t^j \Phi  + \int_{\Gamma_w}\det(J_\eta)F_3^j\partial_t^j\Phi  \\[5pt]&\quad- \int_{\Gamma}\frac{1}{|N_{h_s}|} \partial_t^j \Phi F_1^j  +\int_{\mathcal{I}} \sigma \left(\frac{\partial_t^j \eta'}{(1+(h'_s)^2)^{3/2}} + \partial_t^j (\mathcal{R}(h'_s, \eta'))\right)' \partial_t^{j+1} \eta dx
				\end{aligned}
			\end{equation*}
			where in the last equality we have used time-differentiated versions of \eqref{green-Phi}. Integrating by parts and using the evolution equation for the contact points  \eqref{evo-contact-per1}, respectively \eqref{evo-contact-per2}, the second term in the last line can be written as
			\begin{equation*}
				\begin{aligned}
					& \int_{\mathcal{I}} \sigma \Big(\frac{\partial_t^j \eta'}{(1+(h'_s)^2)^{3/2}} + \partial_t^j (\mathcal{R}(h'_s, \eta'))\Big)' \partial_t^{j+1} \eta dx  \\[5pt]&=\sigma   \Big(\frac{\partial_t^j \eta'}{(1+(h'_s)^2)^{3/2}} + \partial_t^j (\mathcal{R}(h'_s, \eta'))\Big)\partial_t^{j+1} \eta \Big|^{x=1}_{x=-1}  \\[5pt]&\quad -\int_{\mathcal{I}} \sigma \Big(\frac{\partial_t^j \eta'}{(1+(h'_s)^2)^{3/2}} + \partial_t^j (\mathcal{R}(h'_s, \eta'))\Big) \partial_t^{j+1} \eta' dx\\[5pt]
					&= -(\partial_t^{j+1}\eta)^2(t,-1) -(\partial_t^{j+1}\eta)^2(t,1) - \frac{d}{dt}\int_{\mathcal{I}} \frac{\sigma}{2} \frac{(\partial_t^j\eta')^2}{(1+ (h'_s)^2)^{3/2}}dx \\[5pt]&\quad - \int_{\mathcal{I}}\sigma\partial_t^{j+1}\eta'\partial_t^j  (\mathcal{R}(h'_s,\eta') )dx.
				\end{aligned}
			\end{equation*}
		We show that the last integral can be written as the sum of a time-derivative and other terms. To do that,  we shall distinguish the cases $j=1$ and $j=2$. When $j=1$, using the chain rule yields
			\begin{equation*}
				\begin{aligned}
				\int_{\mathcal{I}}&\sigma \partial_t^2 \eta'\partial_t(\mathcal{R}(h'_s,\eta'))dx=\int_{\mathcal{I}} \frac\sigma2 \partial_t(\partial_t \eta')^2 \partial_{z_2}\mathcal{R}(h'_s,\eta')dx \\[5pt]&= \frac{d}{dt}\int_{\mathcal{I}}\frac\sigma2 (\partial_t \eta')^2\partial_{z_2}\mathcal{R}(h'_s,\eta')dx -  \int_{\mathcal{I}}\frac\sigma2 (\partial_t \eta')^3\partial^2_{z_2}\mathcal{R}(h'_s,\eta')dx.
				\end{aligned}
			\end{equation*}
			When $j=2$, we have
			\begin{equation}\label{term-R2}
				\begin{aligned}
					\int_{\mathcal{I}}\sigma \partial_t^3 \eta'\partial_t^2(\mathcal{R}(h'_s,\eta'))dx=&\int_{\mathcal{I}} \sigma \partial_t^3 \eta' \partial^2_{z_2}\mathcal{R} (h'_s,\eta')(\partial_t\eta')^2 dx\\
                    &+\int_{\mathcal{I}} \sigma\partial_t^3 \eta' \partial_{z_2}\mathcal{R}(h'_s,\eta')\partial_t^2 \eta'dx =I +II.
				\end{aligned}
			\end{equation}
		We write the first term as
			\begin{equation*}
				\begin{aligned}
					I= \ & \frac{d}{dt}\int_{\mathcal{I}} \sigma \partial_t^2 \eta' (\partial_t\eta')^2\partial^2_{z_2}\mathcal{R} (h'_s,\eta')dx\\[5pt]&- 2\int_{\mathcal{I}}\sigma (\partial_t^2 \eta')^2\partial_t \eta' \partial^2_{z_2}\mathcal{R}(h'_s,\eta')dx  -\int_{\mathcal{I}} \sigma \partial_t^2\eta' (\partial_t \eta')^3\partial^3_{z_2}\mathcal{R} (h'_s,\eta')dx,
				\end{aligned}
			\end{equation*}
			and, using again the chain rule, the second term as
			\begin{equation*}
				\begin{aligned}
					II&=\int_{\mathcal{I}} \frac\sigma2 \partial_t(\partial_t^2 \eta')^2 \partial_{z_2}\mathcal{R}(h'_s,\eta')dx \\&= \frac{d}{dt}\int_{\mathcal{I}}\frac\sigma2 (\partial_t^2 \eta')^2\partial_{z_2}\mathcal{R}(h'_s,\eta')dx -  \int_{\mathcal{I}}\frac\sigma2(\partial_t^2 \eta')^2\partial^2_{z_2}\mathcal{R}(h'_s,\eta')\partial_t\eta' dx.
				\end{aligned}
			\end{equation*} Thus, the left-hand side in \eqref{term-R2} reads
			\begin{equation*}
				\begin{aligned}
					\int_{\mathcal{I}}&\sigma \partial_t^3 \eta' \partial_t^2(\mathcal{R}(h'_s,\eta'))dx \\&= \frac{d}{dt}\int_{\mathcal{I}}\sigma \Big(\frac12(\partial_t^2\eta')^2 \partial_{z_2}\mathcal{R}(h'_s,\eta') + \partial_t^2\eta' (\partial_t \eta')^2\partial^2_{z_2}\mathcal{R}(h'_s,\eta')\Big)dx\\[5pt]&\qquad
					- \int_{\mathcal{I}} \sigma \Big( \frac{5}{2}(\partial_t^2\eta')^2\partial_t\eta' \partial^2_{z_2}\mathcal{R}(h'_s,\eta') + \partial_t^2\eta' (\partial_t\eta')^3\partial^3_{z_2}\mathcal{R}(h'_s,\eta')\Big)dx.
				\end{aligned}
			\end{equation*}
			Gathering all together and introducing the residual energies $\mathcal{Q}_j(h'_s, \eta')$ for $j=1,2$, we obtain \eqref{highene-eq-per}.
		\end{proof}
		
		\section{Energies and dissipations} \label{sec-enediss}In this section we address the different energies and dissipations of the problem we are studying. Before introducing their definitions, we recall the method of improved energy and dissipation used in \cite{GuoTice2018} that permits to close a scheme of a priori estimates when the structure of the nonlinearities does not directly allow it. More precisely, let us consider a system of PDEs that admits an energy-dissipation equality of the form
	\begin{equation}\label{EDeq}
		\frac{d}{dt}\mathcal{E}_\parallel (t)+ \mathcal{D}_\parallel (t)= \mathcal{N}(t)
	\end{equation}with some basic energy $\mathcal{E}_\parallel\geq 0$ and some basic dissipation $\mathcal{D}_\parallel\geq 0$. We assume that the nonlinearity $\mathcal{N}$ has a structure that prevents a bound of the type
	\begin{equation*}
		|\mathcal{N}(t)|\leq C \mathcal{E}_\parallel^\theta(t)\mathcal{D}_\parallel (t)\quad \text{with} \quad \theta>0,
	\end{equation*} which would be necessary to obtain a standard parabolic a priori estimate. The impossibility to work directly with the basic energy and dissipation requires the search of an improved energy $\mathcal{E} (t)= \mathcal{E}_\parallel (t)+\mathcal{E}_\perp(t)$ and an improved dissipation $\mathcal{D}(t)= \mathcal{D}_\parallel(t)+ \mathcal{D}_\perp(t)$ for which the analogous bound can be obtained, that is,
	\begin{equation}\label{control-N}
		|\mathcal{N}(t)|\leq C \mathcal{E}^\theta(t) \mathcal{D}(t)\quad \text{with} \quad \theta>0,
	\end{equation}and such that the additional energy and  dissipation  $\mathcal{E}_\perp$ and $\mathcal{D}_\perp$ verify the bounds
	\begin{equation}\label{prop-addED}
		\frac{d}{dt} \mathcal{E}_\perp (t)\leq C\mathcal{D}(t) \qquad \text{and} \qquad	\mathcal{D}_\perp(t) \leq C \left( \mathcal{D} _\parallel (t)+ \mathcal{E}^\theta(t) \mathcal{D}(t)\right).
	\end{equation}Indeed, the basic energy-dissipation equality \eqref{EDeq} combined with  \eqref{prop-addED} and the control \eqref{control-N}  imply that \begin{equation*}
 \frac{d}{dt}\mathcal{E}_\parallel(t)+ C \mathcal{D}(t) \leq C \mathcal{E}^\theta (t)\mathcal{D}(t).
	\end{equation*} Consequently, assuming $\mathcal{E}(t)$ sufficiently small and integrating in time then yields the desired parabolic a priori estimate
	\begin{equation}\label{apriori-EDest}
	  \mathcal{E}(t) + \int_0^t\mathcal{D}(s)ds \leq C \mathcal{E}(0).
	\end{equation}
	The crucial point of this improved energy-dissipation method is to find suitable additional energy $\mathcal{E}_\perp$ and dissipation $\mathcal{D}_\perp$ that obey the structure \eqref{prop-addED}. Note that the first bound in \eqref{prop-addED} can be reformulated in the case when $\mathcal{E}_\perp$ is a finite sum of squares of Hilbert norms. Indeed, given a Hilbert space $\mathcal{H}$ endowed with the norm $\|\cdot\|_\mathcal{H}$ and the scalar product $(\cdot, \cdot)_\mathcal{H}$,  Cauchy-Swartz and Young's inequalities yield the bound
	\begin{equation*}
		\frac{d}{dt}  \|f(t)\|^2_\mathcal{H} =2(f(t), \partial_t f (t))_\mathcal{H} \leq \|f (t)\|^2_\mathcal{H} + \|\partial_t f (t)\|^2_\mathcal{H}.
	\end{equation*} Therefore, we can replace  the first bound in \eqref{prop-addED} by
 \begin{equation*}
     \|f (t)\|^2_\mathcal{H} + \|\partial_t f (t)\|^2_\mathcal{H} \leq C \mathcal{D}(t).
 \end{equation*} One way to verify the previous bound is to consider an improved dissipation that contains both $\mathcal{E}_\perp$ and its time-derivative analogous. Nevertheless, in the case that these terms are not already contained in the basic dissipation $\mathcal{D}_\parallel$, they have to verify the second bound in \eqref{prop-addED}.\\\\
The goal of this paper is to derive an estimate of the type \eqref{apriori-EDest}. As a first step, we reformulate the energy-dissipation equations, obtained in the previous section, in the form \eqref{EDeq}. To this end, we introduce the basic energy
	\begin{equation}\label{basic-en}
		\mathcal{E}_\parallel(t)= \sum_{j=0}^{2}\|\partial_t^j \eta(t)\|^2_{H^1(\mathcal{I})}
	\end{equation}
	and the basic dissipation
	\begin{equation}\label{basic-dis}
		{\mathcal{D}}_\parallel(t)= \sum_{j=0}^{2}\left(\| \nabla \partial_t^j \Phi(t)\|^2_{L^2(\Omega_s)}  + (\partial_t^{j+1}\eta)^2 (t,-1)+ (\partial_t^{j+1}\eta)^2 (t,1)\right).
	\end{equation}
	Note that the norms in these basic energy and dissipation are not exactly the ones appearing in \eqref{ene-eq-per} and \eqref{highene-eq-per}. We will show later that the two formulations are equivalent, allowing us to work directly with the quantities in \eqref{basic-en}-\eqref{basic-dis}. From Proposition \ref{prop-higherenergy} we know that
	\begin{equation*}
		\frac{d}{dt}\mathcal{E}_\parallel (t)+ {\mathcal{D}}_\parallel(t) =\mathcal{S}_1(t) + \mathcal{S}_2(t)
	\end{equation*}with $\mathcal{S}_j$ as in \eqref{Sj}. The structure of these terms will require a higher regularity than the one provided at the basic level. Our strategy is to bootstrap from the control of $\mathcal{E}_\parallel $ and $\mathcal{D}_\parallel$ to higher spatial regularity by employing elliptic estimates, derived in Section \ref{sec-ellest}, necessary to close the scheme of a priori estimates. This extra control is encoded in the improved energy
	\begin{equation}\label{impro-ene}
		\mathcal{E}(t)= \mathcal{E}_\parallel(t) + \mathcal{E}_\perp(t)= \mathcal{E}_\parallel(t) + \|\eta (t)\|^2_{H^{3/2+}(\mathcal{I})} + \|\partial_t \eta (t)\|^2_{H^{3/2+}(\mathcal{I})}
	\end{equation} and in the improved dissipation
	\begin{equation}\label{tot-diss}\begin{aligned}
		\mathcal{D}(t)= \mathcal{D}_\parallel(t) + \mathcal{D}_\perp(t)= \ &\mathcal{D}_\parallel (t)+ \sum_{j=0}^{1} \| \partial_t^j\eta (t)\|^2_{H^{5/2} (\mathcal{I})}+\| \partial_t^2\eta (t)\|^2_{H^{3/2+} (\mathcal{I})} \\[2pt]&+\sum_{j=0}^{2}\| \partial_t^j \Phi(t)\|^2_{L^2(\Omega_s)}+  \| \partial_t\Phi (t)\|^2_{\mathring{H}^2(\Omega_s)} ,\\[5pt]
        \end{aligned}
	\end{equation}
	where $\|\partial_t \Phi\|_{\mathring{H}^2(\Omega_s)}= \|\nabla \partial_t \Phi\|_{H^1(\Omega_s)}$.

\section{Additional dissipation}\label{sec-adddiss}

\subsection{Control of the $L^2$-norm of $ \partial_t^j \Phi$}\label{subsec-L2pot}
The energy-dissipation equalities derived in Section \ref{sec-ED} show that using the equations we can control only the transformed Dirichlet norm of the velocity potential but not its $L^2$-norm. Although the latter is finite since we are dealing with a bounded domain, we cannot apply Poincaré inequality to control it because neither the potential vanish on the boundary $\partial \Omega_s$ nor it has zero mean over $\Omega_s$. Differently from the Stokes case studied in \cite{GuoTice2018}, here we do not have a condition on the tangential component of the potential on the boundary that gives an additional term in the basic dissipation and allows to obtain the desired control. Nevertheless, we are able to control the $L^2$-norm of the potential with the entire basic dissipation $\mathcal{D}_\parallel$ by estimating the mean of the potential in the fluid-domain  using the Dirichlet boundary condition at $\Gamma$ and resorting to a crucial homogeneous trace inequality. This type of trace inequality has been recently studied in \cite{LeoTice19} for infinite strip-like domains and in \cite{LanMin24} for the floating structures problem with a bounded or unbounded fluid domain. The authors showed that the range of the trace mapping defined on the homogeneous space $\dot{H}^1$  do not coincide in general with the homogeneous fractional space $\dot{H}^{1/2}$ but depends on whether the boundary is bounded or not. Since $\partial\Omega_s$ is bounded, the answer is positive. More precisely, let us consider the homogeneous fractional Sobolev space
\begin{align*}
\dot{H}^{s}(\Gamma)=\left\{ f\in L^1(\Gamma) \ | \ \|f\|_{\dot{H}^s(\Gamma)}<\infty\right\},
\end{align*}
where
\begin{align*}
    \|f\|_{\dot{H}^s(\Gamma)}^2&=\int_{\Gamma}\int_{\Gamma}\frac{|f(\hat{x})-f(\hat{y})|^2}{|\hat{x}-\hat{y}|^{1+2s}}d\sigma(\hat{y})d\sigma(\hat{x})\\[5pt]&\sim\int_{\mathcal{I}}\int_{\mathcal{I}}\frac{|f(x,h_s(x))-f(x',h_s(x'))|^2}{|x-x'|^{1+2s}}dx'dx= \|f(\cdot,h_s(\cdot))\|^2_{\dot{H}^{s}(\mathcal{I})}
\end{align*}
and $s\in (0,1)$. For functions $\Phi$ defined on $\Omega_s$, we use the notation $$\|\Phi\|_{\dot{H}^{s}(\Gamma)}=\|\Phi_{|_\Gamma}\|_{\dot{H}^{s}(\Gamma)},$$
with a slight abuse of notation. Here $\Phi_{|_{\Gamma}}$ denotes the trace of $\Phi$ on $\Gamma$.
 We then have the following continuity result for the trace mapping on $\Gamma$:
\begin{proposition}\label{prop-homotrace}
 There exists a constant $C>0$, depending only on $h_s$, such that
\begin{equation*}
\|\Phi\|_{\dot{H}^{1/2}(\Gamma)} \leq C \| \nabla \Phi\|_{L^2(\Omega_s)}.
\end{equation*}
In addition, the trace mapping is onto and admits a right inverse.
\end{proposition}
\begin{proof}
   The proof follows from \cite[Theorem 1]{LanMin24} and the smoothness of the stationary surface $h_s$. Indeed, in \cite{LanMin24} the authors derived the homogeneous trace inequality in the case when the boundary is a part of the horizontal line $\{y=0\}$. We then introduce the function $\widetilde{\Phi} (x,y)=\Phi(x, h_s(x)-y)$, which is defined on a suitable domain $\widetilde{\Omega}$ contained in the upper half-space. Applying the theorem at the boundary $\Gamma_h=\{x\in \mathcal{I},\  y=0\}$ implies that
   \begin{equation*}
\|\Phi\|_{\dot{H}^{1/2}(\Gamma)}=\|\widetilde{\Phi}\|_{\dot{H}^{1/2}(\Gamma_h)} \leq C \| \nabla \widetilde{\Phi}\|_{L^2(\widetilde{\Omega})} \leq C\left(\|h_s\|_{W^{1,\infty}(\mathcal{I})}\right) \| \nabla \Phi\|_{L^2(\Omega_s)},
   \end{equation*}
   where in the last inequality we have used the chain rule.
\end{proof}
\begin{remark}Although this result is analogous to its non-homogeneous counterpart, it ceases to hold when dealing with unbounded boundaries. Indeed, in this case, the range of the trace operator defined on $\dot{H}^1$ is the so-called \textit{screened} homogeneous Sobolev space $\ddot{H}^{1/2}$, which is strictly larger than $\dot{H}^{1/2}$. For instance, in the case when $\mathcal{I}=\mathbb{R}_+$, the semi-norm of   $f\in\ddot{H}^{1/2}(\Gamma)$ is given by
\begin{equation*}
 \|f\|^2_{\ddot{H}^{1/2}(\Gamma)}=\int_{\mathbb{R}_+}\int_{\mathbb{R}_+ \cap B_1(x)} \frac{|f(x, h_s(x))-f(x', h_s(x'))|^2}{|x-x'|^2} dx'dx,
\end{equation*}
where $B_1(x)$ is the one-dimensional screening  ball centered in $x$ with radius $1$. We refer the interested reader to \cite{LeoTice19, LanMin24} for the definition of homogeneous Sobolev spaces in a general framework and for their properties. We know from \cite[Theorem 1]{LanMin24} that the trace mapping from $\dot{H}^1(\Omega_s)$ to $\ddot{H}^{1/2}(\Gamma)$ is continuous and onto in the unbounded case. With such a result at hand, one might expect to extend our analysis to unbounded configurations, for instance when the fluid touches one vertical wall at $x=0$ and it is free to flow in the right half-space $x>0$. However, this is not possible because classical Sobolev embeddings, which we strongly use in the estimates of the nonlinear terms, cease to hold for these screened spaces and our analysis cannot be carried out. Indeed, there exists a function in $\ddot{H}^{1/2}(\Gamma)$ that does not belong to $L^p(\Gamma)$ for any $1\leq p\leq \infty$, see Theorem 3.11 and Remark 3.12 in \cite{LeoTice19}. As suggested by the authors, embeddings into other types of spaces, for instance weighted Lebesgue spaces, may hold but, up to now, this interesting question is open.
\end{remark}
We now show that the $L^2$-norms of the velocity potential and its time-derivatives are controlled by the basic dissipation.
\begin{proposition}\label{prop-controlL2}
Let $(\partial_t^j\eta, \partial_t^j\Phi)$ satisfy \eqref{DNpb-per}-\eqref{evo-contact-per}, \eqref{DNpb-per1}-\eqref{evo-contact-per1} and \eqref{DNpb-per2}-\eqref{evo-contact-per2}, respectively for $j=0,1,2$. If $\eta$ has zero mean over $\mathcal{I}$, then there exists a constant $C>0$, depending only on the domain $\Omega_s$, such that
\begin{equation}\label{PhiD}
\|\Phi\|^2_{L^2(\Omega_s)} + \|\partial_t \Phi\|^2_{L^2(\Omega_s)} + \|\partial_t^2\Phi\|^2_{L^2(\Omega_s)}\leq C \mathcal{D}_\parallel.
\end{equation} with $\mathcal{D}_\parallel$ defined in \eqref{basic-dis}.
\end{proposition}
\begin{proof}
First, using the Dirichlet condition in \eqref{DNpb-per} and \eqref{evo-contact-per}, we have that
\begin{equation*}
    \begin{aligned}
       \int_\mathcal{I} \Phi(x, h_s(x))dx &= -g\int_\mathcal{I} \eta(x)dx + \sigma \Big(\frac{\eta'}{(1+ (h'_s)^2)^{3/2}} + \mathcal{R}(h'_s, \eta') \Big)\Big|^{x=1}_{x=-1}\\[5pt]
       &= -g\int_\mathcal{I}\eta(x)dx - \partial_t \eta(-1)- \partial_t \eta(1).
    \end{aligned}
\end{equation*}Let us denote by $\overline{f}$ the mean of a one-dimensional function $f$ over $\mathcal{I}$. Then, the zero-mean assumption on $\eta$ implies that
\begin{equation}\label{est-meantrace}
    \big| \overline{\Phi(\cdot, h_s(\cdot))}\big|\leq \frac{1}{|\mathcal{I}|}\left(|\partial_t \eta(-1)| + |\partial_t \eta(1)|\right).
\end{equation}
 Note that both $\partial_t\eta$ and $\partial_t^2\eta$ have zero mean over $\mathcal{I}$ since $\eta$ has zero mean over $\mathcal{I}$. Then, we combine the Dirichlet conditions in \eqref{DNpb-per1} and \eqref{DNpb-per2}  with \eqref{evo-contact-per1} and \eqref{evo-contact-per2} to obtain that
\begin{equation}
\label{est-meantrace12}
\begin{aligned}
     &\big| \overline{\partial_t \Phi(\cdot, h_s(\cdot))}\big|\leq \frac{1}{|\mathcal{I}|} \left(|\partial_t^2 \eta(-1)| + |\partial_t^2 \eta(1)|\right),\\[5pt]&  \big| \overline{\partial_t^2\Phi(\cdot, h_s(\cdot))}\big|\leq  \frac{1}{|\mathcal{I}|}\left(|\partial_t^3 \eta(-1)| + |\partial_t^3 \eta(1)|\right).
     \end{aligned}
\end{equation}Since
for any $(x, y)\in \Omega_s$ we can write
\begin{equation*}\begin{aligned}
    \Phi(x,y)&= \Phi(x, h_s(x)) - \int_{y}^{h_s(x)}\partial_z \Phi (x,z) dz,
    \end{aligned}
\end{equation*}integrating the identity above over $\Omega_s$ and using the boundedness of $h_s -h_w$ yield that
\begin{equation*}
    \begin{aligned}
   \int_{\Omega_s}& \Phi(x,y)dxdy\\
   &=\int_\mathcal{I} \Phi(x, h_s(x)) (h_s - h_w)(x)dx - \int_\mathcal{I}\int_{h_w(x)}^{h_s(x)}\int_{y}^{h_s(x)} \partial_z \Phi (x,z) dz dy dx \\[5pt]
  &= \int_\mathcal{I} \left(\Phi(x, h_s(x)) - \overline{\Phi(\cdot, h_s(\cdot))}\right) (h_s -h_w)(x) dx   \\[5pt]
  & \quad + \overline{\Phi(\cdot,h_s(\cdot))} \int_\mathcal{I}(h_s -h_w)(x)dx
   - \int_\mathcal{I}\int_{h_w(x)}^{h_s(x)}\int_{y}^{h_s(x)} \partial_z \Phi (x,z) dz dy dx
   \\[5pt]
  & \leq C_{s,w}\Big(\big\|\Phi(\cdot, h_s(\cdot)) - \overline{\Phi(\cdot, h_s(\cdot))}\big\|_{L^2(\mathcal{I})}+  \big|\overline{\Phi(\cdot, h_s(\cdot))}\big| +\|\nabla \Phi\|_{L^2(\Omega_s)}\Big)
   \end{aligned}
\end{equation*}with $C_{s,w}>0$ depending only on $h_s-h_w$. Note that
\begin{equation*}
    \begin{aligned}
        \big\|\Phi(\cdot, h_s(\cdot)) - \overline{\Phi(\cdot, h_s(\cdot))}\big\|^2_{L^2(\mathcal{I})}&= \frac{1}{|\mathcal{I}|^2}\int_\mathcal{I}\Big|\int_\mathcal{I} \left(\Phi(x, h_s(x)) - \Phi(x', h_s(x'))\right)dx' \Big|^2 dx\\[5pt]
        \leq\frac{1}{|\mathcal{I}|}\int_\mathcal{I}\int_\mathcal{I}& |\Phi(x, h_s(x)) - \Phi(x', h_s(x')|^2 dx'dx\\[5pt]
             \leq C \int_\mathcal{I}\int_\mathcal{I}& \frac{|\Phi(x, h_s(x)) - \Phi(x', h_s(x')|^2}{|x-x'|^2} dx'dx= C\|\Phi\|^2_{\dot{H}^{1/2}(\Gamma)},
    \end{aligned}
\end{equation*}where in the last inequality we have used that, for any $x,x'\in \mathcal{I}$, there exists a constant $C>0$ such that $|x-x'|^{\alpha}\leq C$ with $\alpha>0$.
Combining the last two inequalities with \eqref{est-meantrace} and using the homogeneous trace inequality in Proposition \ref{prop-homotrace}, we obtain that
\begin{equation}\label{est-meanpot}\begin{aligned}
   \int_{\Omega_s}\Phi(x,y)dxdy &\leq C \left(\|\Phi\|_{\dot{H}^{1/2}(\Gamma)} + |\partial_t \eta(-1)|+ |\partial_t \eta(1)|+\|\nabla \Phi\|_{L^2(\Omega_s)}\right) \\[5pt]&\leq C \left( \|\nabla \Phi\|_{L^2(\Omega_s)} + |\partial_t \eta(-1)|+ |\partial_t \eta(1)|\right).\\[5pt]
     \end{aligned}
\end{equation} Finally, after denoting the mean of $\Phi$ in $\Omega_s$ by $\Phi_{\Omega_s}$, we use Poincaré-Wirtinger inequality and \eqref{est-meanpot} to get
\begin{equation*}
\begin{aligned}
    \|\Phi\|^2_{L^2(\Omega_s)}&= \|\Phi-\Phi_{\Omega_s} \|^2_{L^2(\Omega_s)} + \frac{1}{|\Omega_s|}\Big|\int_{\Omega_s}\Phi(x,y)dxdy\Big|^2
     \\[5pt]&\leq C\Big(\|\nabla \Phi\|^2_{L^2(\Omega_s)} + (\partial_t \eta)^2 (-1) + (\partial_t \eta)^2 (1)\Big)\leq C \mathcal{D}_\parallel.
\end{aligned}
\end{equation*}
 Arguing in the same way  using \eqref{est-meantrace12}, we derive the analogous bounds for the mean of $\partial_t \Phi$ and $\partial_t^2 \Phi$ in $\Omega_s$ and, consequently,
 \begin{align*}
     & \|\partial_t\Phi\|^2_{L^2(\Omega_s)}\leq C\left(\|\nabla \partial_t\Phi\|^2_{L^2(\Omega_s)} + (\partial_t^2 \eta)^2 (-1) + (\partial_t^2 \eta )^2(1)\right)\leq C \mathcal{D}_\parallel,\\[5pt]&
        \|\partial_t^2\Phi\|^2_{L^2(\Omega_s)}\leq C\left(\|\nabla \partial_t^2\Phi\|^2_{L^2(\Omega_s)} + (\partial_t^3 \eta )^2(-1) + (\partial_t^3 \eta)^2 (1)\right)\leq C \mathcal{D}_\parallel.
 \end{align*}
  This concludes the proof.
		\end{proof}

\subsection{Higher regularity for $\partial_t^j \eta$}
We focus now  on the elliptic problems in $\Gamma$ solved by $\partial_t^j \eta$ for $j=0,1, 2$. Our aim is to gain spatial regularity taking advantage of the ellipticity and using the control on the time derivatives.

Let us consider the abstract elliptic problem
	\begin{equation}\label{elliptic-per}
		\begin{aligned}
				-g \psi + \sigma \left(\frac{\psi'}{(1+(h_s')^2)^{3/2}}\right)' &=\Psi_{|_{\Gamma}}(\cdot,h_s(\cdot)) + g_3\quad \text{in} \quad \mathcal{I}, \\[5pt]
			\sigma\frac{\psi'}{(1 + (h_s')^2)^{3/2}}(\pm 1)&= g_4^\pm.
		\end{aligned}
	\end{equation}
It is clear that the operator associated with \eqref{elliptic-per} is uniformly elliptic. We have the following elliptic regularity result:
\begin{proposition}\label{prop-reg-ell}
Let $\Psi\in H^{1}(\Omega_s)$, $g_3\in H^s(\mathcal{I})$ with $-1\leq s\leq 1/2$ and $g_4^\pm\in \mathbb{R}$.
Then,  \eqref{elliptic-per} admits a unique solution $\psi\in H^{s+2}(\mathcal{I})$ and there exists a constant $C>0$ such that
\begin{equation*}\begin{aligned}
	\| \psi\|_{H^{s+2}(\mathcal{I})} \leq C \left( \|\Psi\|_{H^1(\Omega_s)}+ \|g_3\|_{H^{s}(\mathcal{I})} + |g_4^-|+ |g_4^+|\right).
	\end{aligned}
\end{equation*}
\end{proposition}
\begin{proof}
Since $\Psi\in H^1(\Omega_s)$, we have that $\Psi_{|_{\Gamma}}(\cdot,h_s(\cdot))$ $\in H^{1/2}(\mathcal{I})\subset H^{s}(\mathcal{I})$. Standard elliptic regularity theory then implies the existence and uniqueness of the solution $\psi\in H^{s+2}(\mathcal{I})$ to \eqref{elliptic-per}. Moreover, there exists a constant $C>0$ such that
\begin{equation*}\begin{aligned}
	\|\psi\|_{H^{s+2}(\mathcal{I})}&\leq C \left(\|\Psi_{|_{\Gamma}}(\cdot,h_s(\cdot)) + g_3\|_{H^{s}(\mathcal{I})} + |g_4^-|+ |g_4^+|\right)\\[5pt]& \leq C \left(\|\Psi\|_{H^1(\Omega_s)}+ \|g_3\|_{H^{s}(\mathcal{I})} + |g_4^-|+ |g_4^+|\right).
		\end{aligned}
\end{equation*}
\end{proof}
We want to apply Proposition \ref{prop-reg-ell} to the elliptic problems in $\Gamma$ solved by $\partial_t^j \eta$ for $j=0,1,2$. The elliptic problem solved by $\partial_t^j\eta$ is formed by the second line in \eqref{DNpb-per} with \eqref{evo-contact-per} for $j=0$, the second line in \eqref{DNpb-per1} with \eqref{evo-contact-per1} for $j=1$ and the second line in \eqref{DNpb-per2} with \eqref{evo-contact-per2} for $j=2$. They can be all written in the form \eqref{elliptic-per} with
$\psi= \partial_t^j \eta$, $\Psi=\partial_t^j \Phi$ and
\begin{equation}\label{g3g4-j}\begin{aligned}
   g_3= -\sigma (\partial_t^j\mathcal{R}(h'_s, \eta'))', \qquad g_4^\pm = \mp \partial_t^{j+1} \eta ( \pm 1 ) -\sigma \partial_t^j\mathcal{R}(h'_s, \eta')( \pm 1)
    \end{aligned}
\end{equation}for $j=0,1,2$.
We then state the following regularity result:

\begin{proposition}\label{prop-controletaj}
  Let $\mathcal{R}$ be given by \eqref{R} and $\mathcal{E}$, $\mathcal{\mathcal{D}}$ by \eqref{impro-ene}-\eqref{tot-diss}. Then, $(g_3,g_4^\pm)$ in \eqref{g3g4-j} belongs to $H^{1/2}(\mathcal{I})\times \mathbb{R}$ for $j=0,1$ and to $H^{-1/2+}(\mathcal{I})\times \mathbb{R}$ for $j=2$. Moreover,  $\psi=\partial_t ^j \eta$ uniquely solves \eqref{elliptic-per} with $\Psi=\partial_t^j \Phi\in H^1(\Omega_s)$ and data $(g_3,g_4)$ for $j=0,1,2$. For $j=0,1$, $ \partial_t^j \eta \in H^{5/2}(\mathcal{I})$ and  $\partial_t^2 \eta \in H^{3/2+}(\mathcal{I})$. In particular, there exists a constant $C>0$ such that
\begin{equation*}
    \|\eta\|^2_{H^{5/2}(\mathcal{I})}+  \|\partial_t\eta\|^2_{H^{5/2}(\mathcal{I})}+  \|\partial_t^2\eta\|^2_{H^{3/2+}(\mathcal{I})}\leq C (\mathcal{D}_\parallel + \big(1+\|\eta\|^2_{H^{3/2+}(\mathcal{I})}\big)\mathcal{E} \mathcal{D}).
\end{equation*}
    \end{proposition}
    \begin{proof}
  We first consider the data \eqref{g3g4-j} in the case $j=0$.
    By means of the chain rule, we have that
        \begin{equation}\label{g3-j=0}
            g_3 = -\sigma \partial_{z_1} \mathcal{R}(h'_s, \eta') h_s'' -\sigma \partial_{z_2} \mathcal{R}(h'_s, \eta') \eta''.
        \end{equation} Using the critical product estimate in Lemma \ref{lemma-prodest} then yields
        \begin{equation}\label{g3-bound}\begin{aligned}
           & \|g_3\|^2_{H^{1/2}(\mathcal{I})} \\[5pt]&\leq C ( \|\partial_{z_1} \mathcal{R}(h'_s, \eta')\|^2_{H^1(\mathcal{I})}\|h''_s\|^2_{H^{1/2}(\mathcal{I})} + \|\partial_{z_2} \mathcal{R}(h'_s, \eta')\|^2_{H^{1/2+}(\mathcal{I})}\|\eta''\|^2_{H^{1/2}(\mathcal{I})})
            \end{aligned}
        \end{equation}for some constant $C>0$.
    Applying again the chain rule and using the uniform bounds in Lemma \ref{lemma-R}, thanks to the continuous embedding $H^{1/2+}(\mathcal{I}) \subset L^{\infty}(\mathcal{I})$ we obtain that
    \begin{equation}\label{g3-1}\begin{aligned}
        \|\partial_{z_1} \mathcal{R}(h'_s, \eta')\|^2_{H^1(\mathcal{I})}&\leq C ( \|(\eta')^2\|^2_{L^2(\mathcal{I})} + \|\eta'\eta''\|^2_{L^2(\mathcal{I})})\\
        &\leq C\|\eta\|^2_{H^{3/2+}(\mathcal{I})}\|\eta\|^2_{H^{2}(\mathcal{I})} \leq C \mathcal{E}\mathcal{D}.
        \end{aligned}
    \end{equation}
 On the other hand, we infer from Lemma \ref{lemma-R} that 
\begin{align}\label{g3-2}
   \|\partial_{z_2}\mathcal{R}(h_s', \eta')\|^2_{H^{1/2+}(\mathcal{I})}
   \leq  C \|\eta\|^2_{H^{3/2+}(\mathcal{I})}\leq C\mathcal{E}. 
\end{align}
Combining \eqref{g3-1}-\eqref{g3-2} with \eqref{g3-bound} and using the smoothness of $h_s$, we find that $g_3\in H^{1/2}(\mathcal{I})$ and
 \begin{equation*}\begin{aligned}
            \|g_3\|^2_{H^{1/2}(\mathcal{I})}\leq C \mathcal{E}\mathcal{D}.
            \end{aligned}
\end{equation*}
Concerning the boundary data $$g_4^\pm=\mp \partial_t \eta ( \pm 1) - \sigma \mathcal{R}(h'_s, \eta')( \pm 1), $$
we use again Lemma \ref{lemma-R} and the continuous embedding $H^{3/2+}(\mathcal{I})\subset C^1(\overline{\mathcal{I}})$  to obtain that $g_4 \in \mathbb{R}$ and
\begin{equation*}
    (g_4^\pm)^2\leq C ((\partial_t \eta)^2 ( \pm 1) + (\eta')^4( \pm 1)) \leq C ( \mathcal{D}_\parallel +\|\eta\|^4_{H^{3/2+}(\mathcal{I})})\leq C\left(\mathcal{D}_\parallel +  \mathcal{E}\mathcal{D}\right)
\end{equation*}
In the case $j=1$, by applying the chain rule we write
\begin{align*}
  g_3&=  \left(\partial_t \mathcal{R}(h'_s, \eta')\right)'= \left(\partial_{z_2}\mathcal{R}(h'_s,\eta')\partial_t \eta'\right)' \\[5pt]&=\left(\partial_{z_1}\partial_{z_2} \mathcal{R}(h'_s, \eta') h''_s + \partial_{z_2}^2 \mathcal{R}(h'_s, \eta') \eta''\right) \partial_t \eta' + \partial_{z_2}\mathcal{R}(h'_s, \eta') \partial_t\eta''=I + II +III.
\end{align*}
The third term having the same structure as the second term in the right-hand side of \eqref{g3-j=0}, we repeat the previous analysis and find that 
\begin{equation}\label{III}\| III\|^2_{H^{1/2}(\mathcal{I})}\leq  C\|\partial_{z_2}\mathcal{R}(h_s', \eta')\|^2_{H^{1/2+}(\mathcal{I})} \|\partial_t \eta'' \|^2_{H^{1/2}(\mathcal{I})}\leq C\mathcal{E}\mathcal{D}.\end{equation}
Using Lemmas \ref{lemma-prodest} and \ref{lemma-R}  together with the algebra property of $H^{1/2+}(\mathcal{I})$, we also infer that
\begin{equation}\label{I}
\begin{aligned}
    \|I\|^2_{H^{1/2}(\mathcal{I})} &\leq C \|\partial_{z_1}\partial_{z_2}\mathcal{R}(h'_s, \eta')\|^2_{H^{1/2+}(\mathcal{I})}\|h''_s\|^2_{H^{1/2+}(\mathcal{I})}\|\partial_t \eta'\|^2_{H^{1/2}(\mathcal{I})}\\[5pt]&\leq C \|\eta\|^2_{H^{3/2+}(\mathcal{I})}\|\partial_t\eta\|^2_{H^{3/2}(\mathcal{I})}\leq C \mathcal{E}\mathcal{D}
\end{aligned}
\end{equation}
and 
\begin{equation}\label{II}
\begin{aligned}
   & \|II\|^2_{H^{1/2}(\mathcal{I})}\leq C \|\partial_{z_2}^2 \mathcal{R}(h_s', \eta')\|^2_{H^{1/2+}(\mathcal{I})}\|\partial_t \eta'\|^2_{H^{1/2+}(\mathcal{I})}   
    \|\eta''\|^2_{H^{1/2}(\mathcal{I})}\\[5pt]&\leq C (1 + \|\eta\|^2_{H^{3/2+}(\mathcal{I})})\|\partial_t \eta\|^2_{H^{3/2+}(\mathcal{I})}
    \|\eta\|^2_{H^{5/2}(\mathcal{I})}\leq C(1+ \|\eta\|^2_{H^{3/2+}(\mathcal{I})})\mathcal{E}\mathcal{D}.
\end{aligned}
\end{equation}
 Gathering \eqref{III}-\eqref{II} together, we then obtain that $g_3\in H^{1/2}(\mathcal{I})$ and 
$$\|g_3\|^2_{H^{1/2}(\mathcal{I})}\leq C (1+ \|\eta\|^2_{H^{3/2+}(\mathcal{I})})\mathcal{E}\mathcal{D}.$$
Concerning the boundary data  
\begin{align*}g_4^\pm &= \mp \partial_t^2 \eta(\pm 1) - \sigma \partial_t \mathcal{R}(h'_s, \eta')( \pm 1)\\[5pt]&=\mp \partial_t^2 \eta(\pm 1) - \sigma\partial_{z_2}\mathcal{R}(h'_s,\eta')( \pm 1)\partial_t \eta'( \pm 1),\end{align*}
arguing as in the case $j=0$ yields that $g_4 \in \mathbb{R}$ and
\begin{equation*}
    (g_4^\pm)^2\leq C \left((\partial^2_t \eta)^2 ( \pm 1) + (\eta')^2( \pm 1)(\partial_t\eta')^2( \pm 1)\right) \leq C \left( \mathcal{D}_\parallel + \mathcal{E}\mathcal{D}\right).
\end{equation*}
In the case $j=2$, the data \eqref{g3g4-j} read
  $$g_3=\left(\partial_t^2 \mathcal{R}(h'_s,\eta')\right)'= \left(\partial^2_{z_2}\mathcal{R}(h'_s, \eta')(\partial_t \eta')^2 + \partial_{z_2}\mathcal{R}(h'_s, \eta')\partial_t^2 \eta'\right)'$$
  and
$$g_4^\pm = \mp \partial_t^3 \eta( \pm1 ) -\sigma \partial_t^2\mathcal{R}(h'_s,\eta')(\pm 1).\\[5pt]$$
  We claim that $\partial_t^2 \mathcal{R}(h'_s,\eta')\in H^{1/2+}(\mathcal{I})$, which in turn implies that $g_3\in H^{-1/2+}(\mathcal{I})$, and that $g_4^\pm\in \mathbb{R}$. Indeed, after combining Lemma \ref{lemma-R} with the algebra property of $H^{1/2+}(\mathcal{I})$, it follows that 
  \begin{align*}
      \|g_3\|^2_{H^{-1/2+}(\mathcal{I})}&\leq C\|\partial_t^2 \mathcal{R}(h'_s,\eta')\|^2_{H^{1/2+}(\mathcal{I})}\\[5pt]&\leq C ( (1+ \|\eta\|^2_{H^{3/2+}(\mathcal{I})})\|\partial_t \eta\|^4_{H^{3/2+}(\mathcal{I})} +\|\eta\|^2_{H^{3/2+}(\mathcal{I})}\| \partial_t^2\eta\|^2_{H^{3/2+}(\mathcal{I})})
      \\[5pt]&\leq C  (1+ \|\eta\|^2_{H^{3/2+}(\mathcal{I})}) \mathcal{E}\mathcal{D}
  \end{align*}and, thanks to the continuous embedding $H^{1/2+}(\mathcal{I})\subset C(\overline{\mathcal{I}})$, 
  \begin{equation*}
      (g_4^\pm)^2\leq C((\partial_t^3\eta )^2( \pm1 ) + \|\partial_t^2\mathcal{R}(h'_s, \eta')\|^2_{H^{1/2+}(\mathcal{I})})\! \leq C(\mathcal{D}_\parallel + (1+\|\eta\|^2_{H^{3/2+}(\mathcal{I})})\mathcal{E}\mathcal{D}).
  \end{equation*}
  In addition, we know from Proposition \ref{prop-controlL2} that
\begin{equation}
   \|\partial_t^j \Phi\|^2_{H^{1}(\Omega_s)}\leq C \mathcal{D}_\parallel \quad \text{for} \quad j=0,1,2.
\end{equation}
Therefore, by applying Proposition \ref{prop-reg-ell} with $s=1/2$ and $s=-1/2+$, we obtain that $\partial_t^j \eta\in H^{5/2}(\mathcal{I})$ for $j=0,1$ and $\partial_t^2 \eta \in H^{3/2+}(\mathcal{I})$. In particular, there exists a constant $C>0$ such that
\begin{equation*}
    \|\eta\|^2_{H^{5/2}(\mathcal{I})}+  \|\partial_t\eta\|^2_{H^{5/2}(\mathcal{I})}+  \|\partial_t^2\eta\|^2_{H^{3/2+}(\mathcal{I})}\leq C (\mathcal{D}_\parallel + \big(1+\|\eta\|^2_{H^{3/2+}(\mathcal{I})}\big)\mathcal{E} \mathcal{D}).
\end{equation*}
\end{proof}

		\section{Elliptic estimates in $\Omega_s$}\label{sec-ellest}
		
		We know from Section \ref{subsec-L2pot} that the improved dissipation provides a control of $\partial_t^j \Phi$ in $H^1(\Omega_s)$ for $j=0,1,2$. To estimate the nonlinear terms in Section \ref{sec-NL}, the controlled regularity is not sufficient to derive bounds having the desired structure $\mathcal{E}^\theta\mathcal{D}$ with $\theta>0$. However, we manage to gain the necessary regularity by leveraging elliptic estimates for both the potential and its time derivative in the same spirit as \cite{GuoTice2018}. Here, the fluid domain $\Omega_s$ is only Lipschitz due to the presence of corners and standard elliptic regularity theory \cite{AgmDouNir1964} cannot be applied.
      Instead, we rely on the theory in \cite[Section 4]{Grisvard1985} for elliptic boundary value problems in non-smooth domains. The key tool will be Theorem \ref{theo-reg-neu} where solvability in $H^2$ of the Neumann problem in convex curvilinear plane domains is derived. This result was already used in \cite{Poyferre2019,LanMin24, MingWang2024}. Here we detail the proof for the sake of completeness.\\
      To this end, given a bounded  domain $D\subset \mathbb{R}^2$, we introduce the quotient space $\mathring{H}^2(D)=H^2(D)/\mathbb{R}$ endowed with the norm
\begin{equation*}
    \|\Psi\|_{\mathring{H}^2(D)}= \|\nabla \Psi\|_{H^1(D)}.
\end{equation*}

      \subsection{Neumann problem in a polygon}
First, we look at a simple model of this problem to understand how the opening angle at the corners affects the formation and behavior of singularities in the solution.
Let $\mathcal{P}$ be a plane bounded domain having a polygonal boundary with angles $\omega_j$ for $j=1,\dots, M$ ($M\geq 3$). We denote by $\Gamma_j$ each side of the polygon and by $n_j$ the unit outward normal vector to $\Gamma_j$. We address the case of an exact polygon, that is, $n_j$ is not parallel to $n_{j+1}$ for all $j$, excluding the case of mixed problems on a flat boundary.
We consider the Neumann problem
\begin{equation}\label{Neu-pb}
\begin{cases}
	\Delta \Psi = f  \quad & \text{in} \quad \mathcal{P}\\[5pt]
\nabla \Psi \cdot n_j = g_j  & \text{on} \quad \Gamma_j.
	\end{cases}
		\end{equation}
In the next proposition, we state a second-order regularity assertion:
\begin{proposition}\label{prop-solva-Neu}
Let $f \in L^2(\mathcal{P})$ and $g_j\in H^{1/2}(\Gamma_j)$ for $j=1,\dots,M$ satisfy the compatibility condition
\begin{equation}\label{comp-pol}\int_{\mathcal{P}} f = \sum_{j=1}^M \int_{\Gamma_j}g_j.\end{equation}
Assume that $\omega_j \in (0, \pi)$ for all $j$. Then, \eqref{Neu-pb} admits a solution $\Psi\in H^2(\mathcal{P})$, unique up to an additive constant, and there exists a constant $C>0$ such that
\begin{equation}\label{reg-est}
    \|\Psi\|_{\mathring{H}^2(\mathcal{P})} \leq C \Big(\|f\|_{L^2(\mathcal{G})} + \sum_{j=1}^M \|g_j\|_{H^{1/2}(\Gamma_j)}\Big).
\end{equation}
	\end{proposition}
    \begin{proof}
    The existence and uniqueness up to an additive constant of a variational solution $\Psi\in H^1(\mathcal{P})$ to \eqref{Neu-pb} are standard and hold regardless of the sizes of $\omega_j$.
    The $H^2$-regularity  directly follows from Theorem 4.4.3.7 and Corollary 4.4.3.8 in \cite{Grisvard1985} for the solvability of \eqref{Neu-pb} in $H^2(\mathcal{P}) $ with homogeneous and non-homogeneous boundary conditions, respectively. Indeed, it is shown there that the variational solution admits the decomposition
    \begin{equation}\label{decomposition}
        \Psi = \Psi_{\rm reg} + \sum_{\substack{j=1,\dots,M\\[2pt] -1 < \lambda_{j,m}<0}} C_{j,m}S(\lambda_{j,m}).
    \end{equation} Above, $\Psi_{\rm reg}\in H^2(\mathcal{P})$ is the regular part and $S(\lambda_{j,m})\in H^1(\mathcal{P})\setminus H^2(\mathcal{P})$ are singular functions that depend on the eigenvalues $\lambda_{j,m}$ of an operator associated with the Laplace operator, also called \emph{pencil operator} \cite{KozMazRos97}, that captures its behavior at the corners. These eigenvalues depend on the boundary conditions considered for the Laplace equation and, in the Neumann case, they read
    \begin{equation}\label{eigen-Neu}
    \lambda_{j,m}= m\frac{\pi}{\omega_j}, \quad m\in \mathbb{Z}, \ j=1,\dots,M.
    \end{equation}
    Since $\omega_j\in (0, \pi)$ for all $j$, the summation in \eqref{decomposition} disappears and $\Psi\in H^2(\mathcal{P})$.
Once we have established the existence of a $H^2$-regular solution, unique up to an additive constant, provided that the compatibility condition \eqref{comp-pol} holds, we know that the image of $H^2(\mathcal{P})$ through the operator $\Delta_n=(\Delta, n_1\cdot \nabla, \dots, n_N\cdot \nabla)$ is a closed subspace of codimension $1$ in $L^2(\mathcal{P})\times \prod_{j=1}^N H^{1/2}(\Gamma_j)$ and that the mapping
$$\Delta_n:\quad  \mathring{H}^2(\mathcal{P})  \rightarrow  \mathrm{Im} \Delta_n $$ is a continuous isomorphism of Banach spaces. Then, by applying the bounded inverse theorem, we obtain the regularity estimate \eqref{reg-est}.
    \end{proof}
    Proposition \ref{prop-solva-Neu} implies that the solution to the Neumann problem belongs to $H^2$ in convex polygons. In general, the same regularity is not achieved when we consider mixed Dirichlet-Neumann boundary conditions, even if the data are regular. In this case, the eigenvalues of the pencil operator read
\begin{equation*}
    \lambda_{j,m} = (2m+1)\frac{\pi}{2\omega_j}, \quad m\in\mathbb{Z}, \ j=1,\dots, M.
\end{equation*}Then, the summation  in \eqref{decomposition} disappears only if $\omega_j\in (0, \pi/2)$, while singularities arise if $\omega_j\in [\pi/2, \pi)$. Therefore, the solution is $H^2$-regular only for polygons with acute angles, otherwise it is only $H^1$-regular. In this sense, we infer that the mixed Dirichlet-Neumann problem is more singular than the Neumann problem.

We would like to apply the previous regularity assertion for the transformed Neumann problems that the potential and its time derivative solve. The gain of $H^2$-regularity will be sufficient to derive bounds for the nonlinear terms in Section \ref{sec-NL} necessary to close the scheme of a priori estimates. However,
Proposition \ref{prop-solva-Neu} deals only with plane domains with polygonal boundary; this is not the case in our configuration, as $\Gamma$ is curvilinear close to the contact points. For this reason, in the next section we study the Neumann problem in the curvilinear domain $\Omega_s$.

\subsection{Neumann problem in $\Omega_s$}
Here we argue as in \cite[Section 5.2]{Grisvard1985} and \cite[Section 5]{GuoTice2018}, where solvability for the Dirichlet problem and the Stokes problem in curvilinear domains were studied, respectively.
Our goal is to  derive a second-order regularity estimate for the Neumann problem
\begin{equation}\label{Neu-pb-stat}
\begin{cases}
	\Delta \Psi = f  \quad & \text{in} \quad \Omega_s,\\[5pt]
\nabla \Psi  \cdot \frac{N_{h_s}}{|N_{h_s}|} = g_1  & \text{on} \quad\Gamma,\\[5pt]
\nabla \Psi  \cdot n = g_2  & \text{on} \quad \Gamma_w.
	\end{cases}
		\end{equation}
In order to obtain such an estimate, we introduce the following map, which transforms the neighborhood of the contact points of $\Omega_s$ into a neighborhood of the vertex $(0,0)$ in the polygon $\mathcal{G}$ previously introduced. Due to the symmetry of $\Omega_s$ with respect to the $y$-axis close to the surface $h_s$, we do this only for the point $(-1,h_s(-1))$.
Let us consider the plane cone of opening $\omega\in (0,\pi)$
      \begin{equation}\label{cone}
			\mathcal{K}_\omega= \left\{ (x,y)\in \mathbb{R}^2 \ | \  0<r<\infty, \ -\frac{\pi}{2}<\theta<-\frac{\pi}{2}+\omega  \right\},
		\end{equation}
where $(r,\theta)$ are standard polar variables,
with sides
\begin{equation}\label{sides}\Gamma_-=\left\{r>0, \ \theta= -\frac{\pi}{2} + \omega\right\} \quad  \text{and} \quad \Gamma_+=\left\{ r>0, \ \theta= -\frac{\pi}{2} \right\}.\end{equation}

We define the angle $\omega\in (0, \pi)$ formed by the stationary surface at the corners of $\Omega_s$ through the relation
\begin{equation}\label{angle-corner}
    -\textrm{cotan}{(\omega)}= h'_s(-1).
\end{equation}Note that the assumption $\llbracket \gamma \rrbracket/\sigma \in (-1,1)$ in Proposition \ref{sta-exiuni} guarantees that $h'_s(-1)$ is finite, which, in turn, implies that $\omega\in (0,\pi)$.

\begin{proposition}\label{transformation} Let $\omega$ be as in \eqref{angle-corner}, $\mathcal{K}_\omega$, $\Gamma_\pm$ be as in \eqref{cone}-\eqref{sides}
and \begin{equation}\label{r}0<r<\min \left\{1, \frac{h_s(-1)-\max_{\overline{\mathcal{I}}} h_w}{2}\right\}.\end{equation} There exists a smooth diffeomorphism $\mathcal{T}: \mathcal{K}_\omega\to \mathcal{T}(\mathcal{K}_\omega)\subset \mathbb{R}^2$ satisfying the following properties:
\begin{enumerate}
\item $\mathcal{T}$ is smooth up to $\overline{\mathcal{K}}_\omega$.
\item $\Gamma_-=\mathcal{T}^{-1}\{ (x,y)\in \R^2\,:\,x=-1, \, y<h_s(-1)\}$;
\item $\mathcal{T}^{-1}(\Gamma\cap B_r((-1,h_s(-1)))\subseteq \Gamma_+\cap B_R(0,0)$ and $\mathcal{T}^{-1}(\Omega_s\cap B_r((-1,h_s(-1))))\subseteq \mathcal{K}_\omega\cap B_R(0,0)$ for $R=\sqrt{2r^2+2r^4\|h_s\|^2_{C^2( \overline{\mathcal{I}})}}$;
\item $\mathfrak{U}=(\nabla \mathcal{T})^{-T}$ is smooth on $\overline{\mathcal{K}}_\omega$ and all its derivatives are bounded;
\item $\det(\mathfrak{U})=1$ and $\partial_x \mathfrak{U}_{i1} +\partial_y \mathfrak{U}_{i2}=0$, for $i=1,2$;
\item $\mathfrak{U}^T\mathfrak{U}$ is uniformly elliptic on $\overline{\mathcal{K}}_\omega$;
\item $\mathfrak{U}^T\mathfrak{U} - \mathbb{I}$, $\mathfrak{U} - \mathbb{I}$ are supported in $\overline{\mathcal{K}}_\omega \cap S_{2r}$, with $S_{2r}=\{(x,y)\in \mathbb{R}^2 \ |\  0\leq x\leq 2r\}$, and satisfy the bounds
     \begin{align*}&\|\mathfrak{U}^T\mathfrak{U} - \mathbb{I} \|_{L^\infty(\mathcal{K}_\omega\cap S_{2r}(0,0))}\leq C r,\\[5pt]
 &\|\mathfrak{U}- \mathbb{I}\|_{C^{0, \gamma}(\partial \mathcal{K}_\omega\cap S_{2r}(0,0))}\leq Cr\quad \text{with} \quad 1/2 <\gamma\leq 1 .\end{align*}
\end{enumerate}
\end{proposition}
\begin{proof}We introduce the diffeomorphism $\mathcal{T}: \mathcal{K}_\omega \rightarrow \mathcal{T}(\mathcal{K}_\omega)\subset \mathbb{R}^2$ given by
\begin{equation}
    \mathcal{T}(x,y)= \big(x-1, y + h_s(-1) + \chi(x) (h_s(-1+x)- h_s(-1) - h'_s(-1)x) \big).
\end{equation}
where $\chi: [0, +\infty)\rightarrow \mathbb{R}$ is a cut-off function with $0\leq \chi\leq 1$ such that $\chi=1$ in $[0,r]$ and $\chi= 0$ in $[0, +\infty) \setminus [0,2r]$, where $r$ satisfies \eqref{r}.
The associated matrix-valued function $\mathfrak{U}= (\nabla \mathcal{T})^{-T}: \mathcal{K}_\omega\rightarrow \mathbb{R}^{2\times 2}$ reads
\begin{equation}\label{Ufrak}
    \mathfrak{U}(x,y)= \left(\begin{matrix}
        1 & -\big(\chi(x) ( h_s(-1+x) -h_s(-1) -h'_s(-1)x)\big)' \\[5pt]
        0& 1
    \end{matrix}\right).
\end{equation}
Properties (1)-(6) are direct or can be shown as in Proposition 5.5 in \cite{GuoTice2018}. Due to the behavior of $\chi$, we have that $\mathfrak{U}^T\mathfrak{U} - \mathbb{I}$ and $\mathfrak{U} - \mathbb{I}$ are supported in $\overline{\mathcal{K}}_\omega \cap S_{2r}$. Moreover,  $\mathbb{I} - \mathfrak{U}$ vanishes on $\Gamma_-$ and, thanks to the Taylor expansions
$$\|h_s(-1+\cdot)-h_s(-1) - h'_s(-1)\cdot\|_{L^\infty(0,2r)}\leq C r^2,$$
$$\|h'_s(-1+\cdot) - h'_s(-1)\|_{L^\infty(0,2r)}\leq C r,$$
we find property (7).
\end{proof}

We first prove an a priori estimate in $H^2$ for a general boundary value problem in a bounded open domain $\mathcal{G}\subset \mathbb{R}^2$ whose boundary is a curvilinear convex polygon. For simplicity, we address the case where $\mathcal{G}$ has only one corner at $(x_0, l(x_0))$, that is, $\overline{\partial\mathcal{G}}\setminus {(x_0,l(x_0))}$ is smooth. Moreover, we assume that $\partial \mathcal{G}$ is the union of the vertical line $\{x=x_0\}$ and the surface parametrized by a smooth function $l(x)$ in a neighborhood of $(x_0, l(x_0))$. We consider the boundary value problem
\begin{equation}\label{Neu-pb-A}
\begin{cases}
	\nabla \cdot ( \mathcal{A}^T\mathcal{A}\nabla \Psi) = f  \quad & \text{in} \quad \mathcal{G},\\[5pt]
\mathcal{A}\nabla \Psi  \cdot \frac{\mathcal{A} N }{|\mathcal{A} N|} = g  & \text{on} \quad\partial \mathcal{G},
	\end{cases}
		\end{equation}
        where $N$ is the outward normal vector to $\partial \mathcal{G}$ and  $
        \mathcal{A}$ is a smooth matrix-valued function in $\overline{\mathcal{G}}$ with all its derivatives bounded such that $\mathcal{A}(x_0, l(x_0))=\mathbb{I}$ and
        $\mathcal{A}^T\mathcal{A}$ is uniformly elliptic in $\overline{\mathcal{G}}$.

\begin{proposition}\label{prop-aprioriH2}Let $f\in L^2(\mathcal{G})$ and $g\in H^{1/2}(\partial \mathcal{G})$. There exists a constant $C>0$ such that a solution $\Psi\in H^2(\mathcal{G})$ to \eqref{Neu-pb-A} satisfies
\begin{equation} \label{reg-est-neu}
	\| \Psi\|_{\mathring{H}^2(\mathcal{G})} \leq C( \|f\|_{L^2(\mathcal{G})}+ \|g\|_{H^{1/2}(\partial \mathcal{G})}).
	\end{equation}
\end{proposition}
\begin{proof}
The proof is an adaptation of Lemma 5.2.3 in \cite{Grisvard1985} to the Neumann case with non-homogeneous boundary data.
Let us consider the smooth diffeomorphism $\mathcal{O} : \mathcal{K}_\omega \rightarrow \mathcal{O}(\mathcal{K}_\omega)\subset \mathbb{R}^2$  given by
\begin{equation*}
    \mathcal{O}(x,y)= (x-x_0, y + l(x_0)+ \eta(x) (l(x_0 +x)-l(x_0)- l'(x_0)x) )
\end{equation*}where $\eta: [0, +\infty)\rightarrow \mathbb{R}$ is a cut-off function with $0\leq \eta\leq 1$ such that $\eta=1$ in $[0,r]$ and $\eta= 0$ in $[0, +\infty) \setminus [0,2r]$,  with $r$ sufficiently small. Note that $\mathcal{O}$ is analogous to the diffeomorphism $\mathcal{T}$
introduced in Proposition \ref{transformation}, hence the corresponding properties also hold for $\mathcal{O}$.
From properties (2)-(3), we know that $\mathcal{O}^{-1}(\partial \mathcal{G}\cap B_r(x_0,l(x_0)) )$ is the union of two straight segments contained in $\partial \mathcal{K}_\omega$ and $\mathcal{O}^{-1}(\mathcal{G} \cap B_r(x_0, l(x_0)))$ coincides with $\mathcal{K}_\omega$, with $-\mathrm{cotan}(\omega) =l'(x_0)$, near $(0,0)$. Then, we introduce a smooth cut-off function $\chi:\mathcal{G}\rightarrow \mathbb{R} $  with $0\leq \chi \leq 1 $ such that $\chi=1$ in $B_{r/4}(x_0,l(x_0))$, $\chi=0$ in $\mathcal{G} \setminus B_{r/2}(x_0,l(x_0))$ and $\partial_n \chi=0$ on $\partial \mathcal{G}$. Note that the support of $\chi$ is strictly contained in $\mathcal{G} \cap B_{r/2}(x_0,l(x_0))$.
Then,  $\Psi^*=(\chi\Psi)\circ \mathcal{O}$ solves the Neumann problem
\begin{equation}\label{apriori-neupb}
\begin{cases}
    \Delta \Psi^*= f^* + \nabla \cdot \left((\mathbb{I}-\mathfrak{O}^T\mathfrak{O})\nabla\Psi^*\right) \quad & \text{in} \quad \mathcal{P},\\[5pt]
\nabla \Psi^*  \cdot n = g^* + (\mathcal{I} - \mathfrak{O})\nabla \Psi^*\cdot \tfrac{\mathfrak{O} \widetilde{N}}{| \widetilde{N}|}  + \nabla\Psi^*\cdot    (\mathcal{I} - \mathfrak{O})n& \text{on} \quad \partial \mathcal{P}\cap \partial\mathcal{K}_\omega,\\[5pt]
\nabla \Psi^*  \cdot n = 0& \text{elsewhere on $\partial \mathcal{P}$},
\end{cases}
\end{equation}
where $\mathfrak{O}=(\nabla\mathcal{O})^{-T}$,   $\widetilde{N}$ is the outward normal vector to $\partial \mathcal{P} \cap \partial \mathcal{K}_\omega$ and $\mathcal{P}$ is a convex polygon such that $\overline{\mathcal{P}}\subset \mathcal{O}^{-1}(\mathcal{G} \cap B_r(-1, h_s(-1)))$, $\overline{\mathcal{P}}$ contains the support of $\Psi^*$ and $\partial\mathcal{P}$ coincides with $\partial\mathcal{K}_\omega$ near $(0,0)$. Above, $$f^*= \left(\nabla \cdot (\mathcal{A}^T\mathcal{A}\nabla(\chi \Psi)) + \nabla\cdot ((\mathbb{I}- \mathcal{A}^T\mathcal{A})\nabla (\chi \Psi)) \right)\circ \mathcal{O}$$ and $$ g^*=\tfrac{|\mathfrak{O}\widetilde{N}|}{|\widetilde{N}|} \left(\mathcal{A}\nabla (\chi \Psi)\cdot \tfrac{\mathcal{A}N}{|\mathcal{A}N|} + (\mathbb{I}-\mathcal{A})\nabla(\chi \Psi)\cdot\tfrac{\mathcal{A}N}{|\mathcal{A}N|} + \nabla (\chi \Psi)\cdot \left(\tfrac{N}{|N|} - \tfrac{\mathcal{A}N}{|\mathcal{A}N|}\right)  \right)\circ\mathcal{O}.$$
Thanks to the smoothness of $\mathcal{A}$ and $\mathfrak{O}$, the one-dimensional product estimate
$$\|h_1h_2\|_{H^{1/2}}\leq C \|h_1\|_{C^{0,\gamma}}\|h_2\|_{H^{1/2}} \quad \text{for $1/2<\gamma\leq 1$}$$ and a trace theorem, the right-hand sides  in \eqref{apriori-neupb} belong to $L^2(\mathcal{P})$ and $H^{1/2}(\partial \mathcal{P}\cap \partial \mathcal{K}_\omega)$, respectively.
Since all the angles of $\mathcal{P}$ are less than $\pi$, we use \eqref{reg-est} and obtain
\begin{equation*}\begin{aligned}&\|\Psi^*\|_{\mathring{H}^2(\mathcal{P})}\leq C \big( \|f^*\|_{L^2(\mathcal{P})} + \|g^*\|_{H^{1/2}(\partial\mathcal{P}\cap \mathcal{K}_\omega)} + \|\nabla \Psi^*\|_{L^2(\mathcal{P})} \\[5pt]
&+\|\mathbb{I}- \mathfrak{O}^T\mathfrak{O}\|_{L^\infty(\mathcal{P}\cap S_{2r})} \|\nabla^2 \Psi^*\|_{L^2(\mathcal{P})} + \|\mathbb{I}- \mathfrak{O}\|_{C^{0,\gamma}(\partial \mathcal{P} \cap \partial \mathcal{K}_\omega \cap S_{2r} )}\|\nabla \Psi^*\|_{H^1(\mathcal{P})}
\big),
\end{aligned}
\end{equation*}where $C>0$ is a constant independent of $r$. Note that we have used the fact that $ \mathfrak{O}^T\mathfrak{O}- \mathbb{I}$ and $\mathfrak{O} - \mathbb{I}$ are supported in $S_{2r}$. Due to property (7) in Proposition \ref{transformation} choosing $r$ sufficiently small yields that
\begin{equation*}\begin{aligned}&\|\Psi^*\|_{\mathring{H}^2(\mathcal{P})}\leq C \big( \|f^*\|_{L^2(\mathcal{P})} + \|g^*\|_{H^{1/2}(\partial\mathcal{P}\cap \mathcal{K}_\omega)} + \|\nabla \Psi^*\|_{L^2(\mathcal{P})}
\big)
\end{aligned}
\end{equation*}
and changing coordinates back to $\mathcal{G}$ implies that
\begin{equation*}\begin{aligned}\|\chi\Psi\|_{\mathring{H}^2(\mathcal{G})}&\leq C \big(\|\nabla \cdot (\mathcal{A}^T\mathcal{A}\nabla(\chi \Psi))\|_{L^2(\mathcal{G})} +\|\mathcal{A}\nabla (\chi \Psi)\cdot \tfrac{\mathcal{A}N}{|\mathcal{A}N|} \|_{H^{1/2}(\partial \mathcal{G})}\\[5pt]&
+ \|\mathbb{I}-\mathcal{A}^T\mathcal{A}\|_{L^\infty(\mathrm{supp}\chi\cap \mathcal{G})}\|\nabla^2 (\chi \Psi)\|_{L^2(\mathcal{G})} + \|\nabla (\chi \Psi)\|_{L^2(\mathcal{G})}\\[5pt]&
 + \big(\|\mathbb{I}-\mathcal{A}\|_{C^{0,\gamma}(\mathrm{supp}\chi \cap \partial \mathcal{G})} + \|\tfrac{N}{|N|}-\tfrac{\mathcal{A}N}{|\mathcal{A}N|}\|_{C^{0,\gamma}(\mathrm{supp}\chi\cap \partial \mathcal{G})}\big) \|\nabla(\chi \Psi)\|_{H^1(\mathcal{G})}
\big).
\end{aligned}
\end{equation*} Exploiting the fact that $\mathcal{A}$ is smooth and $\mathcal{A}(x_0, l(x_0))= \mathbb{I}$, we choose a possibly smaller $r$ and obtain
\begin{align*}
\|\chi\Psi\|_{\mathring{H}^2(\mathcal{G})}\leq C \big(& \|\nabla \cdot (\mathcal{A}^T\mathcal{A}\nabla(\chi \Psi) )\|_{L^2(\mathcal{G})} \\
&+\| \mathcal{A}\nabla (\chi \Psi)\cdot \tfrac{\mathcal{A}N}{|\mathcal{A} N|}\|_{H^{1/2}(\partial\mathcal{G})}
 + \|\nabla (\chi \Psi)\|_{L^2(\mathcal{G})}\big).
\end{align*}
Since $\partial \mathcal{G}$ is smooth away from the corners, we know from standard elliptic theory that $(1-\chi)\Psi\in H^2(U)$ where $U$ is a $C^2$-domain contained in $\mathcal{G}$ whose boundary coincides with $\partial \mathcal{G}$ away from the region $\{\chi=1\}$. Moreover, we have that
\begin{align*}
\|(1-\chi)\Psi\|_{\mathring{H}^2(\mathcal{G})}
\leq C \big( &\|\nabla \cdot (\mathcal{A}^T\mathcal{A}\nabla((1-\chi) \Psi) )\|_{L^2(\mathcal{G})} \\[5pt]&+ \|\mathcal{A}\nabla ((1-\chi) \Psi)\cdot \tfrac{\mathcal{A}N}{|\mathcal{A} N|} \|_{H^{1/2}(\partial \mathcal{G})}
 + \|\nabla ((1-\chi) \Psi)\|_{L^2(\mathcal{G})}\big).
\end{align*}
Gathering the previous estimates together yields that
\begin{equation}\label{estH2-H1}
\|\Psi\|_{\mathring{H}^2(\mathcal{G})}\leq C \big( \| f \|_{L^2(\mathcal{G})} + \| g\|_{H^{1/2}(\partial\mathcal{G})}
 + \|\nabla \Psi\|_{L^2(\mathcal{G})}\big)
\end{equation}
and combining \eqref{estH2-H1} with a standard variational estimate gives \eqref{reg-est-neu}.
\end{proof}

We now state the solvability result for the Neumann problem in $H^2(\Omega_s)$.
\begin{theorem}\label{theo-reg-neu}
Let $f\in L^2(\Omega_s)$, $g_1\in H^{1/2}(\Gamma)$, $g_2\in H^{1/2}(\Gamma_w)$ satisfy
\begin{equation*}
    \int_{\Omega_s} f = \int_{\Gamma} g_1 + \int_{\Gamma_w} g_2.
\end{equation*} Then, \eqref{Neu-pb-stat} admits a solution $\Psi \in H^2(\Omega_s)$ to \eqref{Neu-pb-stat}, unique up to an additive constant. Moreover, it satisfies the regularity estimate \eqref{reg-est-neu}.
\end{theorem}
\begin{proof}
The proof is an adaptation of Theorem 5.2.2 in \cite{Grisvard1985} to the Neumann case. The existence and uniqueness up to an additive constant of a variational solution $\Psi\in H^1(\Omega_s)$ to \eqref{Neu-pb-stat} are standard. Since $\partial \Omega_s$ is smooth away from the corners, we use standard elliptic theory \cite{AgmDouNir1964} to infer that $\Psi\in H^2(U)$ where $U$ is a $C^2$-domain contained in $\Omega_s$ whose boundary coincides with $\partial \Omega_s$ except close to $(\pm1, h_s(\pm 1))$.\\
We now address the regularity issue near the corners of $\Omega_s$. Let $\chi:\Omega_s\rightarrow \mathbb{R} $ be a smooth cut-off function with $0\leq \chi \leq 1 $ such that $\chi=1$ in $B_{r/4}(-1,h_s(-1))$, $\chi=0$ in $\Omega_s \setminus B_{r/2}(-1,h_s(-1))$ with radius $r$ satisfying \eqref{r} and $\partial_n \chi=0$ on $\partial \Omega_s$. Note that the support of $\chi$ is strictly contained in $\Omega_s \cap B_{r/2}(-1, h_s(-1))$.
Then,  $\widetilde{\Psi}=\chi\Psi\in H^1(W)$ satisfies
\begin{equation*}
\begin{cases}
    \Delta \widetilde{\Psi}= \widetilde{f} \quad & \text{in} \quad W,\\[5pt]
\nabla \widetilde{\Psi}  \cdot n = \widetilde{g}  & \text{on} \quad \partial W ,
\end{cases}
\end{equation*}
with $\widetilde{f}=\Delta (\chi \Psi)\in L^2(W)$ and $\widetilde{g}\in H^{1/2}(\partial W)$ such that
\begin{align*}\widetilde{g}= \chi g_1 \quad \text{on}& \quad \partial W \cap \Gamma, \quad \widetilde{g}= \chi g_2  \quad \text{on} \quad \partial W \cap \Gamma_w, \quad
\widetilde{g}=0 \quad \text{elsewhere on} \quad \partial W,
\end{align*}

 where $$\Omega_s \cap B_{r/2}(-1,h_s(-1)) \subset W \subset \Omega_s \cap B_r(-1,h_s(-1))$$ is a curvilinear polygon of class $C^2$ with only one angular point such that its boundary coincides with $\partial \Omega_s$ near $(-1, h_s(-1))$. We now consider the smooth diffeomorphism $\mathcal{T}^{-1}$ introduced in Proposition \ref{transformation}. From properties (2)-(3), we know that $\mathcal{T}^{-1}(\Gamma\cap W ) \subset \Gamma_+$,
$\mathcal{T}^{-1}(\Gamma_w\cap W ) \subset \Gamma_-$ with $\Gamma_\pm$ as in \eqref{sides} and $\mathcal{T}^{-1}(W)$ coincides with $\mathcal{K}_\omega$, with $\omega$ as in \eqref{angle-corner}, near $(0,0)$. Then, $\widehat{\Psi}= \widetilde{\Psi} \circ \mathcal{T} \in H^{1}(\mathcal{T}^{-1}(W))$ satisfies
\begin{equation}\label{bv-Ufrak}
\begin{cases}
    \nabla \cdot (\mathfrak{U}^T\mathfrak{U}\nabla \widehat{\Psi}) = \widehat{f}\quad & \text{in} \quad \mathcal{T}^{-1}(W),\\[5pt]
    \mathfrak{U}\nabla \widehat{\Psi} \cdot \frac{\mathfrak{U}N}{|\mathfrak{U}N|}= \widehat{g}& \text{on} \quad \partial\mathcal{T}^{-1}( W),
\end{cases}
\end{equation} where $\mathfrak{U}=(\nabla \mathcal{T})^{-T}$ is given by \eqref{Ufrak}, $N$ is the outward normal vector to $\partial \mathcal{T}^{-1}(W)$,
$\widehat{f}= \widetilde{f} \circ \mathcal{T} \in L^2(\mathcal{T}^{-1}(W))$
and $\widehat{g}= \widetilde{g}\circ \mathcal{T}\in H^{1/2}(\partial \mathcal{T}^{-1}(W))$.
Denoting by $L_\mathfrak{U}$ the operator associated with \eqref{bv-Ufrak}
 $$L_\mathfrak{U}=\left(\nabla\cdot (\mathfrak{U}^T\mathfrak{U}\nabla), \tfrac{\mathfrak{U}N}{|\mathfrak{U}N|} \cdot \mathfrak{U}\nabla\right): \ H^2(\mathcal{T}^{-1}(W)) \rightarrow \ L^2(\mathcal{T}^{-1}(W))\times H^{1/2}(\partial \mathcal{T}^{-1}(W)),$$ the properties of $\mathfrak{U}$ imply that the principal part of $L_\mathfrak{U}$ with coefficients frozen at $(0,0)$ is exactly the Neumann operator $$\quad \Delta_N= \left(\Delta, n\cdot \nabla\right) : \ H^2(\mathcal{T}^{-1}(W)) \rightarrow \ L^2(\mathcal{T}^{-1}(W))\times H^{1/2}(\partial \mathcal{T}^{-1}(W)).$$
 We now focus on the boundary value problem
\begin{equation*}
    \begin{cases}
    \Delta \widehat{\Psi}= \widehat{f}\quad & \text{in} \quad \mathcal{T}^{-1}(W),\\[5pt]
    \nabla \widehat{\Psi} \cdot n = \widehat{g}\quad & \text{on} \quad \partial \mathcal{T}^{-1}(W).
    \end{cases}
\end{equation*}
Arguing as before, we introduce a
smooth cut-off function $\zeta$, such that $\zeta=1$ near $(0,0)$ and $\partial_n \zeta=0$ on $\partial \mathcal{T}^{-1}(W)$, and choose its support small enough in such a way that $(1-\zeta)\widehat{\Psi}\in H^2(\mathcal{T}^{-1}(W))$ and $\zeta \widehat{\Psi}\in H^1(\mathcal{P})$  solves
\begin{equation}\label{Neu-polygon}
\begin{cases}
    \Delta (\zeta \widehat{\Psi})= \zeta \widehat{f} + \Delta \zeta \widehat{\Psi} + 2\nabla\zeta \cdot \nabla \widehat{\Psi}\quad & \text{in} \quad \mathcal{P},\\[5pt]
    \nabla (\zeta\widehat{\Psi}) \cdot n = \zeta\widehat{g} \quad & \text{on} \quad \partial \mathcal{P},
    \end{cases}
\end{equation}where $\mathcal{P}$ is a convex polygon that coincides with $\partial \mathcal{T}^{-1}(W)$ near $(0,0)$. Clearly, the data in \eqref{Neu-polygon} belong to $L^2(\mathcal{P})$ and $H^{1/2}(\mathcal{P})$, respectively. Applying Proposition \ref{prop-solva-Neu}, we have that $\zeta \widehat{\Psi}\in H^2(\mathcal{P})$ provided that the data satisfies the corresponding compatibility condition. We conclude that $\widehat{\Psi}\in H^{2}(\mathcal{T}^{-1}(W))$ provided
that \begin{equation}\label{compacond}
\int_{\mathcal{T}^{-1}(W)}\widehat{f} = \int_{\partial \mathcal{T}^{-1}(W)}\widehat{g},\end{equation}
and it is unique up to an additive constant. Thus, the kernel of $\Delta_N$ is one-dimensional, the image of $H^2(\mathcal{T}^{-1}(W))$ through $\Delta_N$ is a closed subspace of codimension $1$ and its index is
\begin{equation}\label{indLapNe}
    \mathrm{ind}\ \Delta_N = \mathrm{dim} \ \mathrm{Ker} \Delta_N - \mathrm{codim \ Im }\Delta_N = 1 - 1=0.
\end{equation}
Returning to the boundary-value problem \eqref{bv-Ufrak},
let us introduce the family of operators from $H^2(\mathcal{T}^{-1}(W))$ to $L^2(\mathcal{T}^{-1}(W)) \times H^{1/2}(\partial\mathcal{T}^{-1}(W))$ $$A(\tau)= \tau  L_\mathfrak{U} + (1-\tau)\Delta_N, \quad \tau\in [0,1],$$
that continuously depends on $\tau$. Proposition \ref{prop-aprioriH2} yields that $\mathrm{dim \ Ker} A(\tau)=1$ and that $A(\tau)$ is a semi-Fredholm operator for all $\tau\in [0,1]$.  Hence we known from Theorem IV.5.17 in \cite{Kato1995} that its index is independent of $\tau$ and   \eqref{indLapNe} yields
$$ \mathrm{ind}\ L_\mathfrak{U}= \mathrm{ind}\ A(1)= \mathrm{ind}\ A(0)=\mathrm{ind} \ \Delta_N =0 .$$
Since $\mathrm{dim \ Ker} L_{\mathfrak{U}}=1$ due to Proposition \ref{prop-aprioriH2}, we obtain that the image of $H^2(\mathcal{T}^{-1}(W))$ through $L_{\mathfrak{U}}$ is a closed subspace of codimension $1$ in $L^2(\mathcal{T}^{-1}(W))\times H^{1/2}(\partial\mathcal{T}^{-1}(W))$ formed by the pairs $(\widehat{f},\widehat{g})$ that satisfy \eqref{compacond}.
Changing coordinates back to $\Omega_s$, this implies that $\widetilde{\Psi}\in H^2(W)$ provided that $$\int_{W} \widetilde{f} = \int_{\partial W} \widetilde{g}.$$
Due to the symmetry of $\Omega_s$ with respect to the $y$-axis close to the surface $h_s$, after introducing a reflection, the same analysis can be developed in a neighborhood of the contact point $(1, h_s(1))$. Gathering together with the regularity assertion away from the corners, we conclude that the image of $H^2(\Omega_s)$ through the operator $\left(\Delta, \frac{N_{h_s}}{|N_{h_s}|}\cdot \nabla , n\cdot \nabla\right)$ is a closed subspace of codimension $1$ in $L^2(\Omega_s)\times H^{1/2}(\Gamma)\times H^{1/2}(\Gamma_w)$ formed by the triples $(f, g_1,g_2)$ that satisfy \begin{equation}\label{compa-triple}\int_{\Omega_s} f = \int_{\Gamma}g_1 + \int_{\Gamma_w} g_2.\end{equation}In other words, there exists a solution $\Psi\in H^2(\Omega_s)$ to \eqref{Neu-pb-stat} provided that \eqref{compa-triple} holds.
Moreover, by Proposition \ref{prop-aprioriH2} it is unique up to an additive constant and there exists a constant $C>0$ such that
\begin{equation*}
\|\Psi\|_{\mathring{H}^2(\Omega_s)} \leq C \big(\|f\|_{L^2(\Omega_s)} +\|g_1\|_{H^{1/2}(\Gamma)}+ \|g_2\|_{H^{1/2}(\Gamma_w)}\big).
\end{equation*}
\end{proof}

	\subsection{Transformed Neumann problem in $\Omega_s$}We now consider the transformed Neumann problem
\begin{equation}\label{trans-neumann}
	\begin{cases}\begin{aligned}
		\nabla \cdot (A_\eta\nabla \Psi)=f\quad &\mbox{in} \quad \Omega_s,\\[5pt]
		\Sigma_\eta \nabla \Psi \cdot N_h =g_1 \quad &\mbox{on}\quad \Gamma,\\[5pt]
		\Sigma_\eta \nabla \Psi \cdot n=g_2    \quad &\mbox{on} \quad \Gamma_w,
	\end{aligned}\end{cases}\\[5pt]
\end{equation} with $\Sigma_\eta$ and $A_\eta$ given by \eqref{Sigmaeta}-\eqref{Aeta}.  Throughout the rest of this section, we assume that $\eta\in H^{3/2+}(\mathcal{I})$. The following second-order regularity assertion holds:

\begin{proposition}\label{prop-reg-trans}
	Let $f\in L^2(\Omega_s)$, $g_1\in H^{1/2}(\Gamma)$, $g_2\in H^{1/2}(\Gamma_w)$ satisfy
    \begin{equation*}
        \int_{\Omega_s}f = \int_{\Gamma} g_1 + \int_{\Gamma_w}\det(J_\eta)g_2
    \end{equation*}and $\omega\in (0, \pi)$ be as in \eqref{angle-corner}. There exists $\alpha>0$ such that, for $\|\eta\|_{H^{3/2+}(\mathcal{I})}< \alpha$, \eqref{trans-neumann} admits a unique solution $\Psi\in H^2(\Omega_s)$ up to an additive constant. Moreover, there exists a constant $C>0$ such that
	\begin{equation}\label{reg-est-transneu}
	\| \Psi\|_{\mathring{H}^2(\Omega_s)} \leq C\big( \|f\|_{L^2(\Omega_s)}+ \|g_1\|_{H^{1/2}(\Gamma)}+ \|g_2\|_{H^{1/2}(\Gamma_w)}\big).
	\end{equation}
	
\end{proposition}
\begin{proof}
Let us define the mapping \begin{equation*}\begin{aligned}
	\mathcal{L} \ : \ \mathring{H}^2(\Omega_s) \  &\rightarrow \ L^2(\Omega_s) \times H^{1/2}(\Gamma)\times H^{1/2}(\Gamma_w)\\[5pt]
	  \ \Psi \quad & \mapsto \quad \left(\Delta \Psi, \nabla \Psi\cdot \tfrac{N_{h_s}}{|N_{h_s}|}, \nabla \Psi\cdot n\right).
\end{aligned}
\end{equation*}	Then, \eqref{trans-neumann} can be recast as the Neumann problem
	\begin{equation}\label{trans-neumann-comp}
		\mathcal{L}(\Psi) = (F(\Psi),G_1(\Psi),G_2(\Psi))
	\end{equation}
with
\begin{equation*}
\begin{aligned}
&F(\Psi)= f+ \nabla \cdot\left((\mathbb{I}- A_\eta)\nabla \Psi\right),\\[5pt] &G_1(\Psi)=
\tfrac{1}{|N_{h_s}|}\left(
g_1 + (\mathbb{I}-\Sigma_\eta)\nabla \Psi\cdot N_h + \partial_x \Psi \ \eta'
\right),
\\[5pt]&
G_2(\Psi)= g_2 + (\mathbb{I}-\Sigma_\eta)\nabla \Psi\cdot n.\\[5pt]
\end{aligned}
\end{equation*}
For $\|\eta\|_{H^{3/2+}(\mathcal{I})}<\alpha $ with $\alpha$ as in Lemma \ref{lemma-diffeo}, we know from Lemma \ref{lemma-prodest} and the regularizing property \eqref{regular-prop} that
$$\Psi\in \mathring{H}^2(\Omega_s) \quad \Rightarrow \quad \left( F(\Psi), G_1(\Psi), G_2(\Psi)\right)\in L^2(\Omega_s) \times H^{1/2}(\Gamma)\times H^{1/2}(\Gamma_w).$$  Since $\omega\in(0, \pi)$, we apply Theorem \ref{theo-reg-neu} to \eqref{trans-neumann-comp} and obtain that $\mathcal{L} $ is an isomorphism. Hence, \eqref{trans-neumann-comp} is equivalent to
\begin{equation}\label{fixed-point-eq}
\Psi = \mathcal{L}^{-1}\left(F(\Psi),G_1(\Psi),G_2(\Psi)\right).
\end{equation}
Moreover, Lemma \ref{lemma-prodest} and a trace inequality yields the existence of a polynomial $P$ with non-negative coefficients and $P(0)=0$ such that
\begin{equation}\label{est-FG1G2}
	\begin{aligned}
&	\|F(\Psi)\|_{L^2(\Omega_s)} + \|G_1(\Psi)\|_{H^{1/2}(\Gamma)} + \|G_2(\Psi)\|_{H^{1/2}(\Gamma_w)}\\[5pt]
	&\leq \|f\|_{L^2(\Omega_s)}+ C\|g_1\|_{H^{1/2}(\Gamma)}+ \|g_2\|_{H^{1/2}(\Gamma_w)} +  P (\|\eta\|_{H^{3/2+}(\mathcal{I})}) \|\Psi\|_{\mathring{H}^2(\Omega_s)}\end{aligned}
\end{equation} and
\begin{equation}\label{est-contraFG1G2}
	\begin{aligned}
&	\|F(\Psi_1) - F(\Psi_2)\|_{L^2(\Omega_s)} + \|G_1(\Psi_1) - G_1(\Psi_2)\|_{H^{1/2}(\Gamma)} \\[5pt]
	&+ \|G_2(\Psi_1)- G_2(\Psi_2)\|_{H^{1/2}(\Gamma_w)}\leq  P (\|\eta\|_{H^{3/2+}(\mathcal{I})}) \|\Psi_1-\Psi_2\|_{\mathring{H}^2(\Omega_s)}.\\[5pt]
		\end{aligned}
\end{equation}Due to \eqref{est-FG1G2}-\eqref{est-contraFG1G2}, choosing a possibly smaller $\alpha$ implies that the mapping $$\Psi \mapsto \mathcal{L}^{-1}\left(F(\Psi),G_1(\Psi),G_2(\Psi)\right)$$ is a contraction. Then, \eqref{fixed-point-eq} admits a unique fixed point in $\mathring{H}^2(\Omega_s)$ or, equivalently, \eqref{trans-neumann} admits a unique solution $\Psi\in H^2(\Omega_s)$ up to an additive constant. In addition, combining \eqref{fixed-point-eq} with \eqref{est-FG1G2} yields the regularity estimate \eqref{reg-est-transneu}.
\end{proof}

\subsection{Elliptic estimates for $\Phi$ and $\partial_t \Phi$}
We want to apply Proposition \ref{prop-reg-trans} to the elliptic problems in $\Omega_s$ for $\Phi$ and $\partial_t \Phi$, respectively. Recall that the elliptic problem solved by $ \Phi$ is given by \eqref{evo-eq-per}, the first and the third lines in \eqref{DNpb-per}, namely,
\begin{equation}\label{NNpb-in}
	\begin{aligned}
		\nabla \cdot (A_{\eta}\nabla \Phi)=0 \quad &\mbox{in} \quad \Omega_s,\\[5pt]
		\Sigma_{\eta} \nabla \Phi \cdot N_h=\partial_t \eta \quad &\mbox{on}\quad \Gamma,\\[5pt]
		\Sigma_{\eta} \nabla \Phi \cdot n=0    \quad &\mbox{on} \quad \Gamma_w.
	\end{aligned}
\end{equation}
 We directly obtain the following second-order estimate:
\begin{proposition}\label{H2Phi}
Let $\omega\in (0, \pi)$ be as in \eqref{angle-corner} and assume that $\|\eta\|_{H^{3/2+}(\mathcal{I})}< \alpha$ with $\alpha$ as in Proposition \ref{prop-reg-trans}. Then, the solution $\Phi$ to \eqref{NNpb-in} belongs to $H^2(\Omega_s)$ and there exists a constant $C>0$ such that
\begin{equation*}
    \| \Phi\|^2_{\mathring{H}^2(\Omega_s)}\leq C \mathcal{E}_\parallel.
\end{equation*}
\end{proposition}
\begin{proof}
    Combining the definition of $\mathcal{E}_\parallel$ in \eqref{basic-en} with Proposition \ref{prop-reg-trans}, it follows that $\Phi\in H^2(\Omega_s)$ and there exists a constant $C>0$ such that
    \begin{equation*}
        \| \Phi\|^2_{\mathring{H}^2(\Omega_s)}\leq C \|\partial_t \eta\|^2_{H^{1/2}(\mathcal{I})}\leq C \mathcal{E}_\parallel.
    \end{equation*}\end{proof}

Recall that the elliptic problem for $\partial_t\Phi$ is given by \eqref{evo-eq-per1}, the first and the third lines in \eqref{DNpb-per1}, namely,
\begin{equation}\label{NNpb-in1}
	\begin{aligned}
		\nabla \cdot (A_{\eta}\nabla \partial_t \Phi)={F^1_2} \quad &\mbox{in} \quad \Omega_s,\\[5pt]
		\Sigma_{\eta}\nabla\partial_t\Phi \cdot N_{h}= \partial_t^2\eta  - {F^1_1}\quad &\mbox{on} \quad \Gamma,\\[5pt]
		\Sigma_{\eta} \nabla \partial_t\Phi \cdot n=F^1_3 \quad &\mbox{on} \quad \Gamma_w,
	\end{aligned}
\end{equation}
with $F^1_i$ as in \eqref{F1}. Then, we derive the following second-order estimate:
\begin{proposition}\label{H2PhiD}
Let $\omega\in (0, \pi)$ be as in \eqref{angle-corner} and assume that $\|\eta\|_{H^{3/2+}(\mathcal{I})}< \alpha$ with $\alpha$ as in Proposition \ref{prop-reg-trans}. Then, the solution $\partial_t \Phi$ to \eqref{NNpb-in1} belongs to $H^2(\Omega_s)$ and there exists a constant $C>0$ such that
\begin{equation*}
    \|\partial_t \Phi\|^2_{\mathring{H}^2(\Omega_s)}\leq C \left(\mathcal{D}_\parallel + \mathcal{E} \mathcal{D}\right).
\end{equation*}
\end{proposition}
\begin{proof}

On the one hand, combining the product estimate \eqref{prod-estH1} with the regularizing property \eqref{regular-prop}, the fact that $\|\eta\|_{H^{3/2+}(\mathcal{I})}<\alpha$, and Proposition \ref{H2Phi} yields
    \begin{equation*}
        \begin{aligned}
        \|F_2^1\|^2_{L^2(\Omega_s)}&= \|\nabla\cdot (\partial_t A_\eta \nabla\Phi)\|_{L^2(\Omega_s)} \leq \|\partial_t A_\eta \nabla \Phi\|^2_{H^1(\Omega_s)}\\[5pt]&\leq C\|\partial_t A_\eta\|^2_{H^{1+}(\Omega_s)}\|\nabla \Phi\|^2_{H^1(\Omega_s)}
        \leq C \|\partial_t \eta^\dag\|^2_{H^{2+}(\Omega_s)}\| \Phi\|^2_{\mathring{H}^2(\Omega_s)}\\[5pt]&\leq C \|\partial_t \eta\|^2_{H^{3/2+}(\mathcal{I})}\| \Phi\|^2_{\mathring{H}^2(\Omega_s)}\leq C \mathcal{E}\mathcal{D}.
        \end{aligned}
    \end{equation*}
 On the other hand, arguing as before but using the product estimate \eqref{prod-estH12} instead of \eqref{prod-estH1}  gives the estimates
\begin{equation*}
    \begin{aligned}
      \|F^1_3\|^2_{H^{1/2}(\Gamma_w)} &=\|\partial_t \Sigma_\eta \nabla \Phi \cdot n\|^2_{H^{1/2}(\Gamma_w)}\leq C\|\partial_t \Sigma_\eta\|^2_{H^{1/2+}(\Gamma_w)} \|\nabla \Phi\|^2_{H^{1/2}(\Gamma_w)}\\[5pt]& \leq C\|\partial_t \eta^\dag\|^2_{H^{2+}(\Omega_s)}\| \nabla \Phi\|^2_{H^1(\Omega_s)}
      \leq C\|\partial_t \eta\|^2_{H^{3/2+}(\mathcal{I})}\| \Phi\|^2_{\mathring{H}^2(\Omega_s)}\leq C \mathcal{E}\mathcal{D}
    \end{aligned}
\end{equation*}
and
\begin{equation*}
    \begin{aligned}
        &\|F^1_1\|^2_{H^{1/2}(\Gamma)}= \|\partial_t (\Sigma_\eta^T N_h)\cdot \nabla \Phi\|_{H^{1/2}(\Gamma)} \\[5pt]&\leq C\big(\|\partial_t \Sigma_\eta N_h\|^2_{H^{1/2+}(\Gamma)} + \| \Sigma_\eta \partial_t N_h\|^2_{H^{1/2+}(\Gamma)}\big) \|\nabla \Phi \|^2_{H^{1/2}(\Gamma)}
        \\[5pt]&\leq C  \big(\|\partial_t\eta^\dag\|^2_{H^{2+}(\Omega_s)} (1+ \|\eta\|^2_{H^{3/2+}(\mathcal{I})})  + \|\eta^\dag\|^2_{H^{2+}(\Omega_s)}\|\partial_t \eta\|^2_{H^{3/2+}(\mathcal{I})}\big) \|\nabla \Phi\|^2_{H^1(\Omega_s)}  \\[5pt]&\leq C \big( 1+ \|\eta\|^2_{H^{3/2+}(\mathcal{I})}\big)
        \|\partial_t \eta \|^2_{H^{3/2+}(\mathcal{I})} \|\Phi\|^2_{\mathring{H}^2(\Omega_s)}\leq C \mathcal{E} \mathcal{D}.
    \end{aligned}
\end{equation*}
Furthermore, Proposition \ref{prop-controletaj} combined with the assumption on the $H^{3/2+}$-norm of $\eta$ yields the bound
$$\|\partial_t^2 \eta\|^2_{H^{3/2+}(\mathcal{I})}\leq C \left(\mathcal{D}_\parallel + \mathcal{E}\mathcal{D}\right).$$
 Thus, applying Proposition \ref{prop-reg-trans}, the solution $\partial_t \Phi$ belongs to $H^2(\Omega_s)$ and there exists a constant $C>0$ such that
\begin{align*}
    \|\partial_t \Phi\|^2_{\mathring{H}^2(\Omega_s)}\leq C\big( \|F^1_2\|^2_{L^2(\Omega_s)} + \|F^1_3\|^2_{H^{1/2}(\Gamma_w)} + \|\partial_t^2\eta- F^1_1\|^2_{H^{1/2}(\Gamma)}  \big)\leq C (\mathcal{D}_\parallel + \mathcal{E}\mathcal{D}).
\end{align*}

\end{proof}

		\section{Terms in the energy-dissipation equalities}\label{sec-NL}
	With the bounds derived in Sections \ref{sec-adddiss} and \ref{sec-ellest} at hand, we now estimate one by one the nonlinear terms appearing on the right-hand side of the basic energy-dissipation equations in Section \ref{sec-ED}. For the sake of brevity, we will assume throughout the section that $\eta$ has zero mean over $\mathcal{I}$ and
  \begin{equation}\label{bound-eta-alpha}\|\eta\|_{H^{3/2+}(\mathcal{I})}< \alpha \end{equation} with $\alpha$ as in Proposition \ref{prop-reg-trans}.
		\begin{proposition}\label{S1-bound} Let $\mathcal{S}_1$ be given by \eqref{Sj}. Then, there exists a constant $C>0$ such that
			\begin{equation*}
				|\mathcal{S}_1| \leq C \sqrt{\mathcal{E}}\mathcal{D}.
			\end{equation*}
		\end{proposition}
		
		\begin{proof} Recall from \eqref{Sj} that
				\begin{equation}\label{S1}
				\begin{aligned}
				\	\mathcal{S}_1= - \int_{\Omega_s}F_2^1 \partial_t \Phi  -\int_{\Gamma} F_1^1 \partial_t \Phi+ \int_{\Gamma_w} \det(J_\eta)F_3^1\partial_t\Phi
					+ \int_{\mathcal{I}} \sigma\mathcal{F}_1(h'_s,\eta')dx
				\end{aligned}
			\end{equation} with $F^1_j$ as in \eqref{F1} for $j=1,2,3$ and $\mathcal{F}_1$ as in \eqref{Fj}.
			 We proceed by estimating each term.  Using Hölder inequality, the explicit expression \eqref{Aeta} of $A_\eta$  and the continuous embedding $H^1(\Omega_s)\subset L^4(\Omega_s)$ yield that
			\begin{equation*}
				\begin{aligned}
					\int_{\Omega_s}& F_2^1 \partial_t \Phi  = \int_{\Omega_s} \nabla \cdot (\partial_t A_\eta\nabla \Phi) \partial_t \Phi  \\[5pt]&\leq \left(\| \nabla \partial_t  A_\eta \|_{L^2(\Omega_s)}\|\nabla \Phi\|_{L^4(\Omega_s)} + \|\partial_t A_\eta\|_{L^4(\Omega_s)}\|\nabla^2 \Phi\|_{L^2(\Omega_s)} \right) \|\partial_t \Phi\|_{L^4(\Omega_s)}\\[5pt]&\leq C \|   \partial_t\eta^\dag\|_{H^2(\Omega_s)}\|\nabla \Phi\|_{H^1(\Omega_s)}\|\partial_t \Phi\|_{H^1(\Omega_s)}	.\end{aligned}
				\end{equation*}
				Using the regularizing property \eqref{regular-prop} for $s=2$ and the control provided by Proposition \ref{H2Phi}, we find that
						\begin{equation*}
						\begin{aligned}
				\int_{\Omega_s} F_2^1 \partial_t \Phi   \leq C  \|\partial_t \eta \|_{{H}^{3/2}(\mathcal{I})}\| \Phi\|_{\mathring{H}^2(\Omega_s)}\|\partial_t \Phi\|_{H^1(\Omega_s)} \leq C \sqrt{\mathcal{E}}\mathcal{D}.
				\end{aligned}
			\end{equation*}
			Let us write the second term in \eqref{S1} as the sum of two terms, namely,
			\begin{equation}\label{est-intF11}
				\begin{aligned}
					\int_{\Gamma} F_1^1 \partial_t \Phi  &= \int_{\Gamma} \partial_t (\Sigma^T_\eta N_h)\cdot \nabla \Phi \partial_t \Phi  \\&=\int_{\Gamma} \partial_t \Sigma_\eta^T N_h \cdot \nabla \Phi \partial_t \Phi
					+\int_{\Gamma} \Sigma_\eta^T \partial_t N_h \cdot \nabla \Phi \partial_t \Phi = I + II .
				\end{aligned}
			\end{equation}
  Recall that $N_h=(-(h_s+\eta)', 1)^T$, so that $\partial_t N_h= (-\partial_t\eta', 0)^T$. We combine the continuous embedding $H^{1/2}(\Gamma)\subset L^3 (\Gamma)$, the product estimate \eqref{prod-estH12} on $\Gamma$, the trace inequality in $H^1(\Omega_s)$ and the regularizing property \eqref{regular-prop} to obtain
			\begin{equation*}
				\begin{aligned}
					|I|& \leq \|\partial_t \Sigma_\eta^T N_h\|_{L^3(\Gamma)}
                \| \nabla \Phi\|_{L^3(\Gamma)}\|\partial_t \Phi\|_{L^3(\Gamma)}\\[5pt]&\leq C\|\partial_t \Sigma_\eta^T N_h\|_{H^{1/2}(\Gamma)}
                \| \nabla \Phi\|_{H^{1/2}(\Gamma)}\|\partial_t \Phi\|_{H^{1/2}(\Gamma)}
                \\[5pt]&\leq C \|N_h\|_{H^{1/2+}(\Gamma)}\|\partial_t \Sigma_\eta\|_{H^{1/2}(\Gamma)}\|\nabla \Phi\|_{H^1(\Omega_s)}\|\partial_t \Phi\|_{H^{1}(\Omega_s)}
                \\[5pt]&\leq C (1 + \|\eta\|_{H^{3/2+}(\mathcal{I})}) \| \partial_t \eta^\dag \|_{H^2(\Omega_s)}\| \Phi\|_{\mathring{H}^2(\Omega_s)}\|\partial_t \Phi\|_{H^{1}(\Omega_s)}\\[5pt]&\leq C \|\partial_t \eta\|_{{H}^{3/2}(\mathcal{I})}\|  \Phi\|_{\mathring{H}^2(\Omega_s)}\|\partial_t \Phi\|_{{H}^1(\Omega_s)}\leq C \sqrt{\mathcal{E}}\mathcal{D}.
				\end{aligned}
			\end{equation*}Note that we have used the smoothness of $h_s$ and \eqref{bound-eta-alpha}. Analogously, we have that
			\begin{equation*}
				\begin{aligned}
					|II|&\leq C\|\Sigma^T_\eta\partial_t N_h\|_{H^{1/2}(\Gamma)}\| \nabla \Phi\|_{H^{1/2}(\Gamma)}\|\partial_t \Phi\|_{H^{1/2}(\Gamma)}\\[5pt]
					&\leq C\|\Sigma_\eta\|_{H^{1/2+}(\Gamma)} \|\partial_t \eta'\|_{H^{1/2}(\mathcal{I})}\| \nabla \Phi\|_{H^{1}(\Omega_s)}\|\partial_t \Phi\|_{H^{1}(\Omega_s)}\\[5pt]
					&\leq  C\|\eta^\dag\|_{H^{2+}(\Omega_s)}\|\partial_t \eta\|_{H^{3/2}(\mathcal{I})} \|  \Phi\|_{\mathring{H}^2(\Omega_s)}\|\partial_t \Phi\|_{H^1(\Omega_s)}\\[5pt]
					&\leq  C\|\eta\|_{H^{3/2+}(\mathcal{I})}\|\partial_t \eta\|_{H^{3/2}(\mathcal{I})} \|  \Phi\|_{\mathring{H}^2(\Omega_s)}\|\partial_t \Phi\|_{H^1(\Omega_s)}\leq C \sqrt{\mathcal{E}}\mathcal{D} .
				\end{aligned}
			\end{equation*} For the third term in \eqref{S1}, we combine the product estimate \eqref{prod-estH12} on $\Gamma_w$ with the previous arguments to derive
			\begin{equation}\label{intGammaw1}
				\begin{aligned}
					\int_{\Gamma_w}& \mathrm{det}(J_\eta) F^1_3 \partial_t \Phi =\int_{\Gamma_w}\mathrm{det}(J_\eta) \partial_t \Sigma_\eta\nabla\Phi\cdot n \partial_t \Phi
\\[5pt]&\leq
C\|\mathrm{det} (J_\eta)\partial_t \Sigma_\eta\|_{H^{1/2}(\Gamma_w)}\|\nabla \Phi\|_{H^{1/2}(\Gamma_w)}\|\partial_t \Phi\|_{H^{1/2}(\Gamma_w)}
\\[5pt]&\leq C \|\mathrm{det} (J_\eta)\|_{H^{1+}(\Omega_s)} \|\partial_t \Sigma_\eta\|_{H^{1}(\Omega_s)}\| \nabla \Phi\|_{H^1(\Omega_s)}\|\partial_t \Phi\|_{H^1(\Omega_s)}
\\[5pt]&\leq C (1+\|\eta^\dag\|_{H^{2+}(\Omega_s)} )\| \partial_t \eta^\dag\|_{H^2(\Omega_s)}   \| \Phi\|_{\mathring{H}^2(\Omega_s)}\|\partial_t \Phi\|_{H^1(\Omega_s)}
			\\[5pt]&	\leq  C(1+\|\eta\|_{H^{3/2+}(\mathcal{I})} )\|\partial_t \eta\|_{H^{3/2}(\mathcal{I})}   \| \Phi\|_{\mathring{H}^2(\Omega_s)}\|\partial_t \Phi\|_{H^1(\Omega_s)}	\leq C \sqrt{\mathcal{E}}\mathcal{D}.
				\end{aligned}
			\end{equation}Thanks to the uniform bounds in Lemma \ref{lemma-R} and the continuous embedding $H^{1/2}(\mathcal{I})\subset L^4(\mathcal{I})$, the last term in \eqref{S1} has the control
			\begin{equation*}
				\begin{aligned}
					\int_{\mathcal{I}} \sigma \mathcal{F}_1(h'_s,\eta')dx &=  \int_{\mathcal{I}} \sigma \frac{(\partial_t \eta')^3}{2}\partial^2_{z_2}\mathcal{R}(h'_s,\eta')dx\\[5pt]&\leq C\|\partial_t \eta'\|_{L^2(\mathcal{I})} \|\partial_t \eta'\|^2_{L^4(\mathcal{I})}\left\|\partial^2_{z_2}\mathcal{R}\right\|_{L^\infty(\mathbb{R}^2)}
					\\[5pt]&\leq C \|\partial_t \eta\|_{H^1(\mathcal{I})} \|\partial_t \eta \|^2_{{H}^{3/2}(\mathcal{I})} \leq C \sqrt{\mathcal{E}} \mathcal{D}.
				\end{aligned}
			\end{equation*}	

            \end{proof}

        \begin{proposition}\label{S2-bound} Let $\mathcal{S}_2$ be given by \eqref{Sj}. Then, there exists a constant $C>0$ such that
		\begin{equation*}
			|\mathcal{S}_2| \leq C (\sqrt{\mathcal{E}} + \mathcal{E})\mathcal{D}.
		\end{equation*}
	\end{proposition}
		\begin{proof}Recall  from \eqref{Sj} that
			\begin{equation}\label{S2}
				\begin{aligned}
					\ \mathcal{S}_2= &- \int_{\Omega_s}F_2^2 \partial_t^2 \Phi-\int_{\Gamma} F_1^2 \partial_t^2 \Phi + \int_{\Gamma_w} \det(J_\eta)F_3^2\partial_t^2\Phi
					+ \int_{\mathcal{I}} \sigma\mathcal{F}_2(h'_s,\eta')dx.
				\end{aligned}
			\end{equation}with $F^2_j$ as in \eqref{F2} for $j=1,2,3$ and $\mathcal{F}_2$ as in  \eqref{Fj}.  To estimate the first term in \eqref{S2}, we combine the continuous embedding $H^{1}(\Omega_s)\subset L^4(\Omega_s)$, the regularizing property \eqref{regular-prop} and the trace inequality to get
            \begin{equation*}
				\begin{aligned}	\int_{\Omega_s}&F_2^2\partial_t^2 \Phi = -\int_{\Omega_s}\left(2\nabla \cdot (\partial_t A_\eta \nabla \partial_t \Phi) + \nabla \cdot (\partial_t^2 A_\eta \nabla \Phi) \right)\partial_t^2 \Phi \\[5pt]
&\leq C\left(\|  \partial_t A_\eta\|_{H^1(\Omega_s)}\|\nabla \partial_t \Phi\|_{H^1(\Omega_s)}
					 + \|\nabla \partial_t^2 A_\eta\|_{H^1(\Omega_s)}\|\nabla \Phi\|_{H^1(\Omega_s)}\right)\|\partial_t^2\Phi\|_{L^4(\Omega_s)}\\[5pt]
					&\leq C \big(\| \partial_t \eta^\dag\|_{H^2(\Omega_s)}\|\partial_t \Phi\|_{\mathring{H}^2(\Omega_s)} + \| \partial_t^2 \eta^\dag\|_{H^2(\Omega_s)}\| \Phi\|_{\mathring{H}^2(\Omega_s)} \big)\|\partial_t^2 \Phi\|_{H^1(\Omega_s)}
					\\[5pt]&\leq C \big(\|\partial_t \eta \|_{H^{3/2}(\mathcal{I})} \|\partial_t \Phi\|_{\mathring{H}^2(\Omega_s)} + \| \partial_t^2 \eta\|_{H^{3/2}(\mathcal{I})} \| \Phi\|_{\mathring{H}^2(\Omega_s)}    \big)\|\partial_t^2 \Phi\|_{H^1(\Omega_s)} \\[5pt]& \leq C\big( \sqrt{\mathcal{E} \mathcal{D}} + \sqrt{\mathcal{D}\mathcal{E}}\big) \|\partial_t^2\Phi\|_{H^1(\Omega_s)}\leq C \sqrt{\mathcal{E}}\mathcal{D}.
				\end{aligned}
			\end{equation*}

	Let us write the second term in \eqref{S2} as the sum of two:
			\begin{equation*}
				\begin{aligned}
					\int_{\Gamma} F_1^2 \partial_t^2 \Phi  &= -\int_{\Gamma} \left(2\partial_t (\Sigma^T_\eta N_h)\cdot \nabla \partial_t\Phi + \partial_t^2 (\Sigma_\eta^T N_h)\cdot \nabla \Phi\right)\partial_t^2 \Phi  = I +II.
				\end{aligned}
			\end{equation*}
  Since $I$ has the same structure as \eqref{est-intF11}, we infer that
   \begin{equation*}
       |I| \leq C \|\partial_t \eta\|_{H^{3/2}(\mathcal{I})}\|\partial_t\Phi\|_{\mathring{H}^2(\Omega_s)}\|\partial_t^2 \Phi\|_{H^1(\Omega_s)}\leq C \sqrt{\mathcal{E}}\mathcal{D}.
   \end{equation*}
Note that $\partial_t \Sigma_\eta \partial_t N_h=(0,0)^T$ and $\partial_t^2 N_h= (-\partial_t^2 \eta', 0 )^T$. Hence we use the continuous embedding $H^{1/2}(\Gamma)\subset L^3(\Gamma)$ and argue as before to get the estimate
   \begin{equation*}
       \begin{aligned}
       |II|&= \Big|\int_\Gamma \left( \partial_t^2\Sigma_\eta^T\cdot \nabla \Phi  - \partial_t^2 \eta'\partial_x \Phi \right) \partial_t^2 \Phi  \Big|\\[5pt]& \leq  C\left(\|\partial_t^2 \Sigma_\eta\|_{L^3(\Gamma)}\|\nabla \Phi\|_{L^3(\Gamma)} + \|\partial_t^2 \eta'\|_{L^3(\mathcal{I})}\|\partial_x \Phi\|_{L^3(\Gamma)}\right)\| \partial_t^2 \Phi \|_{L^3(\Gamma)}
       \\[5pt]&
       \leq C \big( \|\partial_t^2 \eta^\dag\|_{H^2(\Omega_s)} + \|\partial_t^2 \eta \|_{H^{3/2}(\mathcal{I})} \big)
       \| \nabla \Phi\|_{H^1(\Omega_s)} \|\partial_t^2 \Phi\|_{H^1(\Omega_s)}\\[5pt]&
       \leq C \|\partial_t^2 \eta\|_{H^{3/2}(\mathcal{I})} \|\Phi\|_{\mathring{H}^2(\Omega_s)}\|\partial_t^2 \Phi\|_{H^1(\Omega_s)}\leq C \sqrt{\mathcal{E}} \mathcal{D}.
       \end{aligned}
   \end{equation*}
The third term in \eqref{S2}, given by the boundary integral on $\Gamma_w$, reads
			\begin{equation*}
					\int_{\Gamma_w} \mathrm{det}(J_\eta) F^2_3 \partial_t^2 \Phi =-\int_{\Gamma_w}\mathrm{det}(J_\eta) \left(\partial_t \Sigma_\eta\nabla\partial_t\Phi + \partial_t^2 \Sigma_\eta \nabla\Phi\right)\cdot n \partial_t^2 \Phi = I + II.
\end{equation*} Note that the first term of the sum has the same structure as \eqref{intGammaw1}, hence we have
\begin{equation*}
    |I| \leq C \|\partial_t \eta\|_{H^{3/2}(\mathcal{I})}\|\partial_t \Phi\|_{\mathring{H}^2(\Omega_s)} \|\partial_t^2 \Phi\|_{H^1(\Omega_s)}\leq C \sqrt{\mathcal{E}}\mathcal{D}.
\end{equation*} The estimate of the second term follows using the product estimate \eqref{prod-estH12} on $\Gamma_w$ and arguing as before:
       \begin{equation*}\begin{aligned}
       |II| &\leq C\|\mathrm{det} (J_\eta)\partial_t^2 \Sigma_\eta\|_{H^{1/2}(\Gamma_w)}\|\nabla \Phi\|_{H^{1/2}(\Gamma_w)}\|\partial_t^2 \Phi\|_{H^{1/2}(\Gamma_w)}
\\[5pt]&\leq C \|\mathrm{det} (J_\eta)\|_{H^{1+}(\Omega_s)} \|\partial_t^2 \Sigma_\eta\|_{H^{1}(\Omega_s)}\| \nabla \Phi\|_{H^1(\Omega_s)}\|\partial_t ^2\Phi\|_{H^1(\Omega_s)}
\\[5pt]&\leq C (1+\|\eta^\dag\|_{H^{2+}(\Omega_s)} )\| \partial_t^2 \eta^\dag\|_{H^2(\Omega_s)}   \| \Phi\|_{\mathring{H}^2(\Omega_s)}\|\partial_t^2 \Phi\|_{H^1(\Omega_s)}
			\\[5pt]&	\leq  C(1+\|\eta\|_{H^{3/2+}(\mathcal{I})} )\|\partial_t ^2\eta\|_{H^{3/2}(\mathcal{I})}   \| \Phi\|_{\mathring{H}^2(\Omega_s)}\|\partial_t^2 \Phi\|_{H^1(\Omega_s)}	\leq C \sqrt{\mathcal{E}}\mathcal{D}.
			\end{aligned}
			\end{equation*}
Gathering the last two inequalities together gives the desired bound for the third term in \eqref{S2}.
            Thanks to the uniform bounds in Lemma \ref{lemma-R} and the continuous embedding $H^{1/2}(\mathcal{I})\subset L^q(\mathcal{I})$ with $1\leq q<\infty$, we estimate the last term in \eqref{S2} by
			\begin{equation*}
				\begin{aligned}
					&\int_{\mathcal{I}} \sigma \mathcal{F}_2(h'_s,\eta')dx =  \int_{\mathcal{I}} \sigma \big((\partial_t^2 \eta')^2 \partial_t\eta'\partial^2_{z_2}\mathcal{R}(h'_s,\eta') + \partial_t^2 \eta' (\partial_t \eta')^3 \partial^3_{z_2}\mathcal{R}(h'_s,\eta') \big) dx\\[5pt]&\leq \sigma\big(\|\partial_t \eta'\|_{L^2(\mathcal{I})} \|\partial_t^2 \eta'\|^2_{L^4(\mathcal{I})}\left\|\partial^2_{z_2}\mathcal{R}\right\|_{L^\infty(\mathbb{R}^2)} + \|\partial_t^2 \eta'\|_{L^2(\mathcal{I})} \|\partial_t \eta'\|^3_{L^6(\mathcal{I})}\left\|\partial^3_{z_2}\mathcal{R}\right\|_{L^\infty(\mathbb{R}^2)}\big)
					\\[5pt]&\leq C \big(\|\partial_t \eta\|_{H^1(\mathcal{I})} \|\partial_t^2 \eta \|^2_{{H}^{3/2}(\mathcal{I})} + \|\partial_t^2 \eta\|_{H^1(\mathcal{I})}\|\partial_t \eta \|^3_{{H}^{3/2}(\mathcal{I})}\big) \leq C \big( \sqrt{\mathcal{E}} + \mathcal{E}\big) \mathcal{D}.
				\end{aligned}
    \end{equation*}
		\end{proof}

\section{Main results}	\label{sec-mainres}	
In this section we close the scheme of energy estimates and show the main results of the paper. First of all, let us define the following energies and dissipation
\begin{equation}\label{def-realenedis}
\begin{aligned}
\mathfrak{E}(t)=	&\sum_{j=0}^2\int_{\mathcal{I}} \Big(g\frac{(\partial_t^j \eta)^2}{2} + \frac{\sigma}{2}\frac{ (\partial_t^j\eta')^2}{(1+ (h'_s)^2)^{3/2}} \Big)dx, \quad
\mathfrak{F}(t)= \sum_{j=0}^2 \int_{\mathcal{I}}   \sigma \mathcal{Q}_j(h'_s,\eta') dx\\[5pt]
\mathfrak{D}(t)=&\sum_{j=0}^2 \int_{\Omega_s}\det(J_\eta)|\Sigma_\eta\nabla \partial_t^j \Phi|^2 + (\partial_t^{j+1} \eta)^2(t,-1) + (\partial_t^{j+1} \eta)^2(t,1),
\end{aligned}
\end{equation}
with
\begin{align*}
    \mathcal{Q}_0(h_s', \eta') = & \int_0^{\eta'}\mathcal{R}(h_s', z)dz, \qquad
    \mathcal{Q}_1(h'_s,\eta')=  \frac{(\partial_t\eta')^2}{2}\partial_{z_2}\mathcal{R}(h'_s,\eta'), \\[5pt]
				  \mathcal{Q}_2(h'_s,\eta')=& \frac{(\partial_t^2\eta')^2}{2}\partial_{z_2}\mathcal{R}(h'_s,\eta') + \partial_t^2\eta' (\partial_t\eta')^2\partial^2_{z_2}\mathcal{R}(h'_s,\eta'),
\end{align*}
with $\mathcal{R}$ given by \eqref{R}. Note that, thanks to Propositions \ref{prop-higherenergy}, \ref{S1-bound} and \ref{S2-bound},  we derive the estimate
\begin{align}\label{open}
\frac{d}{d t}(\mathfrak{E}(t)+\mathfrak{F}(t))+\mathfrak{D}(t)\leq C \big(\sqrt{\mathcal{E}(t)}+ \mathcal{E}(t)\big)\mathcal{D}(t).
\end{align}
From \eqref{open} we see that,  in order to close the energy estimate,  we need to compare the improved energy $\mathcal{E}(t)$ with $\mathfrak{E}(t)+\mathfrak{F}(t)$ and the improved dissipation $\mathfrak{D}(t)$ with $\mathcal{D}(t)$. This is done in the following lemma, whose proof follows steps similar to those in Section 8 of \cite{GuoTice2018}.
\begin{lemma}\label{lem-comparison}Let $\mathfrak{E}(t)$, $\mathfrak{F}(t)$ and $\mathfrak{D}(t)$ be given by \eqref{def-realenedis}.
There exists a constant $\alpha^*>0$ such that, if
\begin{align*}
\sup_{t\in[0,T]} \mathcal{E}(t)\leq \alpha^*,
\end{align*}
then 
\begin{equation}\label{comparison}
    \mathfrak{E}(t)\lesssim \mathcal{E}_{\parallel}(t)\lesssim \mathfrak{E}(t), \quad \mathfrak{D}(t)\lesssim \mathcal{D}_{\parallel}(t)\lesssim \mathfrak{D}(t)\quad\text{and}\quad
|\mathfrak{F}(t)|\leq \frac{1}{2}\mathfrak{E}(t).
\end{equation}
\end{lemma}
\begin{proof}
The first chain of inequalities in \eqref{comparison} is directly derived from the fact that the stationary surface $h_s$ is smooth. For the second one, we know from \eqref{est-detJeta} that
\begin{align*}
\|\det(J_\eta)-1\|_{L^\infty(\Omega_s)}\leq C\|\eta\|_{H^{3/2+}(\mathcal{I})}
\end{align*}
and, arguing similarly, also that
\begin{align*}
\|\Sigma_\eta-\mathbb{I}\|_{L^\infty(\Omega_s)}\leq C \|\eta\|_{H^{3/2+}(\mathcal{I})}.
\end{align*}Since $\|\eta(t)\|^2_{H^{3/2+}(\mathcal{I})} \leq  \mathcal{E}(t)$, there exists $\alpha^*>0$ such that, if $\sup_{t\in[0,T]} \mathcal{E}(t)\leq \alpha^*$, then
\begin{equation}\label{est-diff-detSigma}
\|\det(J_\eta)-1\|_{L^\infty(\Omega_s)}\leq \frac{1}{4}, \qquad \|\Sigma_\eta-\mathbb{I}\|_{L^\infty(\Omega_s)}\leq \frac{1}{4}.\end{equation}
Together, these two inequalities yield the second chain of inequalities in \eqref{comparison}.

Finally, we prove the third inequality in \eqref{comparison}. Combining the first bound in \eqref{bounds-R} with the continuous embedding $H^{3/2+}(\mathcal{I})\subset W^{1,\infty}(\mathcal{I})$, we estimate the term with $\mathcal{Q}_0$ by
\begin{equation*}
\begin{aligned}
\int_{\mathcal{I}}\sigma \mathcal{Q}_0(h_s',\eta') dx &\leq C \int_{\mathcal{I}}|\eta'|^3dx\leq C\|\eta'\|_{L^\infty(\mathcal{I})}\|\eta'\|^2_{L^2(\mathcal{I})} \\[5pt]&\leq C\|\eta\|_{H^{3/2+}(\mathcal{I})}\|\eta'\|^2_{L^2(\mathcal{I})}  \leq C \sqrt{\mathcal{E}}\mathcal{E}_{\parallel}.
\end{aligned}
\end{equation*}
Analogously, other bounds in \eqref{bounds-R} yield
\begin{align*}
    &\int_{\mathcal{I}} \mathcal{Q}_1(h_s,\eta')dx\leq C \int_{\mathcal{I}}|\eta'||\pa_t\eta'|^2 dx\leq C\sqrt{\mathcal{E}}\mathcal{E}_{\parallel},\\[5pt]
    &\int_{\mathcal{I}} \mathcal{Q}_2(h_s,\eta')dx\leq C \int_{\mathcal{I}}\left(|\eta'| |\pa^2_t\eta'|^2+|\pa^2_t\eta'||\pa_t\eta'|^2\right) dx\leq C\sqrt{\mathcal{E}}\mathcal{E}_{\parallel}.
\end{align*}
Note that in the last inequality we have used the continuous embedding $H^{1/4}(\mathcal{I})\subset L^4(\mathcal{I})$ together with an interpolation inequality to obtain the bound
\begin{align*}
   \|\pa_t\eta'\|_{L^4(\mathcal{I})}^2 &\leq C \|\pa_t\eta'\|^2_{H^{1/4}(\mathcal{I})}\leq C\|\pa_t\eta'\|_{L^2(\mathcal{I})}\|\pa_t\eta'\|_{H^{1/2}(\mathcal{I})}\\[5pt]&\leq C \|\pa_t\eta\|_{H^1(\mathcal{I})}\|\pa_t \eta\|_{H^{3/2}(\mathcal{I})}.
\end{align*}

Gathering all estimates and considering a possibly smaller $\alpha^*$ yields the third inequality in \eqref{comparison}.
\end{proof}
We now state the main results of the paper, which consist of a global-in-time higher-order bound and a decay estimate.
\begin{theorem}\label{mainTh}Assume that $\eta$ have zero mean over $\mathcal{I}$. There exists a constant $\widetilde{\alpha}>0$ such that, if
\begin{equation}
\begin{aligned}\label{small-hyp}
\sup_{t\in [0,T]}\mathcal{E}(t)+\int_{0}^T \mathcal{D}(t)dt \leq \widetilde{\alpha},
\end{aligned}
\end{equation}
 then 
\begin{align*}
\sup_{t\in [0,T]}\mathcal{E}(t)+\int_{0}^T \mathcal{D}(t)dt\leq C \mathcal{E}(0).
\end{align*}In addition, there exists a constant $\lambda>0$ such that
\begin{equation*}
 \sup_{t\in[0,T]} \Big(\mathcal{E}_\parallel (t) + \int_{\Omega_s}|\nabla \Phi(t)|^2 + (\partial_t \eta)^2(t,-1) + (\partial_t \eta)^2(t,1)\Big) \leq C \mathcal{E}_\parallel(0)e^{-\lambda T}.\\\vspace{0.5em}
\end{equation*}
\end{theorem}
\begin{proof}Define $\widetilde{\alpha}=\min(\alpha, \alpha^*)$ with $\alpha$ and $\alpha^*$ as in Proposition \ref{prop-reg-trans} and Lemma \ref{lem-comparison}.
After integrating \eqref{open} over $(0,t)$, we obtain that
\begin{align*}
\mathfrak{E}(t)+\mathfrak{F}(t)+\int_{0}^t \mathfrak{D}(s)ds\leq C \Big( \mathfrak{E}(0)+\int_0^t \big(\sqrt{\mathcal{E}(s)}+ \mathcal{E}(s)\big)\mathcal{D}(s)ds\Big),
\end{align*}
and, due to \eqref{comparison}, we get
\begin{equation}\label{est_Eparallel}
\begin{aligned}
    \mathcal{E}_{\parallel}(t)+\int_0^t \mathcal{D}_{\parallel}(s)ds\leq C\Big( \mathcal{E}_\parallel(0)+\int_0^t \big(\sqrt{\mathcal{E}(s)} + \mathcal{E}(s)\big)\mathcal{D}(s)ds\Big).
\end{aligned}
\end{equation}
Recalling from \eqref{tot-diss} that the additional dissipation is given by
\begin{align*}
    \mathcal{D}_\perp(t)\!= \! \sum_{j=0}^{1} \| \partial_t^j\eta(t)\|^2_{H^{5/2} (\mathcal{I})}\!+\! \| \partial_t^2\eta(t)\|^2_{H^{3/2+}(\mathcal{I})}\!+ \!\sum_{j=0}^{2}\| \partial_t^j \Phi(t)\|^2_{L^2(\Omega_s)} \!+\!\|\partial_t\Phi(t)\|^2_{\mathring{H}^2(\Omega_s)},
\end{align*}
 a consequence of Propositions \ref{prop-controlL2}, \ref{prop-controletaj}, \ref{H2PhiD} and \eqref{small-hyp} is that \begin{equation}\label{bound-D}
    \mathcal{D}_\perp(t)\leq C \left(\mathcal{D}_{\parallel}(t)+ \mathcal{E}(t)\mathcal{D}(t)\right).
\end{equation}After combining \eqref{bound-D} with \eqref{est_Eparallel} and \eqref{small-hyp}, up to taking a smaller $\widetilde{\alpha}$, we infer that there exists $C=C(\widetilde{\alpha})>0$ such that
\begin{equation}\label{bound-Eparallel}
\begin{aligned}
    \mathcal{E}_{\parallel}(t)+\int_0^t \mathcal{D}(s)ds\leq C\mathcal{E}_\parallel(0).
\end{aligned}
\end{equation}
We know from \eqref{impro-ene} that the additional energy is given by
\begin{align*}
    \mathcal{E}_\perp(t)= \|\eta(t)\|^2_{H^{3/2+}(\mathcal{I})} + \|\partial_t \eta(t)\|^2_{H^{3/2+}(\mathcal{I})}.
\end{align*}
As mentioned in Section \ref{sec-enediss}, exploiting the fact that $H^{3/2+}(\mathcal{I})$ are Hilbert spaces and using Cauchy-Schwarz's inequality, we find that
\begin{align*}
\frac{d}{dt}\mathcal{E}_\perp(t)=& \ 2\left(\eta(t),\pa_t\eta(t)\right)_{H^{3/2+}(\mathcal{I})}+2\left(\pa_t\eta(t), \pa_t^2\eta(t)\right)_{H^{3/2+}(\mathcal{I})}
\\[5pt]\leq  & \ \|\eta(t)\|^2_{H^{3/2+}(\mathcal{I})}+2\|\pa_t\eta(t)\|^2_{H^{3/2+}(\mathcal{I})}+\|\pa_t^2\eta(t)\|^2_{H^{3/2+}(\mathcal{I})}
\leq 2\mathcal{D}
\end{align*}
and integrating the previous inequality over $(0,t)$ yields
\begin{equation}\label{bound-Eperp}
\begin{aligned}
\mathcal{E}_{\perp}(t)\leq \mathcal{E}_{\perp}(0) + 2\int_0^t\mathcal{D}(s)ds.
\end{aligned}
\end{equation}
Gathering \eqref{bound-Eparallel}-\eqref{bound-Eperp} together gives the desired higher-order a priori estimate, namely, 
\begin{align*}
   \mathcal{E}(t)+\int_0^t\mathcal{D}(s)ds\leq C \mathcal{E}(0).
\end{align*}
We now derive the decay estimate. Combining \eqref{open} with \eqref{comparison}, \eqref{bound-D} and the smallness condition \eqref{small-hyp} yields
\begin{equation*}
    \frac{d}{dt}\left(\mathfrak{E}(t) + \mathfrak{F}(t)\right) + C\mathcal{D}(t) \leq 0 .
\end{equation*}By definition, we directly have that $\mathfrak{E}\leq C \mathcal{D}$ and since the third inequality in \eqref{comparison} implies\begin{equation}\label{compa-E+F}\frac{1}{2}\mathfrak{E}(t)\leq \mathfrak{E}(t) + \mathfrak{F}(t)\leq \frac{3}{2} \mathfrak{E}(t),\end{equation} 
we infer that there exists a  constant $\lambda>0$ such that
\begin{equation*}
     \frac{d}{dt}\left(\mathfrak{E}(t) + \mathfrak{F}(t)\right) + \lambda\left(\mathfrak{E}(t) + \mathfrak{F}(t)\right) \leq 0 .
\end{equation*}After integrating this differential inequality in time and using again \eqref{comparison} and the first chain of inequalities in \eqref{comparison}, we end up with the decay estimate
\begin{equation}\label{est-decayEpar}
    \mathcal{E}_\parallel (t) \leq C \mathcal{E}_\parallel(0) e^{-\lambda t}.
\end{equation}
In addition, we know from the energy-dissipation equality \eqref{ene-eq-per} that
\begin{equation*}\begin{aligned}
				 &\int_{\Omega_s} \det(J_\eta)|\Sigma_\eta\nabla \Phi|^2 + (\partial_t \eta)^2 (t,-1) + (\partial_t \eta)^2 (t,1)\\[5pt]& = -\int_\mathcal{I} \Big(g \eta \partial_t \eta + \sigma\frac{\eta' \partial_t\eta'}{(1+ (h'_s)^2)^{3/2}}  + \sigma \mathcal{R}(h'_s, \eta')\partial_t \eta'\Big) dx.
                \end{aligned}
		\end{equation*}Therefore, using the uniform bounds in Lemma \ref{lemma-R} and the continuous embedding $H^{3/2+}(\mathcal{I})\subset W^{1, \infty}(\mathcal{I})$, we get the estimate
\begin{equation}\label{est-D-Epar}\begin{aligned}
				 &\int_{\Omega_s} \det(J_\eta)|\Sigma_\eta\nabla \Phi|^2 + (\partial_t \eta)^2 (t,-1) + (\partial_t \eta)^2 (t,1)\\[5pt]& \leq  C \left( 1+\|\eta'\|_{L^\infty(\mathcal{I})}\right)\|\eta\|_{H^1(\mathcal{I})}\|\partial_t \eta\|_{H^1(\mathcal{I})}\leq C \big(1+\sqrt{\mathcal{E}}\big)\mathcal{E}_\parallel\leq C\mathcal{E}_\parallel,
                \end{aligned}
\end{equation}where the last inequality follows from \eqref{small-hyp}. Gathering \eqref{est-decayEpar}-\eqref{est-D-Epar} together and employing \eqref{est-diff-detSigma}, we conclude that
\begin{equation*}
   \mathcal{E}_\parallel (t) + \|\nabla \Phi(t)\|^2_{L^2(\Omega_s)} + (\partial_t \eta)^2(t,-1) + (\partial_t \eta)^2(t,1) \leq C \mathcal{E}_\parallel(0)e^{-\lambda t}.
\end{equation*}
\end{proof}

\subsection*{Acknowledgments}
The authors would like to thank the anonymous referees for the helpful comments that improved the paper.
EB was partially supported by the
Horizon Europe EU Research and Innovation Programme through the Marie Sklodowska-Curie THANAFSI Project No. 101109475 and by the Gruppo Nazionale
per l’Analisi Matematica, la Probabilità e le loro Applicazioni (GNAMPA) of the
Istituto Nazionale di Alta Matematica (INdAM). EB acknowledges support from DMAT, funded by MUR through the 2023-2027 Excellence Department Grant.
EB and FG were partially supported by the Junta de Andaluc\'ia through the grant P20-00566.
AC acknowledges financial support from  2023
Leonardo Grant for Researchers and Cultural Creators, BBVA Foundation. The BBVA Foundation
accepts no responsibility for the opinions, statements, and contents included in the project and/or the results thereof, which are entirely the responsibility of the authors. AC was partially supported by the Severo Ochoa Programme for Centers of Excellence Grant CEX2019-000904-S and CEX-2023-001347-S funded by MCIN/AEI/10.13039/501100011033. AC and FG were partially supported by the MICINN through the grant PID2020-114703GB-I00. FG was partially supported by the MICINN through the grants EUR2020-112271 and PID2022-140494N, by the Fundaci\'on de Investigaci\'on de la Universidad de Sevilla through the
grant FIUS23/0207, and by MINECO grant  RED2022-134784-T funded by MCIN/AEI/10.13039/50110001103 (Spain). FG acknowledges support
from IMAG, funded by MICINN through the Maria de Maeztu Excellence Grant CEX2020-001105-M/AEI/10.13039/501100011033. FG was partially supported by the Institute of Advanced Study and the Harish-Chandra Fund.

\subsection*{Conflicts of Interest} We confirm that we do not have any conflict of interest.

\subsection*{Data Availability} The manuscript has no associated data.

		 \appendix

		 \section{Technical estimates}\label{appendix}
		
		 \begin{lemma}\label{lemma-prodest}
Let $D\subset\mathbb{R}^2$ be a bounded Lipschitz domain and $F\in H^{1}(D)$, $G\in H^{1+}(D)$ be scalar functions. Then,  $FG \in H^1(D)$ and there exists $C=C(D)>0$ such that
\begin{equation}\label{prod-estH1}
\|FG\|_{H^{1}(D)}\leq C\|F\|_{H^{1}(D)}\|G\|_{H^{1+}(D)}.
\end{equation}
Let $\gamma\subset \partial D$ and $f \in H^{1/2}(\gamma)$, $g \in H^{1/2+}(\gamma)$ be scalar functions. Then, $fg\in H^{1/2}(\gamma)$ and there exists $C=C(D)>0$ such that
\begin{equation}\label{prod-estH12}
	\|fg\|_{H^{1/2}(\gamma)}\leq C\|f\|_{H^{1/2}(\gamma)}\|g\|_{H^{1/2+}(\gamma)}.
\end{equation}
		 \end{lemma}
\begin{proof}
First of all, note that $\nabla (FG)= \nabla F G + F \nabla G$. Hölder inequality and the embedding $H^{1+}(D)\subset L^\infty(D)$ imply
\begin{equation*}
\|FG\|_{L^2(D)} + \|\nabla F G\|_{L^2(D)}\leq C \|F\|_{H^1(D)}\|G\|_{L^\infty(D)}\leq  C\|F\|_{H^1(D)}\|G\|_{H^{1+}(D)}
\end{equation*}
where $C=C(D)>0$.
Furthermore, we recall the continuous embeddings $H^{\eps}(D)\subset  L^{2/(1-\eps)}(D)$ for any $\eps\in (0,1)$ and $H^1(D)\subset L^q(D)$ for any $q\in [1, \infty)$. Therefore,  using again Hölder inequality with exponents $(2/\eps, 2/(1-\eps))$ yields
\begin{align*}
\|F \nabla G\|_{L^2(D)}&\leq \| F\|_{L^{2/\eps}(D)} \| \nabla G\|_{L^{2/(1-\eps)}(D)}\\[5pt]&\leq C\|F\|_{H^1(D)} \| \nabla G\|_{H^{\eps}(D)}\leq C \|F\|_{H^1(D)} \| G\|_{H^{1+}(D)}.
\end{align*}Thus, \eqref{prod-estH1} is proved. Let us now introduce the extensions to $D$ of $f$ and $g$, respectively denoted by $\widetilde{f}$ and $\widetilde{g}$. Then, $\widetilde{f}\in H^1(D)$, $\widetilde{g} \in H^{1+}(D)$ and
\begin{equation*}
	\|\widetilde{f}\|_{H^1(D)} \leq C \|f\|_{H^{1/2}(\gamma)}, \qquad	\|\widetilde{g}\|_{H^{1+}(D)} \leq C \|g\|_{H^{1/2+}(\gamma)}.
\end{equation*} The first part of the lemma implies that $\widetilde{f}\widetilde{g}\in H^{1}(D)$ and the product estimate \eqref{prod-estH1} holds. Since $(\widetilde{f}\widetilde{g})_{|_\gamma}= fg$, we conclude that $fg\in H^{1/2}(\gamma)$ and
\begin{equation*}
\|fg\|_{H^{1/2}(\gamma)}\leq C\| \widetilde{f}\widetilde{g}\|_{H^{1}(D)}\leq C \| \widetilde{f}\|_{H^{1}(D)}\| \widetilde{g}\|_{H^{1+}(D)}\leq C \|f\|_{H^{1/2}(\gamma)} \|g\|_{H^{1/2+}(\gamma)},
\end{equation*} which proves \eqref{prod-estH12}.
\end{proof}

Next, we state the uniform bounds for $\mathcal{R}$ in \eqref{R} used throughout the paper, as well as the estimates on the $H^{1/2+}(\mathcal{I})$-norm of its derivatives used in Proposition \ref{prop-controletaj}.

		 \begin{lemma}{\cite[Proposition B.1]{GuoTice2018}}
		 	\label{lemma-R}
Let $\mathcal{R}$ be given by \eqref{R}. Then, there exists a constant $C>0$ such that
\begin{equation}\label{bounds-R}
\begin{aligned}
\Big| \frac{1}{z_2^3}\int_0^{z_2}\mathcal{R}(z_1,s)ds\Big| + \Big|\frac{\mathcal{R}(z_1,z_2)}{z_2^2} \Big| + \Big| \frac{\partial_{z_2}\mathcal{R}(z_1,z_2)}{z_2} \Big| + \Big| \frac{\partial_{z_1}\mathcal{R}(z_1,z_2)}{z_2^2}\Big|&\\[5pt]+ | \partial^2_{z_2}\mathcal{R}(z_1,z_2)|+ \Big|\frac{\partial^2_{z_1}\mathcal{R}(z_1,z_2) }{z_2^2}\Big|
+ \Big|\frac{\partial_{z_1}\partial_{z_2}\mathcal{R}(z_1,z_2) }{z_2}\Big|&
\\[5pt]+ | \partial^3_{z_2}\mathcal{R}(z_1,z_2)| + | \partial^2_{z_2}\partial_{z_1}\mathcal{R}(z_1,z_2)| +\Big| \frac{\partial^2_{z_1}\partial_{z_2}\mathcal{R}(z_1,z_2)}{z_2}\Big|&
\leq& C\\[5pt]
\end{aligned}
\end{equation} for any $(z_1,z_2)\in \mathbb{R}^2$.
In addition, if $\eta\in H^{3/2+}(\mathcal{I})$, we have that
\begin{align}
  & \| \partial_{z_2}\mathcal{R}(h_s', \eta')\|^2_{H^{1/2+}(\mathcal{I})}\leq C \|\eta\|^2_{H^{3/2+}(\mathcal{I})}, \label{est-derR}\\[5pt]
   & \|\partial_{z_2}^2 \mathcal{R}(h_s', \eta')\|^2_{H^{1/2+}(\mathcal{I})}\leq C \big(1 + \|\eta\|^2_{H^{3/2+}(\mathcal{I})}\big),\label{est-der2R}\\[5pt]
     & \| \partial_{z_1}\partial_{z_2}\mathcal{R}(h_s', \eta')\|^2_{H^{1/2+}(\mathcal{I})}\leq C \|\eta\|^2_{H^{3/2+}(\mathcal{I})}. \label{est-der1der2R}
\end{align}

	\end{lemma}	
\begin{proof}
We observe that
\begin{equation*}\mathcal{R}(z_1,z_2)= \frac{z_1 + z_2}{\sqrt{1+(z_1+ z_2)^2}} - \frac{z_1}{\sqrt{1+z_1^2}} -\frac{z_2}{(1+z_1^2)^{3/2}}\end{equation*}
can be also written in the integral form
\begin{align*}
    \mathcal{R}(z_1,z_2)=-z_2^2\int_0^1s\int_0^1(z_1+s\mu z_2)\bigg(\frac{1}{(1+s\mu z_2)^\frac{3}{2}}-\frac{2}{1+z_1^2}\bigg)d\mu ds = z_2^2 \ r(z_1,z_2),
\end{align*} where $|r(z_1,z_2)|\leq C$ for any $(z_1,z_2)\in \mathbb{R}^2$.
Using analogous integral versions for the first-order derivatives of $\mathcal{R}$, we obtain that
\begin{align*}
   \partial_{z_2}\mathcal{R}(z_1,z_2)=z_2 \ r_1(z_1,z_2), \qquad  \partial_{z_2}^2 \mathcal{R}(z_1,z_2)=r_2(z_1,z_2),
\end{align*} where $|r_i(z_1,z_2)|\leq C$ for $i=1,2$ and any $(z_1,z_2)\in \mathbb{R}^2$. The rest of the bounds in \eqref{bounds-R} follows by repeating the same argument.

In order to prove \eqref{est-derR}, we recall that
\begin{align*}
\| \partial_{z_2}\mathcal{R}(h_s', \eta')\|^2_{H^{1/2+}(\mathcal{I})} = \| \partial_{z_2}\mathcal{R}(h_s', \eta')\|^2_{L^2(\mathcal{I})} +\|\partial_{z_2}\mathcal{R}(h_s', \eta')\|^2_{\dot{H}^{1/2+}(\mathcal{I})}, \end{align*}
where
\begin{align*}\|\partial_{z_2}\mathcal{R}(h_s', \eta')\|^2_{\dot{H}^{1/2+}(\mathcal{I})}=\int_{\mathcal{I}}\int_{\mathcal{I}}\frac{|\partial_{z_2}\mathcal{R}(h_s'(x), \eta'(x)) - \partial_{z_2}\mathcal{R}(h_s'(x'), \eta'(x'))|^2}{|x-x'|^{2+ 2\delta}}dx'dx
\end{align*}for $0<\delta<1/2$.
After writing
\begin{align*}
    &\partial_{z_2}\mathcal{R}(h_s'(x), \eta'(x))\! -\! \partial_{z_2}\mathcal{R}(h_s'(x'), \eta'(x'))\!  \\[5pt]&\!=\!\partial_{z_2}\mathcal{R}(h_s'(x), \eta'(x)) \!- \!\partial_{z_2}\mathcal{R}(h_s'(x), \eta'(x')) \!+\! \partial_{z_2}\mathcal{R}(h_s'(x), \eta'(x'))\!-\!\partial_{z_2}\mathcal{R}(h_s'(x'), \eta'(x')),
\end{align*}
Mean-Value Theorem yields that 
\begin{align*}
   & |\partial_{z_2}\mathcal{R}(h_s'(x), \eta'(x))\! - \!\partial_{z_2}\mathcal{R}(h_s'(x), \eta'(x'))| \leq  \|\partial^2_{z_2} \mathcal{R}(h'_s(x), \cdot )\|_{L^\infty(\mathbb{R})}|\eta'(x)- \eta'(x')|\\[5pt]&
   |\partial_{z_2}\mathcal{R}(h_s'(x), \eta'(x'))\!-\!\partial_{z_2}\mathcal{R}(h_s'(x'), \eta'(x'))|\! \leq \! \|\partial_{z_1}\partial_{z_2} \mathcal{R}(\cdot, \eta'(x') )\|_{L^\infty(\mathbb{R})}|h'_s(x)\!- \!h'_s(x')|\end{align*}
Since we know from \eqref{bounds-R} that
\begin{equation*}
     |\partial^2_{z_2} \mathcal{R}(z_1,z_2)|  + \left|\frac{\partial_{z_1}\partial_{z_2}\mathcal{R}(z_1, z_2)}{z_2}\right|\leq C  \quad \forall (z_1,z_2)\in \mathbb{R}^2,
\end{equation*}therefore implying
$$\|\partial_{z_2}^2 \mathcal{R}\mathcal(h'_s(x), \cdot)\|_{L^\infty(\mathbb{R})}\leq C, \quad \|\partial_{z_1}\partial_{z_2} \mathcal{R}(\cdot, \eta'(y))\|_{L^\infty(\mathbb{R})}\leq C |\eta'(y)|,$$
it follows that 
\begin{align*}
 \|\partial_{z_2}\mathcal{R}(h_s', \eta')\|^2_{\dot{H}^{1/2+}(\mathcal{I})} \! &\leq \!C \!\int_{\mathcal{I}}\int_{\mathcal{I}}  \Big( \frac{|\eta'(x)\!-\! \eta'(x')|^2}{|x-x'|^{2+2\delta}} \!+ \!\frac{|\eta'(x')|^2 |h'_s(x)\!-\! h'_s(x')|^2}{|x-x'|^{2+2\delta}}\Big)dx'dx
 \\[5pt]
 &\leq C \big(\|\eta'\|^2_{\dot{H}^{1/2+}(\mathcal{I})} + \|\eta'\|^2_{L^\infty(\mathcal{I})}\|h'_s\|^2_{\dot{H}^{1/2+}(\mathcal{I})}\big).
\end{align*}
Thanks to the continuous embedding $H^{1/2+}(\mathcal{I}) \subset L^\infty(\mathcal{I})$ and the smoothness of $h_s$, we deduce that 
\begin{equation}\label{bound-seminorm}
    \|\partial_{z_2}\mathcal{R}(h_s', \eta')\|^2_{\dot{H}^{1/2+}(\mathcal{I})} \leq C\|\eta'\|^2_{H^{1/2+}(\mathcal{I})}. 
\end{equation}
To conclude, we still need to control the $L^2$-norm. Again, we know from \eqref{bounds-R} that 
$$ \Big|\frac{\partial_{z_2}\mathcal{R}(z_1, z_2)}{z_2}\Big|\leq C  \quad \forall (z_1,z_2)\in \mathbb{R}^2,$$
therefore implying that
\begin{equation}\label{bound-L2norm}
    \|\partial_{z_2}\mathcal{R}(h_s', \eta')\|^2_{L^2(\mathcal{I})} = \int_{\mathcal{I}} |\partial_{z_2}\mathcal{R}(h_s'(x), \eta'(x))|^2 dx \leq C \int_{\mathcal{I}} |\eta'(x)|^2 dx.
\end{equation}
Gathering \eqref{bound-seminorm}-\eqref{bound-L2norm} together, we then obtain that
\begin{equation*}
   \| \partial_{z_2}\mathcal{R}(h_s', \eta')\|^2_{H^{1/2+}(\mathcal{I})}\leq C \|\eta'\|^2_{H^{1/2+}(\mathcal{I})}\leq C \|\eta\|^2_{H^{3/2+}(\mathcal{I})}.
\end{equation*}
 The remaining estimates \eqref{est-der2R}-\eqref{est-der1der2R} are obtained by arguing in the same fashion and using the fact that
\begin{equation*}
     |\partial^3_{z_2} \mathcal{R}(z_1,z_2)|  + \left|{\partial_{z_2}^2\partial_{z_1}\mathcal{R}(z_1, z_2)}\right| + \Big| \frac{\partial^2_{z_1}\partial_{z_2}\mathcal{R}(z_1,z_2)}{z_2}\Big|\leq C  \quad \forall (z_1,z_2)\in \mathbb{R}^2.
\end{equation*}

\end{proof}

\bibliography{referencesUniv03}

\end{document}